\documentclass[11pt,reqno]{amsart}
\usepackage[
    top=1in, bottom=1in, inner=1in, outer=1in,
    marginparwidth=2cm %
    ]{geometry}

\usepackage{amssymb,latexsym, amsmath, amsxtra, mathrsfs, bm}
\usepackage[dvips]{graphics}
\usepackage[T1]{fontenc}

\usepackage[all]{xy}
\usepackage{xcolor}
\usepackage{subcaption}
\usepackage{verbatim}
\usepackage[abs]{overpic}
\usepackage{hyperref}
\usepackage{mathtools}
\usepackage{enumitem}

\usepackage{thmtools}
\usepackage{thm-restate}

\DeclareMathAlphabet{\mathbbold}{U}{bbold}{m}{n}

\allowdisplaybreaks[3]

\renewcommand{\=}{\coloneqq}
\makeatletter
\def\th@plain{%
	\thm@notefont{}%
	\itshape %
}
\def\th@definition{%
	\thm@notefont{}%
	\normalfont %
}
\makeatother
\theoremstyle{plain}
\newtheorem{theorem}{Theorem}[section]
\newtheorem*{theorem*}{Theorem}
\newtheorem{thmx}{Theorem}

\newtheorem{thml}{Theorem}
      
\newtheorem*{conj*}{Conjecture}
\newtheorem{lemma}[theorem]{Lemma}
\newtheorem{prop}[theorem]{Proposition}
\newtheorem{cor}[theorem]{Corollary}

\theoremstyle{definition}
\newtheorem{definition}[theorem]{Definition}
\newtheorem{rem}[theorem]{Remark}

\newtheorem{example}[theorem]{Example}
\newtheorem{notation}[theorem]{Notation}

\theoremstyle{remark}

\numberwithin{equation}{section}
\numberwithin{theorem}{section}
\numberwithin{table}{section}
\numberwithin{figure}{section}

\newcommand{\diam}  {\operatorname{diam}}

\newcommand{\card} {\operatorname{card}}
\newcommand{\supp}{\operatorname{supp}}

\newcommand{\R}{\mathbb{R}}

\newcommand{\Q}{\mathbb{Q}}
\newcommand{\N}{\mathbb{N}}      

\providecommand{\abs}[1]{\lvert#1\rvert}
\providecommand{\Absbig}[1]{\bigl\lvert#1\bigr\rvert}

\providecommand{\norm}[1]{\|#1\|}

\renewcommand{\:}{\colon}

\newcommand{\cA}{\mathcal{A}}
\newcommand{\cB}{\mathcal{B}}
\newcommand{\cC}{\mathcal{C}}

\newcommand{\cH}{\mathcal{H}}

\newcommand{\cK}{\mathcal{K}}
\newcommand{\cL}{\mathcal{L}}
\newcommand{\cM}{\mathcal{M}}

\newcommand{\cO}{\mathcal{O}}
\newcommand{\cP}{\mathcal{P}}

\newcommand{\sP}{\mathscr{P}}

\newcommand{\hb}{\widehat{b}}
\newcommand{\hc}{\widehat{c}}

\newcommand{\ha}{\widehat{a}}
\newcommand{\hA}{\widehat{A}}
\newcommand{\hB}{\widehat{B}}
\newcommand{\hC}{\widehat{C}}

\newcommand{\hbN}{\widehat{\mathbb{N}}}

\newcommand{\hW}{\widehat{W}}

\newcommand{\hfa}{\widehat{\mathbf{a}}}
\newcommand{\fa}{\mathbf{a}}
\newcommand{\fb}{\mathbf{b}}
\newcommand{\fc}{\mathbf{c}}

\newcommand{\hpi}{\widehat{\pi}}

\newcommand{\hrho}{\widehat{\rho}}

\newcommand{\hSigma}{\widehat{\Sigma}}

\newcommand{\tphi}{\widetilde{\phi}}

\newcommand{\tpsi}{\widetilde{\psi}}

\newcommand{\PPP}{\mathcal{P}}

\newcommand{\MMM}{\mathcal{M}}

\newcommand{\Holder}[1] {\CCC^{0,#1}}

\newcommand{\Hseminorm}[2] {\abs{#2}_{#1}}

\newcommand{\Hseminormbig}[2] {\Absbig{#2}_{#1}}

\newcommand{\Hnorm}[2] {\norm{#2}_{#1}}

\renewcommand{\le}{\leqslant}
\renewcommand{\leq}{\leqslant}
\renewcommand{\ge}{\geqslant}
\renewcommand{\geq}{\geqslant}

\newcommand{\vertiii}[1]{{\left\vert\kern-0.25ex\left\vert\kern-0.25ex\left\vert #1 
		\right\vert\kern-0.25ex\right\vert\kern-0.25ex\right\vert}}
\newcommand{\dd}{\mathrm{d}}

\newcommand{\CCC}{C}

\newcommand{\bcf}{E}

\newcommand{\Lock}{\operatorname{Lock}}

\newcommand{\Flock}{\operatorname{Lock_F}}

\newcommand{\tcO}{\widetilde{\mathcal{O}}}

\newcommand{\mea}{Q}

\newcommand{\Fix}{\operatorname{Fix}}

\newcommand{\Mmax}{\mathcal{M}_{\max}}

\newcommand{\RM}{\mathcal{F}}

\newcommand{\Gau}{G}

\newcommand{\bad}{\mathfrak{E}}

\newcommand{\well}{\mathfrak{Z}}

\newcommand{\ral}{\mathfrak{R}}

\newcommand{\irr}{\operatorname{irr}}

\newcommand{\erg}{\operatorname{erg}}

\newcommand{\conv}{\operatorname{conv}}

\begin{document}

\title[Ergodic optimization for Gauss's continued fraction map]
 {Ergodic optimization for Gauss's continued fraction map}

\author{Yinying~Huang \and Oliver~Jenkinson \and Zhiqiang~Li}
\address{Yinying~Huang: School of
Mathematical Sciences, Peking University, Beijing 100871, China.}
\email{miaoyan@stu.pku.edu.cn}
\address{Oliver~Jenkinson, School of Mathematical Sciences, Queen Mary, University of London, Mile End Road, London E1 4NS, United Kingdom}
\email{o.jenkinson@qmul.ac.uk}
\address{Zhiqiang~Li, School of Mathematical Sciences \& Beijing International Center for Mathematical Research, Peking University, Beijing 100871, China}
\email{zli@math.pku.edu.cn}

\subjclass[2020]{Primary: 37A99; Secondary: 37A05, 37B25, 37B65, 37C50, 37E05, 37A44.}

\keywords{Gauss map, continued fractions, maximizing measure, ergodic optimization, typical periodic optimization, typical finite optimization}

\thanks{Y.~H.\ and Z.~L.\ were partially supported by Beijing Natural Science Foundation (JQ25001 and 1214021) and National Natural Science Foundation of China (12471083, 12101017, 12090010, and 12090015).
	}

\begin{abstract}
      The theory of ergodic optimization for distance-expanding maps is extended to Gauss's continued fraction map. Since the set of invariant probability measures is not weak$^*$ closed, we establish a characterisation of the closure of this set, and investigate limit-maximizing measures for H\"older continuous functions.
     Although a Ma\~n\'e cohomology lemma is shown to hold, the typical periodic optimization conjecture is shown to fail, as a consequence of the typical finite optimization property established for a certain class of (rationally maximized) functions. The typical periodic optimization (TPO) property is shown to hold, however, for the class of $\alpha$-H\"older essentially compact functions.
\end{abstract}

\maketitle
\tableofcontents
\section{Introduction}\label{sct_intro}

    The purpose of this work is to study the properties of maximizing measures for Gauss's continued fraction map $\Gau \: [0,1]\to [0,1]$, defined by \begin{equation}\label{gaussdefnintro}
    	\Gau(x)=\begin{cases}
    		\frac{1}{x}-\bigl\lfloor\frac{1}{x}\bigr\rfloor   & \text{if }x\in (0,1]  , \\
    		0 &\text{if }x=0  . 
    	\end{cases}
    \end{equation}
    
	For a general Borel measurable map $T\: X \rightarrow X$ on a metric space $X$, let $\MMM(X,T)$ denote the set of $T$-invariant Borel probability measures on $X$, let $\overline{\MMM(X,T)}$ denote the weak$^*$ closure of $\MMM(X,T)$ in the space $\cP(X)$ of Borel probability measures on $X$, and
define the \emph{ergodic supremum} of a bounded Borel measurable function (potential) $\phi \:X\to\R$ to be 
\begin{equation}\label{e_ergodicmax}
	\mea(T,\phi)\=\sup\limits_{\mu\in\MMM(X,T)}\int\! \phi \, \mathrm{d}\mu.
\end{equation}
Any measure $\mu \in \MMM(X,T)$ that 
attains the supremum in (\ref{e_ergodicmax}) is called a \emph{$(T,\phi)$-maximizing measure},
or simply a \emph{$\phi$-maximizing measure},
and the set of such measures is denoted by
\begin{equation}\label{e_setofmaximizngmeasures}
	\Mmax(T,  \phi) \=\biggl\{ \mu\in\MMM(X,T) : \int \! \phi \,\mathrm{d} \mu = \mea(T,\phi) \biggr\}  . 
\end{equation}
If $\mu$  is a weak$^*$ accumulation point of $\MMM(X,T)$ with $\int \!\phi \,\mathrm{d}\mu=\mea(T,\phi)$,
then $\mu$ is called a \emph{$(T,\phi)$-limit-maximizing measure} and the set of such measures is denoted by
\begin{equation}\label{e_setoflimmaximizngmeasures}
	\Mmax^*(T,  \phi) \=\biggl\{ \mu\in\overline{\MMM(X,T)} : \int \! \phi \,\mathrm{d} \mu = \mea(T,\phi) \biggr\}  . 
\end{equation}

The \emph{ergodic optimization} problem in this article
will be concerned with the study of maximizing,
and limit-maximizing, measures for the Gauss map $\Gau$
and suitable real-valued functions $\phi$.
By contrast with previous investigations of ergodic
optimization for the Gauss map (see e.g.~\cite{Je08,JS10, Pi23}),
or more general countable branch maps
(see e.g.~\cite{BF13,BG10,GG24,GGS25, Io07, JMU06, JMU07}),
here we shall be concerned with those properties of 
maximizing measures that are \emph{typical} in spaces of H\"older functions on $I \= [0,1]$.
Problems of this kind, concerning generic properties, had attracted the interest of Yuan \& Hunt \cite{YH99}, who
conjectured that for an expanding map, or an Axiom A diffeomorphism, there is an
open and dense subset of Lipschitz functions $\phi$
such that the $\phi$-maximizing measure is supported on a single periodic orbit.

The Yuan--Hunt conjecture 
stimulated work by various authors
(see \cite{Bou01, Bou08, CLT01, Mo08, QS12}), and
was eventually 
proved by Contreras \cite{Co16} in the case of 
expanding maps, 
and by Huang, Lian, Ma, Xu \& Zhang \cite{HLMXZ25}
for Axiom A diffeomorphisms; going beyond the setting of uniform hyperbolicity, Li \& Zhang
\cite{LZ25} proved the analogous result for expanding Thurston maps from complex dynamics.
The so-called typical periodic optimization (TPO) conjecture (cf.~\cite{Boc18,Je19}) posits that for a more general
(suitably chaotic) map, and a space of (suitably regular) functions, there is an open dense subset $U$ of the space such that if $\phi\in U$
then the $(T,\phi)$-maximizing measure is unique and supported on a periodic orbit.

While the Gauss map shares certain features of expanding maps, the fact that it has infinitely many branches complicates matters, and the lack of compactness of $\MMM(I,\Gau)$ makes it appropriate to study limit-maximizing measures rather than maximizing measures.
This leads to a new phenomenon: while it can be shown that
typical \emph{periodic} optimization is false,
it seems to be the case that \emph{typical finite optimization} (TFO) holds; in other words, it is conjectured that there is an open dense subset $U$ of the space of Lipschitz functions on $I$ such that if $\phi\in U$
then the $(\Gau,\phi)$-limit-maximizing measure is unique and of finite support.

\subsection{Main results}
Henceforth we write $I\=[0,1]$.
The Gauss map $\Gau:I\to I$  defined by
(\ref{gaussdefnintro})
is such 
that the set $\MMM(I,\Gau)$ of $\Gau$-invariant measures is not weak$^*$ compact. 
 Any
$\Gau$-invariant measure $\mu \in \MMM(I,\Gau)$ satisfying $\mu(I\cap \Q)=0$ will be called \emph{$\Q$-null},  and the set of such measures will be denoted by $\MMM_{\irr}(I,\Gau)$, the notation reflecting the fact that the set of irrational members of $I$ is given full mass by measures in $\MMM_{\irr}(I,\Gau)$. 

The noncompactness of $\MMM(I,\Gau)$ means that
certain continuous functions $\phi$ do not have a maximizing measure; therefore, a satisfactory theory of ergodic optimization for $\Gau$ will require understanding of the closure $\overline{\MMM(I,\Gau)}$.

Each rational number $r_0$ in $(0,1]$ has a finite $\Gau$-orbit $r_0, \, r_1, \, \dots, \, r_n,\, 0$
terminating at the fixed point 0. A related sequence $r_0, \, r_1, \, \dots, \, r_n-1,\, 1,\, 0$ corresponds to the \emph{alternative} form of the (finite) continued fraction representation of $r_0$ (cf.~Section~\ref{sec_Premilinaries}).  
A discrete probability measure concentrating equal mass on each of the
points in a finite sequence of either of the above two forms will be called a \emph{finite-continued-fraction (FCF) measure}.
Although such measures are never $\Gau$-invariant,
they are precisely the ingredient needed to understand the 
closure $\overline{\MMM(I,\Gau)}$, as described by our first main theorem:

\begin{thmx} \label{t_weak_*_closure_M(I,T)}
The closure $\overline{\MMM(I,\Gau)}$
is equal to the convex hull of the union of $\MMM_{\irr}(I,\Gau)$ and the set of finite-continued-fraction measures.
\end{thmx}

A well-known feature of ergodic optimization, 
that is useful both for the identification of specific maximizing measures, and for establishing typical properties of such measures in various function spaces,
is a result known as a \emph{Ma\~n\'e lemma} (see e.g.~\cite{Bou00,Bou01,Bou02,Bou08,Bou11,CLT01,CG95,LT03,Sa99}). 
We establish the following Ma\~n\'e lemma for the Gauss map,
in the context of 
the space $\Holder{\alpha}(I)$
of $\alpha$-H\"older functions, which extends \cite[Lemma~A]{Bou00}:
	\begin{thmx}[Ma\~n\'e Lemma]  \label{t_mane}
		Suppose $\alpha\in (0,1]$ and $\phi\in \Holder{\alpha}(I)$. There exists $u_\phi\in\Holder{\alpha}(I)$ satisfying the functional equation
		\begin{equation}\label{e_mane_equation}
			 u_\phi(x)=\sup_{n\in\N}\Bigl\{\overline{\phi}\Bigl(\frac{1}{n+x}\Bigr)+u_\phi\Bigl(\frac{1}{n+x}\Bigr)\Bigr\}  \quad\text{ for all } x\in I,  
		\end{equation}
        where $\overline{\phi}\=\phi-\mea(\Gau ,\phi)$.
\end{thmx}

Let $T\: X \rightarrow X$ be a Borel measurable map on a metric space $X$, and $\alpha\in(0,1]$.
We say that a function 
$\phi \in \Holder{\alpha}(X)$
has the
 \emph{periodic optimization property} if there exists some $T$-periodic orbit $\cO$ such that the probability measure $\mu_\cO$ uniformly distributed on $\cO$ is the unique $(T,\phi)$-maximizing measure. Similarly, we say that a function 
$\phi \in \Holder{\alpha}(X)$
has the
 \emph{finite optimization property} if there exists some 
 measure $\mu\in\overline{\MMM(X,T)}$ supported on finitely many points
 such that $\mu$ is the unique $(T,\phi)$-limit-maximizing measure. 

\begin{thmx}[Failure of TPO for the Gauss map] \label{t:FailTPO}
Let $\Gau:I\to I$ be the Gauss map.
There is a nonempty open subset of $\Holder{\alpha}(I)$ consisting of functions which do not have the periodic optimization property.
\end{thmx}

Let $\ral^\alpha(\Gau)$ denote the set of those $\alpha$-H\"older
functions such that either the Dirac measure $\delta_0$ at $0$, or some FCF measure, is limit-maximizing; members of $\ral^\alpha(\Gau)$ will be called \emph{rationally maximized}, in view of their close connection to rational orbits (cf.~Definition~\ref{d_rational_orbit_and_rational_measure}).

We prove the following typical finite optimization result:
\begin{thmx}[TFO for rationally maximized potentials] \label{t_rationl_locking_property}
	For $\alpha\in(0,1]$, there is an open subset $U$ of 
    $\Holder{\alpha}(I)$ consisting of functions
    with the finite optimization property, such that $U$ is a dense subset of $\ral^\alpha(\Gau)$.
\end{thmx}

As we shall see (cf.~Example~\ref{ex_rational_maximized}),
there exists a function $\phi \in \ral^\alpha(\Gau)$ with a limit-maximizing measure that is not the Dirac measure $\delta_0$. This, combined with Theorem~\ref{t_rationl_locking_property}, immediately implies Theorem~\ref{t:FailTPO} (cf.~Subsection~\ref{subsct:FailTPO}).

Let $\bad^\alpha(\Gau)$ denote the set of $\alpha$-H\"older
 functions with a maximizing measure whose support does not contain the point $0$ (see Lemma~\ref{l_equivalent_bad_alpha} for equivalent characterisations of $\bad^\alpha(\Gau)$).
 Even though the typical periodic optimization conjecture fails for the Gauss map and the space of $\alpha$-H\"older functions, we are nevertheless able to show that potentials in an open dense subset of $\bad^\alpha(\Gau)$ have the periodic optimization property:

\begin{thmx}[TPO for essentially compact potentials] \label{t_main_theorem_ess_compact}
For $\alpha\in(0,1]$, there is an open subset $U$ of 
$\Holder{\alpha}(I)$ consisting of functions with the periodic optimization property, such that $U$ is a dense subset of $\bad^\alpha(\Gau)$.
\end{thmx}

Inspired by the structure of $\overline{\MMM(I,\Gau)}$
as revealed by Theorem~\ref{t_weak_*_closure_M(I,T)},
and the guiding philosophy that maximizing measures are
generically unique, and 
should be of low complexity, we articulate the following \emph{typical finite optimization (TFO) conjecture}  for the Gauss map $\Gau$: for 
$\alpha\in(0,1]$, and
an open dense subset of $\alpha$-H\"older functions $\phi$, there is a unique $(\Gau,\phi)$-limit-maximizing measure, and it is supported on a finite set.
Note that by Theorem~\ref{t_weak_*_closure_M(I,T)},
and standard convexity arguments,
the finitely supported measures of this conjecture
are either FCF measures, or supported on a single periodic orbit.

As mentioned above, the typical periodic optimization conjecture of Yuan \& Hunt served as a central open problem for a number of years.
Its resolution (in the case of open expanding maps) by Contreras \cite{Co16} relied not only on an original perturbation argument
(cf.~\cite[Proposition~2.6]{Co16}, which was itself inspired by Quas \& Siefkin \cite{QS12}),
but also on work by Bressaud \& Quas \cite{BQ07}, and by Morris \cite{Mo08},
and a fundamental structural result which we shall refer to as the Ma\~n\'e 
lemma\footnote{The terminology ``lemme de Ma\~n\'e'' was used by Bousch in \cite{Bou00},
noting that Ma\~n\'e \cite{Ma96} had proved
an analogous result in the context of Lagrangian flows.}
(it has also been dubbed the Ma\~n\'e--Conze--Guivarc'h lemma \cite{Bou11,Mo09}, in view of \cite{CG95}, or simply a normal form theorem \cite{Je01}),
which had been developed in various settings, notably by
Bousch \cite{Bou00,Bou01,Bou02,Bou08,Bou11}, Contreras, Lopes \& Thieullen \cite{CLT01}, Conze \& Guivarc'h \cite{CG95}, Lopes \& Thieullen \cite{LT03}, and Savchenko \cite{Sa99}.

The  
Ma\~n\'e 
lemma
asserts that, for suitable uniformly hyperbolic dynamical systems $T\:X\to X$ (the early works \cite{Bou00, CLT01} took $T$
to be an expanding map of the circle $X$), for a H\"older function $\phi\:X\to\R$ there exists a H\"older function $u$
(termed a \emph{sub-action} in \cite{CLT01})
such that
$\phi+u-u\circ T\le Q(T,\phi)$ on $X$. 
The freedom to modify\footnote{Note that this modification does not change the ergodic optimization problem: clearly
$\phi+u-u\circ T$ has the same ergodic supremum as $\phi$.}
 $\phi$ by the coboundary $u-u\circ T$ is a crucial step in facilitating subsequent perturbation arguments
in the proof of typical periodic optimization theorems.

Bressaud \& Quas \cite{BQ07} established a closing lemma, guaranteeing that for uniformly hyperbolic maps, every
closed invariant set is approximable (in a certain precise sense) by periodic orbits of period up to $n$, with an accuracy that
is super-polynomial in $n$.
Morris \cite{Mo08} proved that, generically, the (unique) maximizing measure has zero metric entropy: exploiting the existence 
of periodic orbits with small `action' and small period, guaranteed by the Bressaud--Quas 
Closing Lemma, he showed that potential functions can be perturbed so that the new maximizing measures are forced to concentrate mass
near this periodic orbit, and therefore have small entropy. 
Huang, Lian, Ma, Xu \& Zhang \cite{HLMXZ25} developed a proof of typical periodic optimization,
not only for open expanding maps but also for a general class of uniformly hyperbolic maps: their method
combined a novel closing lemma, building on the one of Bressaud \& Quas and the classical one of Anosov, and a corresponding perturbation argument that circumvented 
the intermediate generic zero entropy step.
In the time between the appearance of the initial preprint version of \cite{HLMXZ25} and its eventual publication, an influential expository account by
Bochi \cite{Boc19} gave a simplified presentation of the proof, valid for open expanding maps.
More recently, Li \& Zhang \cite{LZ25} established an analogous typical periodic optimization result for a class of postcritically-finite maps (namely, expanding Thurston maps) arising in complex dynamics, where local closing lemmas are developed and a local perturbation argument devised, in order to avoid finitely many ``bad points'', thereby handling the nonuniform expansion.

In the present article, Theorem~\ref{t_weak_*_closure_M(I,T)} gives a structural description of the closure $\overline{\MMM(I,\Gau)}$ of invariant measures for the Gauss map. In particular, we show that the extreme points of this closure consist precisely of the ergodic $\Gau$-invariant measures together with the finite-continued-fraction measures. To our knowledge, this is the first structural result of this kind for a nonuniformly expanding map with countably many inverse branches. This description provides a necessary framework for understanding the accumulation behavior of maximizing measures, and in particular for the study of limit-maximizing measures, and consequently maximizing measures, for the Gauss map.

Our Theorem \ref{t_mane} represents an analogue of a stronger form of the Ma\~n\'e lemma,
established by Bousch \cite{Bou00, Bou01}, who 
gave conditions for the existence (and uniqueness up to an additive constant) of
a function\footnote{A function $u$
satisfying 
$u(x)=\max_{y\in T^{-1}(x)} (\overline{\phi}+u)(y)$ was
later dubbed a \emph{calibrated} sub-action in \cite{GLT09}.}
$u$ 
satisfying the functional equation 
\begin{equation*}
    u(x)=\max_{y\in T^{-1}(x)} (\overline{\phi}+u)(y),
\end{equation*}
where $\overline{\phi}:=\phi-Q(T,\phi)$.
Such a function can also be viewed as the fixed point of the corresponding Bousch operator (cf.~Definition~\ref{d_Def_BouschOp}).\footnote{The Bousch operator is also known as the Bousch--Lax operator or the Lax operator in the literature as an analogous construction gives the \emph{Lax--Oleinik semi-groups} in the context of Hamiltonian systems.} From a tropical analysis perspective, the Bousch operator and its fixed point $u$ can been seen as tropical analogues of the Ruelle operator and its eigenfunction from the theory of thermodynamic formalism (cf.~\cite{BLL13, LS24}).
Inspired by this analogy and the perspective gained on thermodynamic formalism with countable Markov partitions from the third-named author's on-going work on effective prime orbit theorems for topological Collet--Eckmann maps with Rivera-Letelier, in our setting, we define the Bousch operator using the inverse branches of the Gauss map, directly on $I$ instead of the symbolic space, allowing us to adapt the construction of the fixed point of the Bousch operator to such a discontinuous, non-uniformly expanding dynamical system with a countable Markov partition. To our knowledge, this provides the first formulation and proof of a Ma\~n\'e lemma in such a countable setting.

Our approach to proving Theorem~\ref{t_main_theorem_ess_compact} will be to apply the closing lemma from the uniformly expanding scenario in a neighbourhood of the support of a certain 
maximizing measure, then perform a local analysis inspired by ideas as delineated in \cite{Boc19}, while coping with
difficulties stemming from the fact that $\Gau$ has 
countably many inverse branches, and is discontinuous (see Section \ref{sec_TFO} for further discussion of both the strategy,
and technical details, of this approach).

By contrast the proof of Theorem~\ref{t_rationl_locking_property} is inspired by the proof of a certain periodic locking property (cf.~\cite{BZ15} and \cite[Remark~4.5]{YH99}). 
The technical part in the proof of a key lemma for Theorem~\ref{t_rationl_locking_property} is the construction of a \emph{transport sequence} (following the terminology of \cite{BZ15}) in a given rational orbit. Since there exists more than one FCF measure in a given rational orbit, a more elaborate perturbation argument is needed in the proof of Theorem~\ref{t_rationl_locking_property}.

\subsection{Organisation of the article}
In Section~\ref{sct_Notaion}, some frequently used
notation and assumptions
are fixed.
Section~\ref{sec_Premilinaries} consists of a summary of continued fractions (in particular bounded continued fractions, which are used in the proof of Theorem~\ref{t_main_theorem_ess_compact}), as well as
the Gauss map, its invariant measures, its inverse branches, and its symbolic coding. 
In Section~\ref{sec_maximizing_measures}, we introduce the notion of limit-maximizing measures, and establish some basic properties of maximizing measures.
In Section~\ref{sct_closure_invariant_measures} we introduce finite-continued-fraction measures, and prove a structural theorem for the closure of the set of invariant measures of the Gauss map (Theorem~\ref{t_weak_*_closure_M(I,T)}).
In Section~\ref{Sec_mane_lemma} we discuss the Bousch operator for the Gauss map and prove
the existence of a fixed point of the normalised Bousch operator; we are then
able to establish the Ma\~n\'e Lemma (Theorem~\ref{t_mane}). 
In Section~\ref{sec_TFO} we prove Theorems~\ref{t:FailTPO}, \ref{t_rationl_locking_property}, and~\ref{t_main_theorem_ess_compact}. In fact, slightly stronger versions of Theorems~\ref{t_main_theorem_ess_compact} and~\ref{t_rationl_locking_property}, in terms of \emph{locking}, are formulated and established in Section~\ref{sec_TFO}.
In an appendix we establish the periodic locking property for the Gauss map, which is used in the proof of Theorem~\ref{t_main_theorem_ess_compact}.

\section{Notation} \label{sct_Notaion}

Let $I$ denote the interval $[0,1]$, 
equipped with the Euclidean metric.
Throughout this article, the golden ratio is denoted by $\theta\=\frac{\sqrt{5}+1}{2}$,
we set
$c_0\=\frac{2\sqrt{5}}{5}$, and for each $\alpha\in (0,1]$ we define 
   \begin{equation}\label{e_def_K_alpha}
		K_\alpha\=\frac{ c_0^{-2\alpha}}{1-\theta^{-2\alpha}}=c_0^{-2\alpha}\sum\limits_{n=0}^{+\infty}\theta^{-2\alpha n} .
	\end{equation}
Let $\mathbbold{0} \: I \rightarrow \R$ be the function that sends every $x\in I$ to $0$.

We follow the convention that $\mathbb{N} \coloneqq \{1, \, 2, \, 3, \, \dots\}$, $\N_0 \coloneqq \{0\} \cup \mathbb{N}$, and
$\hbN \= \N\cup \{\infty\}$.
The sets $\N$, $\N_0$, and $\hbN$
will all be equipped with the order relations $<$, $\leq$, $>$, $\geq$, defined in the obvious way.
We endow $\N$ with the discrete topology and $\N^n$, $\N^\N$ with the product topology for each $n\in \N$.
Let $\rho$ be the metric on $\N$ defined by
\begin{equation*}
	\rho(a,b)\=\begin{cases}
		\Absbig{\frac{1}{a}-\frac{1}{b}}  & \text{if } a,b \in \N\text{ and }a\neq b   , \\
		0 &\text{if } a=b   . 
	\end{cases}
\end{equation*}
	Observe that $\rho$ is totally bounded, and the completion of $\N$ with respect to $\rho$ is a compact metric space that can be identified with $\hbN=\N \cup \{\infty\}$. The metric $\rho $ extends to $\hbN$, and we denote the extension by $\hrho$. Note that $\hrho (n,\infty)=\frac{1}{n}$ for every $n\in \N$, and $\hrho(\infty,\infty)=0$.  We also endow $\hbN^n$ with the product topology for each $n\in \N$.
	
For $x\in\R$, let $\lfloor x\rfloor$ denote the largest integer $\leq x$. The cardinality of a set $A$ is denoted by $\card A$. For sets $A$, $B$, denote $A\triangle B=(A\smallsetminus B)\cup (B\smallsetminus A)$.

The collection of all maps from a set $X$ to a set $Y$ is denoted by $Y^X$. We denote the restriction of a map $g \: X \rightarrow Y$ to a subset $Z$ of $X$ by $g|_Z$. 
Given a map $f\: X\rightarrow X$ and a real-valued function $\psi\: X\rightarrow \R$, we write
$
S_n \psi (x)   \coloneqq \sum_{j=0}^{n-1} \psi \bigl( f^j(x) \bigr) 
$
for $x\in X$ and $n\in\N_0$. Note that by definition we always have $S_0 \psi = 0$. We denote the set of bounded functions from $X$ to $\R$ by $B(X)$.

Let $(X,d)$ be a metric space. For subsets $A,\,B\subseteq X$, we set $d(A,B) \coloneqq \inf \{d(x,y):x\in A,\,y\in B\}$, and $d(A,x)=d(x,A) \coloneqq d(A,\{x\})$ for $x\in X$. For each subset $Y\subseteq X$, we denote the diameter of $Y$ by $\diam(Y) \= \sup\{d(x,y): x, \, y\in Y\}$.
For each $r>0$, we define $B^r_d(A)$ to be the open $r$-neighbourhood $\{y\in X : d(y,A)<r\}$ of $A$, and $\overline{B}^r_d(A)$ the closed $r$-neighbourhood $\{y\in X : d(y,A)\leq r\}$ of $A$. 

We set $C(X)$ (resp.\ $B(X)$) to be the space of continuous (resp.\ bounded) functions from $X$ to $\R$, $\mathcal{M}(X)$ the set of finite signed Borel measures, and $\mathcal{P}(X)$ the set of Borel probability measures on $X$. If we do not specify otherwise, we equip $C(X)$ with the uniform norm $\norm{\cdot}_{\infty}$. For a Borel measurable map $g \: X \rightarrow X$, $\mathcal{M}(X,g)$ is the set of $g$-invariant Borel probability measures on $X$ and $\MMM_{\erg}(X,g)$ is the set of ergodic measures in $\mathcal{M}(X,g)$. For each $x\in X$, we denote by $\delta_x$ the Dirac delta measure on $x$ given by $\delta_x(A) = 1$ if $x\in A$ and $0$ otherwise for all Borel measurable set $A \subseteq X$. For $x\in X$, we denote the open (resp.\ closed) ball of radius $r$ centered at $x$ by $B_d^r(x)$ (resp.\ $\overline{B}_d^t(x)$). For $\mu \in \MMM(X)$ and a $\mu$-integrable function $\phi \: X\rightarrow\R$, we write
	\begin{equation*}
		\langle\mu,\phi\rangle\= \int \! \phi \, \dd \mu   . 
	\end{equation*}

The space of real-valued H\"{o}lder continuous functions with an exponent $\alpha\in (0,1]$ is denoted by $\Holder{\alpha}(X,d)$. For each $\psi\in\Holder{\alpha}(X,d)$, we denote
\begin{equation*}   \label{e_Def|.|alpha}
	\Hseminorm{\alpha,\,X}{\psi} \coloneqq \sup \{ \abs{\psi(x)- \psi(y)} / d(x,y)^\alpha  : x, \, y\in X, \,x\neq y \},
\end{equation*}
and the H\"{o}lder norm is given by $\Hnorm{\alpha,\,X}{\psi} \=  \Hseminorm{\alpha,\,X}{\psi}  + \norm{\psi}_{\infty}$. We omit the subscript $X$ when it does not cause confusion.

For a nonempty set $E$ and $n\in\N$, let us denote
$E^\N\=\{ (a_k)_{k\in\N}: a_k\in E\text{ for all }k\in\N\}$, $E^n\=\{ (a_k)_{k=1}^n : a_k\in \cA\text{ for all }1\leq k\in n\}$, and $E^*\=\bigcup_{k\in\N}E^k$.

Define the (left) shift map 
\begin{equation*}
	\sigma = \sigma_{E^\N} \: E^\N \to E^\N
\end{equation*}
by $\sigma  (A )\=  (a_{n+1})_{n\in \N} $
for all
$A = ( a_n)_{n\in \N} \in E^\N$. We often omit the subscript $E^\N$ in $\sigma_{E^\N}$ when it is clear from the context. Sometimes we also consider this shift as defined on words of finite length. Given $A\in E^*$, by $\abs{A}$ we denote the length of the word $A$, i.e.,
the unique $n$ such that $A\in E^n$. If $A\in E^\N$ and $n\in\N$, then denote
$A|_n = a_1 a_2\dots a_n$. 

We denote $\hSigma\=\hbN^{\N}$ and $\Sigma_m\=\{1,\,2,\,\dots,\,m\}^{\N}$.

For each $(a_i)_{i\in\N}\in\N^\N$ and each $n\in\N$ we denote by $[a_1,\dots, a_n]$ the rational number 
\begin{equation*}
	[a_1, \dots, a_n]\=\frac{1}{a_1+\frac{1}{a_2+\frac{1}{\cdots+\frac{1}{a_n}}}},
\end{equation*}
and denote $[a_1,a_2,\dots]\=\lim_{n\to+\infty} [a_1, \dots, a_n]$, the real number in $I$ with continued fraction $(a_n)_{n\in\N}$. We denote by $[\overline{a_{1}, \dots, a_n}]$ the number infinite periodic continued fraction expansion.

\section{The Gauss map and continued fractions}\label{sec_Premilinaries}
\subsection{Continued fractions}
\begin{definition}[Continued fraction transformation]\label{D_Gauss map}
The \emph{continued fraction transformation} (or  \emph{Gauss\footnote{The common convention of referring to $\Gau$ as the Gauss map reflects the foundational work of Gauss \cite{Ga12} (cf.~e.g.~\cite{Br91}) on the statistical properties of continued fractions (having worked on probabilistic aspects of continued fractions in 1799--1800 (cf.~\cite{Ga12}), his understanding of the so-called Gauss measure was mentioned in an 1812 letter to Laplace (see \cite[pp.~147--148]{Br91}).   
    As a self-map of $I=[0,1]$, the value $G(0)$ is habitually defined to be $0$ (see e.g.~\cite{FSU14, Ma87,PW99}).} map}), 
    is the map $\Gau \: I\rightarrow I$ defined by
\begin{equation*}
	\Gau(x)=\begin{cases}
		\frac{1}{x}-\lfloor\frac{1}{x}\rfloor  & \text{if }x\in (0,1]  , \\
		0 &\text{if }x=0 .
	\end{cases}
\end{equation*}
\end{definition}
Note that $\Gau^{-1}(1)=\emptyset$, $\Gau^{-1}(0)=\{0\}\cup\{1/n\}_{n\in \N}$, and
\begin{equation}\label{explicitpreimageset}
    \Gau^{-1}(x)= \{ 1/(a+x): a\in\N \} \quad \text{ for }x\in(0,1) .
\end{equation}
The connection between $\Gau$ and continued fractions is that
	for each irrational $x\in I$, there are unique natural numbers $a_i(x)\=\bigl\lfloor 1\big/\Gau^{i-1}(x)\bigr\rfloor$ such that
\begin{equation}\label{ctd_frac_expansion}
	x=[a_1(x),a_2(x),\dots]\=\lim\limits_{n\to +\infty}[a_1(x),a_2(x),\dots,a_n(x)],
\end{equation}
and (\ref{ctd_frac_expansion}) is the \emph{continued fraction expansion} of $x$ (cf.~\cite[Lemma~4D]{Sc80}).

\begin{lemma}\label{l_continued_fraction_of_rational_number}
	Every $x\in (0,1)\cap \Q$ has precisely two finite continued fraction expansions, one of the form $[a_1,a_2,\dots,a_n]$ with $a_n\geq 2$, and the other of the form $[a_1,a_2,\dots,a_n-1,1]$.
\end{lemma}
\begin{proof}
	See \cite[Lemma~4C]{Sc80}.
\end{proof}
To accommodate the nonuniqueness of continued fraction expansions of rational numbers, it will be convenient to distinguish between those expansions which do, or do not, contain the digit 1, in the following way:

\begin{notation}\label{n_Rn_g_f}
	We let 
	\begin{equation*}
		R_1\=\{1/a : a\in\N,\, a \ge 2\}
		=\{ [a]: a\in\N, \, a \ge 2\}
	\end{equation*} 
    denote the set of unit fractions with the value 1 excluded, or in other words
    \begin{equation*}
        R_1=\Gau^{-1}(0)\smallsetminus\{0,\,1\}.
    \end{equation*}
    For $n\ge 2$, set
	\begin{equation}\label{e_def_sets_R_m}
		R_n\=\Gau^{-(n-1)}(R_1)
        =\{[a_1,a_2,\dots,a_n]:(a_1,\dots,a_n)\in \N^n,\, a_n\geq 2\}.
	\end{equation}
    Note that
    \begin{equation*}
        R_n= \Gau^{-n}(0)\smallsetminus \Gau^{-(n-1)}(0),
    \end{equation*}
    in other words $R_n$ is the set of those points in $I$ whose $\Gau$-orbit lands at 0 for the first time after precisely $n$ iterates.\footnote{Note that the point 1 is not in the image of $\Gau$, so the only eventually fixed orbit containing 1 is the two-element orbit $\{0,\,1\}$.}
    
	For each $n\in \N$, define $\cA_n\subseteq \N^n$ by 
	\begin{equation}\label{e_def_cA_n}
		\cA_n\=\{ (a_1,\dots,a_n)\in \N^n : a_n\geq 2\}  . 
	\end{equation}
	Denote $\cA \=\bigcup_{n=1}^{+\infty}\cA_n \subseteq \N^*$.
	For each $n\in \N$, define $\cB_n\subseteq \N^n$ by 
	\begin{equation*}
		\cB_n\=\{ (a_1,\dots,a_{n})\in \N^{n} : a_n=1\}  . 
	\end{equation*}
	Denote $\cB\=\bigcup_{n=1}^{+\infty}\cB_n \subseteq \N^*$. Clearly, for each $n\in \N$, 
    the sets  $\cA_n$ and $\cB_n$ together form a partition of
    $\N^n$.
	
	For each $n\in \N$, define the bijection $f_n \: \cA_n\to \cB_{n+1}$ by 
	\begin{equation*}
		(a_1,a_2,\dots,a_{n})\mapsto (a_1,a_2,\dots, a_{n}-1,1)  . 
	\end{equation*}
	 Let $f\: \cA \to \cB$ 
    be the function defined by $f|_{\cA_n}\=f_n$
    for all $n\in\N$,
   Define $g_n \: \cA_n \to R_n$ by 
	\begin{equation}\label{e_def_g_n}
		g_n: (a_1,a_2,\dots,a_n)\mapsto [a_1,a_2,\dots,a_n]  . 
	\end{equation}
	It is easy to see that $g_n$ maps $\cA_n$ bijectively to $R_n$ (see Lemma~\ref{l_basic_properties_rational_measures}~(iv)). Let $g\: \cA \to (0,1)$ 
    be the function defined by $g|_{\cA_n}\=g_n$
    for all $n\in\N$.
\end{notation}
Let $a_1,\,a_2,\,a_3,\,\dots$ be variables. We define polynomials $p_0,\,q_0,\,p_1,\,q_1,\,p_2,\,q_2,\,\dots$, with $p_n$ and $q_n$ being polynomials in $a_1,\,\dots,\,a_n$, as follows:

\begin{equation*}
	\begin{aligned}
		p_0&=0,\quad p_{-1}=1,\quad p_n=a_np_{n-1}+p_{n-2},\quad n=1,\,2,\,\dots\\
		q_0&=1,\quad q_{-1}=0,\quad q_n=a_nq_{n-1}+q_{n-2},\quad n=1,\,2,\,\dots
	\end{aligned}
\end{equation*}

\begin{definition}[{\bf Continuants}]\label{d_continuants_and_convergents}
  Fix $n\in \N$. For any finite word $\fa=(a_1,\dots,a_n)\in \N^n$, the integers $p_n(\fa)$ and $q_n(\fa)$ satisfy
  \begin{equation*}
  	[a_1,a_2,\dots,a_n]=\frac{p_n(\fa)}{q_n(\fa)}  . 
  \end{equation*}
The denominator $q_n(\fa)$ is called a \emph{continuant} of the continued fraction $[a_1,\dots,a_n]$. We also define the $q_k(\fa)$ for $k<n$, even if the word $\mathbf{a}$ has length larger than $k$; in this case $q_k(\fa)$ is just the continuant $q_k(\mathbf{a}|_k)$, where $\mathbf{a}|_k$ is the restriction $(a_1,\dots,a_k)$.
\end{definition}

The following lower bound on the growth of continuants is expressed in terms of the constants $c_0$ and $\theta$
(defined in Section~\ref{sct_Notaion}): 
\begin{lemma}\label{l_lower_bound_q_n}
	If $n\in\N$ and $\mathbf{a}\in \N^n$, then $q_n(\mathbf{a})\geq c_0\theta^n$. 
\end{lemma}

\begin{proof}
	See \cite[Lemma~1]{JoSa16}.
\end{proof}

\subsection{Invariant measures}

\begin{notation}
	Fix $\mu\in \MMM(I)$. Define $P_1(\mu),\,P_2(\mu)\in \MMM(I)$ by 
	\begin{equation*}
    P_1(\mu)(A)\= \mu(A\cap\Q)\quad\text{ and } \quad
   P_2(\mu)(A)\=\mu(A\smallsetminus \Q)
    \end{equation*}
	for all Borel sets $A\subseteq I$. 
\end{notation}

\begin{lemma}\label{l_Q_Q^c_invatiant}
	We have $\Gau(I\cap \Q)=[0,1)\cap\Q$, $\Gau^{-1}(I\cap \Q)=I\cap \Q$, and $\Gau(I\smallsetminus \Q)=\Gau^{-1}(I\smallsetminus \Q)=I\smallsetminus \Q$. The set $\MMM(I,\Gau)$ is the convex hull of $\MMM_{\irr}(I,\Gau)\cup\{\delta_0\}$ and $\overline{\MMM(I,G)}=\overline{\MMM_{\irr}(I,G)}$. 
\end{lemma}

\begin{proof}
	The first part follows immediately from Definition~\ref{D_Gauss map}. Clearly, $\delta_0\in \MMM(I,\Gau)$ and $\MMM_{\irr}(I,\Gau)\subseteq \MMM(I,\Gau)$.
	
	Fix $\mu \in \MMM(I,G)$. For each Borel set $A\subseteq I$, since $\Gau^{-1}(I\cap \Q)=I\cap \Q$ and $\Gau^{-1}(I\smallsetminus \Q)=I\smallsetminus \Q$, we have
	\begin{align*}
		\Gau^{-1}(A\cap \Q)&=\Gau^{-1}(A)\cap \Gau^{-1}(I\cap\Q)=\Gau^{-1}(A)\cap \Q \quad\text{ and }\\
		\Gau^{-1}(A\smallsetminus \Q)&=\Gau^{-1}(A)\cap \Gau^{-1}(I\smallsetminus\Q)=\Gau^{-1}(A)\smallsetminus \Q  . 
	\end{align*}
	Thus, for each Borel set $A\subseteq I$, we have
	\begin{equation*}
		\begin{aligned}
			P_1(\mu)\bigl(\Gau^{-1}(A)\bigr)&=\mu\bigl(\Gau^{-1}(A)\cap \Q\bigr)=\mu\bigl(\Gau^{-1}(A\cap \Q)\bigr)=\mu(A\cap \Q)=P_1(\mu)(A),\\
			P_2(\mu)\bigl(\Gau^{-1}(A)\bigr)&=\mu\bigl(\Gau^{-1}(A)\smallsetminus \Q\bigr)=\mu\bigl(\Gau^{-1}(A\smallsetminus \Q)\bigr)=\mu(A\smallsetminus \Q)=P_2(\mu)(A).
		\end{aligned}	
	\end{equation*}
	Hence $P_1(\mu)$ and $P_2(\mu)$ are $\Gau$-invariant. Note that
	\begin{equation*}
		\mu(\{0\})=\mu\bigl(\Gau^{-1}(0)\bigr)=\mu(\{0\})+\sum\limits_{n=1}^{+\infty}\mu(\{1/n\}).
	\end{equation*}
	Since $\mu$ is $\Gau$-invariant, we have $\mu(\{1/n\})=0$ for all $n\in\N$ and  $\mu\bigl(\Gau^{-k}(1/n)\bigr)=0$, for all $n,\,k\in\N$. Consequently, we obtain
	$\mu((I\cap\Q)\smallsetminus\{0\})=0$ and $P_1(\mu)=\mu(\{0\})\delta_0$. When $\mu(\{0\})=1$, we have $\mu=P_1(\mu)=\delta_0$. When $\mu(\{0\})<1$, we have $\frac{1}{\mu(I\smallsetminus\Q)}P_2(\mu)\in \MMM_{\irr}(I,\Gau)$. Therefore, $\mu$ is contained in the convex hull of $\MMM_{\irr}(I,\Gau)\cup\{\delta_0\}$. But $\mu$ was an arbitrary member of 
    $ \MMM(I,G)$, so we see that $\MMM(I,\Gau)$ is the convex hull of $\MMM_{\irr}(I,\Gau)\cup\{\delta_0\}$. 
        
    Now $\MMM_{\irr}(I,\Gau)\subseteq \MMM(I,\Gau)$, so
    $\overline{\MMM_{\irr}(I,\Gau)}\subseteq \overline{\MMM(I,\Gau)}$,
    and the reverse inclusion $\overline{\MMM(I,\Gau)}\subseteq \overline{\MMM_{\irr}(I,\Gau)}$
    follows from $\delta_0\in \overline{\MMM_{\irr}(I,\Gau)}$, since $\lim\limits_{n\to +\infty}[\overline{n}]=0$ and $\delta_{[\overline{n}]}\in \MMM_{\irr}(I,G)$, so the second part now follows. 
\end{proof}

\begin{rem}
	The set $\MMM(I,  \Gau)$ is not closed with respect to the weak$^*$ topology. Write $x_n\=[2, n,\overline{2, n}]$ for $n\in \N$. Then $\{x_n\}_{n\in\N}$ are the periodic points of $\Gau$ satisfying $\Gau^2(x_n)=x_n$, $\lim\limits_{n\to+\infty}x_n=\frac{1}{2}$, and $\lim\limits_{n\to+\infty}\Gau(x_n)=0$. Define
	$\mu_n\=\frac{1}{2}(\delta_{x_n}+\delta_{\Gau(x_n)})$.
	Then the weak$^*$ limit of the sequence $\{\mu_n\}_{n\in\N}$ is $\mu=\frac{1}{2}(\delta_0+\delta_{1/2})$, which is not $\Gau$-invariant (see Lemma~\ref{l_Q_Q^c_invatiant}).
\end{rem}

The following lemma allows us to abuse notation by identifying
$\MMM(I\smallsetminus \Q,\Gau|_{I\smallsetminus\Q})$
with
$\MMM_{\irr}(I,\Gau)\subseteq \MMM(I,\Gau)$. Define $j \: I\smallsetminus\Q \to I$ to be the inclusion map and $j_*\: \MMM(I\smallsetminus\Q,\Gau|_{I\smallsetminus\Q})\to \MMM(I,\Gau)$ to be the pushforward of $j$, i.e., for each Borel subset $A\subseteq I$ and $\mu\in \MMM(I\smallsetminus\Q,\Gau|_{I\smallsetminus\Q})$, we define
\begin{equation}\label{e_def_j_*}
	j_*(\mu)(A)\=\mu\bigl(j^{-1}(A)\bigr)=\mu (A\smallsetminus\Q)  . 
\end{equation} 

\begin{lemma}\label{l_identification_1}
      The map $j_*$ is a continuous bijection from $\MMM(I\smallsetminus\Q,\Gau|_{I\smallsetminus\Q})$ to $\MMM_{\irr}(I,\Gau)$.
\end{lemma}

\begin{proof}
	From the fact that $j$ is continuous, $j_*$ is well-defined and continuous.
    Assume that $\mu_1,\,\mu_2\in \MMM(I\smallsetminus \Q, \Gau|_{I\smallsetminus \Q})$ with $j_*(\mu_1)=j_*(\mu_2)$.
    Then by (\ref{e_def_j_*}), $\mu_1 (A\smallsetminus\Q)=\mu_2 (A\smallsetminus\Q)$ for every Borel measurable subset $A\subseteq I$. So $\mu_1=\mu_2$, and consequently $j_*$ is injective.
	
   Fix $\mu \in \MMM_{\irr}(I,\Gau)$. Define $\nu \in \MMM(I\smallsetminus\Q)$ by $\nu(A)\=\mu(A)$ for each Borel subset $A\subseteq I\smallsetminus\Q$.
   By the assumption that $\mu \in \MMM_{\irr}(I,\Gau)$ and the definition of $\MMM_{\irr}(I,\Gau)$, we get $\nu(I\smallsetminus\Q)=1$. 
   For each Borel subset $A\subseteq I\smallsetminus\Q$, from the fact that $\Gau^{-1}(I\smallsetminus \Q)=I\smallsetminus \Q$, the definition of $\nu$, and the fact that $\mu \in \MMM_{\irr}(I,\Gau)$, 
   \begin{equation*}
   	\nu\bigl((\Gau|_{I\smallsetminus\Q})^{-1}(A)\bigr)=\nu \bigl(\Gau^{-1}(A)\bigr)=\mu\bigl(\Gau^{-1}(A)\bigr)=\mu(A)=\nu(A)  . 
   \end{equation*}
Hence $\nu \in \MMM(I\smallsetminus\Q,\Gau|_{I\smallsetminus\Q})$. For each Borel subset $B\subseteq I$, by (\ref{e_def_j_*}), the definition of $\nu$, and the fact that $\mu \in\MMM_{\irr}(I,\Gau)$, we have $j_*(\nu)(B)=\nu (B\smallsetminus\Q)=\mu (B\smallsetminus\Q)=\mu (B)$. So $j_*(\nu)=\mu$ and consequently $j_*$ is a surjective map from $\MMM(I\smallsetminus\Q,T|_{I\smallsetminus\Q})$ to $\MMM_{\irr}(I,\Gau)$. This lemma now follows. 
\end{proof}

\subsection{Inverse branches}

The following notational conventions for
inverse branches of the Gauss map
follow Jordan \& Sahlsten \cite{JoSa16}.

\begin{definition}[Inverse branches]
	For each $a\in \N$, 
    define
    $\Gau_a \: I\to I$ by
	\begin{equation}\label{e_def_T_a}
		\Gau_a(x)\=\frac{1}{a+x}\quad\text{ for all }x\in I   , 
	\end{equation}
    so that (\ref{explicitpreimageset}) becomes
    \begin{equation*}\label{explicitpreimageset2}
        \Gau^{-1}(x)=\{ \Gau_a(x): a\in\N\}
        \quad \text{ for all }x\in(0,1)  . 
    \end{equation*}
    Let us denote $I_a\=\Gau_a(I)$. For each $n\in \N$ and $\fa = (a_1,a_2,\dots,a_n)\in \N^n$, define $\Gau_\mathbf{a}\:I\to I$ to be 
    \begin{equation}\label{e_def_T_fa}
	\Gau_\fa\=\Gau_{a_1}\circ \Gau_{a_2}\circ \dots \circ \Gau_{a_n}  . 
    \end{equation}
In other words, $\Gau_\mathbf{a}(x)=[a_1,a_2,\dots,a_n+x]$ for each $\mathbf{a}\in\N^n$ and each $x\in I$. Denote $I_{\fa}\=\Gau_\fa(I)$. 
\end{definition}

\begin{notation}\label{r_T_infty_and_T_hfa}
	It will be convenient to define $\Gau_\infty \: I\to I$ by setting $\Gau_\infty (x)=0$ for each $x\in I$.
    
    Recalling that $\hbN\=\N\cup\{\infty\}$,
    for $n\in \N$ and $\hfa \in \hbN^n$ we define $\Gau_{\hfa}=( \ha_1,\ha_2,\dots,\ha_n )\:I\to I$ by 
	\begin{equation}\label{e_def_T_hfa}
		\Gau_{\hfa}\=\Gau_{\widehat{a}_1}\circ \Gau_{\widehat{a}_2}\circ \dots \circ \Gau_{\widehat{a}_n}  , 
	\end{equation}
	and set 
    \begin{equation*}
        I_{\hfa}\=\Gau_{\hfa}(I).
    \end{equation*}
For  $\widehat{{\bf a}}=a_1a_2\dots \in \hbN^{\N} \cup \bigcup_{n=1}^{+\infty} \hbN^n$,
we define its \emph{$\infty$-index} $\iota(\hfa)$ to be the smallest $k$ such that $\widehat{a}_k=\infty$, so that $\iota(\widehat{{\bf a}})=+\infty$ precisely when $\hA\in \N^{\N} \cup \bigcup_{n=1}^{+\infty} \N^n$.     
\end{notation}

\begin{lemma}\label{l_q_n_and_T_a_prime}
	If $n\in \N$ and $\fa\in \N^n$, then 
	\begin{equation*}\label{e_est_devative_T_a}
		 q_n(\mathbf{a})^{-2} \big/ 4
        \leq \abs{\Gau'_\mathbf{a}}
        \leq q_n(\mathbf{a})^{-2}  , 
	\end{equation*}
and in particular, the length $\diam I_\mathbf{a}$
satisfies
$\frac{1}{4}q_n(\mathbf{a})^{-2}\leq \diam I_\mathbf{a} \leq q_n(\mathbf{a})^{-2}$.
\end{lemma}
\begin{proof}
	See \cite[Lemma~2]{JoSa16}.
\end{proof}

\begin{prop}\label{p_inverse_branches}
	If $n\in \N$, then
	\begin{enumerate}[label=\rm{(\roman*)}]
		\smallskip
		\item\label{p_inverse_branches_est_derivative} $\abs{\Gau_\mathbf{a}^\prime(x)}\leq c_0^{-2} \theta^{-2n}$ for all $x\in I$ and ${\bf a}\in \N^n$.
		\smallskip
		
		\item\label{p_inverse_branches_(ii)} For  $\widehat{{\bf a}}=a_1a_2\dots a_n\in \hbN^n\smallsetminus \N^n$, 
        \begin{equation*}
        I_{\widehat{{\bf a}}}=
        \begin{cases}
            \{0\} &\text{if }\iota(\widehat{{\bf a}})=1,\\
            \{[a_1,a_2,\dots,a_{\iota(\widehat{{\bf a}})-1}]\} & \text{if }\iota(\widehat{{\bf a}}) \ge 2.
        \end{cases}
        \end{equation*}

		\smallskip
		
		\item\label{p_inverse_branches_monotonicity}For  $\fa\in \N^n$, the map $\Gau_{\fa}$ is strictly increasing for $n$ even, and strictly decreasing for $n$ odd.
		\smallskip 
		
		\item\label{p_inverse_branches_continuous} For each $x\in I$, the map
        $ \hbN^n\to I$, $\hfa\mapsto \Gau_{\hfa}(x)$,
        is continuous.
				
		\smallskip
		\item\label{p_inverse_branches_closure_of_preimage} For each $x\in I$, the closure of $\{\Gau_{\mathbf{a}}(x): \mathbf{a}\in \N^n\}$ is 
        the set $\bigl\{\Gau_{\widehat{\mathbf{a}}}(x) : \widehat{\mathbf{a}}\in \hbN^n\bigr\}$.		

		\smallskip
		
		\item\label{p_inverse_branches_image_and_preimage} If $x\in I\smallsetminus \Q$ and $\mathbf{a}\in\N^n$, then $\Gau_{\sigma^i(\mathbf{a})}(x)=\Gau^{i}(\Gau_\mathbf{a}(x))$ for all integer $1\le i\le n-1$.
		
	\end{enumerate}
\end{prop}

\begin{proof}
\ref{p_inverse_branches_est_derivative} Lemmas~\ref{l_lower_bound_q_n} and \ref{l_q_n_and_T_a_prime} immediately yield
$\abs{\Gau_\mathbf{a}^\prime(x)}\leq q_n(\mathbf{a})^{-2}\leq  c_0^{-2} \theta^{-2n}$.

\smallskip

\ref{p_inverse_branches_(ii)} follows from the definition of $\Gau_{\widehat{{\bf a}}}$ (see (\ref{e_def_T_hfa})) and the fact that $\Gau_\infty \equiv 0$ on I.

\smallskip

\ref{p_inverse_branches_monotonicity} is immediate from 
the fact that $\Gau_a$ is strictly decreasing for each $a\in \N$ (cf.~(\ref{e_def_T_a})),
and $\Gau_{\fa}$
is the composition of $n$ such maps
(cf.~(\ref{e_def_T_fa})). 

\smallskip

\ref{p_inverse_branches_continuous} Suppose 
as $m$ tends to infinity, $\hfa_m=a_{m,1}a_{m,2}\dots a_{m,n}$ converges
to $\hfa=a_1a_2\dots a_n$ in $\hbN^n$. 

If $\hfa \in \N^n$, then $a_{m,k}=a_k$ for all $k\in\{1,\,2,\,\dots,\,n\}$ when $m$ is large enough. So, $\Gau_{\hfa_m}(x)=T_{\hfa}(x)$ when $m$ is large enough.

If $\hfa \in \hbN^n\smallsetminus \N^n$, let $l\=\iota(\hfa)$ be the $\infty$-index of $\hfa$. When $l=1$, we get $a_1=\infty$ and $\lim\limits_{m\to +\infty} a_{m,1}=+\infty$. So $\Gau_{\hfa}(x)=0=\lim\limits_{m\to +\infty} \Gau_{\hfa_m}(x)$. When $l\geq 2$, we obtain $\Gau_{\hfa}(x)=[a_1,a_2,\dots, a_{l-1}]$,  $a_{m,k}=a_k$ for all $k\in\{1,\, 2,\, \dots,\, l-1\}$ when $m$ is large enough, and $\lim\limits_{m\to +\infty}a_{m,l}=+\infty$. Hence, $\lim\limits_{m\to +\infty}\Gau_{\hfa_m}(x)=\Gau_{\hfa}(x)$, and \ref{p_inverse_branches_continuous} follows.

\smallskip

\ref{p_inverse_branches_closure_of_preimage} Denote $W\=\{\Gau_{\fa}(x):\fa\in \N^n\}$ and $\hW\=\bigl\{\Gau_{\hfa}(x):\hfa\in \hbN^n\bigr\}$.
Fix $y\in \hW\smallsetminus W$.
Then $y=\Gau_{\widehat{{\bf b} }}(x)$ for some $\widehat{{\bf b} }=b_1b_2\dots b_n\in \hbN^n\smallsetminus \N^n$.
Let $k\=\iota(\widehat{\fb})$ be the $\infty$-index of $\widehat{\fb}$.
By (ii), we get $y=\Gau_{\widehat{\fb}}(x)=0$ when $k=1$ and $y=\Gau_{\widehat{\fb}}(x)=[b_1,\dots,b_{k-1}]$ when $k\geq 2$.
For each $m$, define $\fb_m\=(m,\dots,m)\in \N^n$ when $k=1$ and define $\fb_m\=( b_1,\dots, b_{k-1},m,\dots, m)\in \N^n$ when $k\geq 2$. It is easy to check $\lim\limits_{m\to +\infty}\Gau_{{\bf b}_m}(x)=y$. Hence, $\hW \subseteq \overline{W}$.

Assume that $z\in \overline{W}$. Then there exists a sequence $\{\fa_m\}_{m\in\N} $ in $ \N^n$ with $
\lim\limits_{m\to +\infty}{\Gau_{\fa_m}(x)}=z$.
Since $\hbN^n$ is compact, there exists $\fc \in \hbN^n$ and a subsequence $\{\fa_{m_k}\}_{k\in \N}$ of $\{\fa_m\}_{m\in\N}$ such that $\{\fa_{m_k}\}_{k\in \N}$ converges to $\fc$ as $k$ tends to infinity. 
 By \ref{p_inverse_branches_continuous}, we obtain that 
 \begin{equation*}
 	\Gau_{\fc}(x)=\lim\limits_{k\to +\infty}{\Gau_{\fa_{m_k}}(x)}=\lim\limits_{m\to +\infty}{\Gau_{\fa_m}(x)}=z  , 
 \end{equation*}
 so $z\in \hW$. 
 But $z$ was an arbitrary member of $\overline{W}$, so we have shown that $\overline{W}\subseteq \hW$, and  \ref{p_inverse_branches_closure_of_preimage} follows.
 
 \smallskip
 
 \ref{p_inverse_branches_image_and_preimage} Assume that $x=[b_1,b_2,\dots,b_n,\dots]\in I\smallsetminus\Q$ and $\fa=[a_1,\dots,a_n]$. Then, by the definition of $\Gau_{\fa}$ and $\Gau$, we see that $\Gau_{\sigma^i(\fa)}(x)=[a_{i+1},\dots,a_n,b_1,b_2,\dots,b_n,\dots]=\Gau^i(\Gau_{\fa}(x))$, as required.
\end{proof}

\begin{notation}
	Given $\psi\in \R^I$, $n\in\N$, and $\hfa\in \hbN^n$, define the function $S_{n,\hfa}\psi\: I\to \R$ by 
	\begin{equation}\label{e_S_n_bf_a}
		S_{n,\hfa}\psi\=\sum\limits_{i=0}^{n-1}\psi\circ \Gau_{\sigma^i(\hfa)}  . 
	\end{equation}
\end{notation}

We need the following standard lemma (cf.~\cite[pp.~144--147]{Wa78}, and see also \cite[Lemma~3.1.2]{MU03} 
in the context of conformal graph directed Markov systems). 

\begin{lemma}\label{l_Distortion}
     Suppose $\alpha\in(0,1]$ and $\phi\in \Holder{\alpha}(I)$. For all $n\in\N$, $\fa\in \N^n$, and $x,\,y\in I$,  
	\begin{equation*}
		\abs{S_{n,\fa}\phi(x)-S_{n,\fa}\phi(y)}
		\leq   K_\alpha \Hseminorm{\alpha}{\phi}\abs{x-y}^{\alpha}  . 
	\end{equation*}
\end{lemma}

\begin{proof}
The inequality clearly holds if $x=y$, while if
$x\neq y$
     then the
     $\alpha$-H\"older assumption, 
     together with the
     intermediate value theorem, implies that there exists some $\xi_i$ in between $x$
     and $y$ for each $0\le i\le n-1$ such that
     \begin{equation*}
        \abs{S_{n,\mathbf{a}}\phi(x)-S_{n,\mathbf{a}}\phi(y)}
		\leq\Hseminorm{\alpha}{\phi}\sum\limits_{i=0}^{n-1}\abs{\Gau_{\sigma^i(\mathbf{a})}(x)-\Gau_{\sigma^i(\mathbf{a})}(y)}^\alpha
		=\Hseminorm{\alpha}{\phi}\sum\limits_{i=0}^{n-1}\Absbig{\Gau'_{\sigma^i(\mathbf{a})}(\xi_i)}^\alpha\abs{x-y}^\alpha ,
    \end{equation*}
     and the result follows readily from Proposition~\ref{p_inverse_branches}~(i) and the definition of $K_\alpha$ (see (\ref{e_def_K_alpha}). 
\end{proof}

\subsection{Symbolic dynamics}

\begin{notation}
Define the \emph{$\N$-valued full shift} $\Sigma$ by
\begin{equation*}
    \Sigma\=\N^\N,
\end{equation*}
and as usual define 
the shift map 
$\sigma \: \Sigma \to \Sigma $ by $\sigma((a_n)_{n\in \N})\=(a_{n+1})_{n\in \N}$. 
\end{notation}

Recall that $\rho(a,b)\=\Absbig{\frac{1}{a}-\frac{1}{b}}$ for $a, \, b\in \N$, and $\hrho$ denotes the extension of $\rho$ to $\hbN$ given by defining $\hrho(a,\infty)\=\frac{1}{a}
$ for $a\in \N$,
and $ \hrho (\infty,\infty) \=0$ (cf.~Section~\ref{sct_Notaion}). 

Recalling that $\theta\=\frac{\sqrt{5}+1}{2}$, 
define the metric $d_\rho$ on $\Sigma$  by
\begin{equation*}\label{e_metric_d_rho}
	d_\rho(A,B)\=\sum\limits_{n\in \N}\theta^{-2n}\rho (a_n,b_n)   , 
\end{equation*}
where $A=(a_n)_{n\in \N}$, $B=(b_n)_{n\in \N}\in \Sigma$.

\begin{definition}[Compactification of $(\Sigma, d_\rho)$]
	Since the metric $d_{\rho}$ is totally bounded, its metric completion, denoted by $\hSigma$, is a compact metric space. In particular, $\hSigma$ is a compactification of $\Sigma$. The compactification $\hSigma$ can be described more explicitly. More precisely, $\hSigma$ can be identified with $\hbN^\N$ equipped with the extended metric $d_{\hrho}$, where
	\begin{equation}\label{e_def_hd_hrho}
		d_{\hrho}\bigl(\hA,\hB\bigr)\=\sum\limits_{n\in \N}\theta^{-2n}\hrho \bigl(\ha_n,\hb_n\bigr)   , 
	\end{equation}
	for $\hA=(\ha_n)_{n\in \N}$, $\hB=\bigl(\hb_n\bigr)_{n\in \N}\in \widehat{\N}^\N$.

    The shift map $\sigma$ extends to a continuous self-map $\sigma:\hSigma\to\hSigma$ given by $\sigma((\ha_n)_{n\in \N})\=(\ha_{n+1})_{n\in \N}$. 
	
	For each $n\in \N$, define the \emph{cylinder set} $\cC(\ha_1,\ha_2,\dots,\ha_n)$ to consist of 
    those sequences $\bigl(\hb_i\bigr)_{i\in \N}\in \hSigma$ such that $\hb_1\hb_2\dots\hb_n=\ha_1\ha_2\dots\ha_n$. 
\end{definition}

\begin{notation}
Define the
homeomorphism  (cf.~\cite[Theorem~1.1]{Mi17})
$\pi\: \Sigma \to I\smallsetminus\Q$ by
\begin{equation*}\label{e_def_pi}
	\pi((a_i)_{i\in \N})\=[a_1,a_2,\dots]  , 
\end{equation*} 
and note that
\begin{equation}\label{conjugacyequation}
\pi\circ \sigma= \Gau\circ \pi.    
\end{equation}
Let \begin{equation*}
    \hpi\: \hSigma\to I
\end{equation*} denote the continuous extension of $\pi$ to $\hSigma$.
\end{notation}

\begin{rem}\label{r_hpi_equi_def}
It is readily checked that $\hpi$ satisfies
\begin{equation}\label{e_def_hpi}
	\hpi(\hA)=\begin{cases}
		[\ha_1,\ha_2,\dots,\ha_n,\dots]  & \text{if } \hA\in \Sigma,\,\text{i.e.,~}\iota(\hA)=+\infty,\\
		0  &  \text{if }\hA\in \hSigma\smallsetminus\Sigma\text{ with }
        \iota(\hA)=1,\\
		\bigl[\ha_1,\ha_2,\dots,\ha_{\iota(\hA)-1}\bigr] &  \text{if } \hA\in \hSigma\smallsetminus\Sigma\text{ with }
      2 \leq \iota(\hA)< +\infty   , 
	\end{cases}
\end{equation}
where $\hA=(\ha_i)_{i\in \N}\in \hSigma$.
\end{rem}

We will need the following lemma:
\begin{lemma}\label{l_est_rho_(a,b)}
	If $a,\, b\in \hbN$ with $a\neq b$,
    then $\abs{x-y}< 2\rho(a,b)$ for all $x\in I_a$, $y\in I_b$.
\end{lemma}

\begin{proof}
	Without loss of generality, assume that $a<b$. If $b=\infty$, then $\abs{x-y} =\frac{1}{x}\leq \frac{1}{a}<2\rho(a,b)$.
	
	If $b\neq \infty$, then $a+1\leq b$. Since $x\in \bigl[\frac{1}{a+1},\frac{1}{a}\bigr]$ and $y\in \bigl[\frac{1}{b+1},\frac{1}{b}\bigr]$, we get
	$\abs{x-y} = x-y \leq \frac{1}{a}-\frac{1}{b+1}<\frac{2}{a}-\frac{2}{b}=2\rho(a,b)$,
	where the second inequality, which is equivalent to $\frac{1}{a}>\frac{2}{b}-\frac{1}{b+1}=\frac{b+2}{b(b+1)}$, follows from $\frac{1}{a}\geq \frac{1}{b-1}$ and $\frac{1}{b-1}>\frac{b+2}{b(b+1)}$.
\end{proof}

The following Lemma~\ref{l_identification_2} allows us to abuse notation by identifying the two sets $\MMM(\Sigma,\sigma_{\Sigma})$ and $\bigl\{\mu\in\MMM\bigl(\hSigma,\sigma_{\hSigma}\bigr):\mu(\Sigma)=1 \bigr\}$.
Let $k \: \Sigma \to \hSigma$ be the inclusion map and $k_*\: \MMM(\Sigma,\sigma_{\Sigma})\to \MMM\bigl(\hSigma,\sigma_{\hSigma}\bigr)$ the pushforward of $k$, i.e., for each Borel subset $W\subseteq \hSigma$ and $\mu\in \MMM(\Sigma,\sigma_{\Sigma})$, define
\begin{equation}\label{e_def_k_*}
	k_*(\mu)(W)\=\mu\bigl(k^{-1}(W)\bigr)=\mu (A\cap\Sigma)  . 
\end{equation} 

\begin{lemma}\label{l_identification_2}
	The map $k_*$ is a continuous bijection from $\MMM(\Sigma,\sigma_{\Sigma})$ to $\bigl\{\mu\in\MMM\bigl(\hSigma,\sigma_{\hSigma}\bigr):\mu(\Sigma)=1 \bigr\}$.
\end{lemma}

\begin{proof}
	Since $k$ is continuous, $k_*$ is well-defined and continuous.
	If $\mu_1,\,\mu_2\in \MMM(\Sigma,\sigma_{\Sigma})$ 
    are such that $k_*(\mu_1)=k_*(\mu_2)$, 
	then by (\ref{e_def_k_*}), $\mu_1 (W\cap\Sigma)=\mu_2 (W\cap \Sigma)$ for every Borel subset $W\subseteq \hSigma$, 
	so $\mu_1=\mu_2$, therefore $k_*$ is injective.
	
	Fix $\mu\in\MMM\bigl(\hSigma,\sigma_{\hSigma}\bigr)$ such that $\mu(\Sigma)=1$.
	So $\mu (\cC(\infty))=0$.
	Define $\nu \in \cP(\Sigma)$ by $\nu(W)\=\mu(W)$ for each Borel subset $W\subseteq \Sigma$. By definition of $\nu$, since $\mu (\cC(\infty))=0$ and $\mu \in \MMM\bigl(\hSigma,\sigma_{\hSigma}\bigr)$, for each Borel subset $W\subseteq \Sigma$ we see that
	\begin{equation*}
		\nu\bigl( \sigma_{\Sigma}^{-1}(W)\bigr)
        =\nu \bigl( \sigma_{\hSigma}^{-1}(W)\smallsetminus\cC(\infty)\bigr)
        =\mu \bigl(\sigma_{\hSigma}^{-1}(W)\smallsetminus\cC(\infty)\bigr)
       =\mu \bigl(\sigma_{\hSigma}^{-1}(W)\bigr)
        =\mu(W)
        =\nu(W)  , 
	\end{equation*}
	hence $\nu\in\MMM(\Sigma,\sigma_{\Sigma})$. For each Borel subset $V\subseteq \hSigma$, by (\ref{e_def_k_*}), we have $k_*(\nu)(V)=\nu (V\cap \Sigma)=\mu (V\cap \Sigma)=\mu (V)$. So $k_*(\nu)=\mu$. Therefore, $k_*$ is a bijection from $\MMM(\Sigma,\sigma_{\Sigma})$ to $\bigl\{\mu\in\MMM \bigl( \hSigma,\sigma_{\hSigma} \bigr) : \mu(\Sigma)=1 \bigl\}$, as required. 
\end{proof}

The following lemma collects some basic properties of $\bigl(\hSigma,d_{\hrho}\bigr)$.

\begin{lemma}\label{l_properties_sigma_and_Sigma}
    The following statements are true:
	\begin{enumerate}[label=\rm{(\roman*)}]
		\smallskip
		\item\label{l_properties_sigma_Lipschitz} The map $\sigma \: \bigl(\hSigma,d_{\hrho}\bigr)\to \bigl(\hSigma,d_{\hrho}\bigr)$ is Lipschitz.
	    \smallskip
	    
	    \item\label{l_properties_sigma_density_invariant_measures} $\MMM(\Sigma,\sigma_{\Sigma})$ is a dense subset of $\MMM\bigl(\hSigma,\sigma_{\hSigma}\bigr)$, and $\MMM_{\operatorname{erg}}\bigl(\hSigma,\sigma_{\hSigma}\bigl)\smallsetminus\MMM(\Sigma,\sigma_{\Sigma}) $ is a dense subset of $\MMM\bigl(\hSigma,\sigma_{\hSigma}\bigr)$.
	    \smallskip
	    
	    \item\label{l_properties_sigma_decomposition_measure_1} If $\mu\in \MMM_{\erg} \bigl(\hSigma,\sigma_{\hSigma}\bigr)$, then either $\mu(\Sigma)=1$ or $\mu(\Sigma)=0$.
	    \smallskip
	    
	    \item\label{l_properties_sigma_decomposition_measure_2} $ \MMM\bigl(\hSigma,\sigma_{\hSigma}\bigr)$ is equal to the convex hull of 
        $\MMM(\Sigma,\sigma_{\Sigma}) \cup \bigl\{\mu\in\MMM\bigl(\hSigma,\sigma_{\hSigma}\bigr):\mu(\Sigma)=0 \bigr\}$.
	\end{enumerate}
\end{lemma}
\begin{proof}
	Suppose $\hA=(\ha_n)_{n\in \N}\in \hSigma$ and $\hB=\bigl(\hb_n\bigr)_{n\in \N}\in \hSigma$. 
    
	\smallskip
	
	\ref{l_properties_sigma_Lipschitz} It is immediate from (\ref{e_def_hd_hrho}) that
	\begin{equation*}\label{e_Lip_constant_sigma}
		d_{\hrho} \bigl(\hA,\hB\bigr)=\frac{\hrho(\ha_1,\hb_1)}{\theta^2}+\sum\limits_{n=2}^{+\infty}\frac{\hrho(\ha_n,\hb_n)}{\theta^{2n}}=\frac{\hrho(\ha_1,\hb_1)}{\theta^2}+\frac{d_{\hrho} (\sigma(\hA),\sigma (\hB))}{\theta^{2}} \geq \frac{ d_{\hrho} (\sigma(\hA),\sigma (\hB))}{\theta^{2}}  , 
	\end{equation*}
	so $d_{\hrho} \bigl(\sigma\bigl(\hA\bigr),\sigma \bigl(\hB\bigr)\bigr)\leq \theta^2 d_{\hrho} \bigl(\hA,\hB\bigr)$, and statement~\ref{l_properties_sigma_Lipschitz} follows.
    
	\smallskip
		
	\ref{l_properties_sigma_density_invariant_measures}	This follows from \cite[Theorem~1.1]{IV25}.
    
	\smallskip
    
	\ref{l_properties_sigma_decomposition_measure_1} Clearly, $\Sigma\triangle \sigma^{-1}(\Sigma)=\sigma^{-1}(\Sigma)\cap \cC(\infty)=\sigma^{-1}(\Sigma)\smallsetminus \Sigma$. 
	So $\mu\bigl(\Sigma\triangle \sigma^{-1}(\Sigma)\bigr)=\mu\bigl(\sigma^{-1}(\Sigma)\bigr)-\mu(\Sigma)=0$. But $\mu$ is ergodic, so 
    statement~\ref{l_properties_sigma_density_invariant_measures} follows (cf.~\cite[Theorem~1.5~(ii)]{Wa82}).
    
	\smallskip
    
	\ref{l_properties_sigma_decomposition_measure_2} Assume that $\mu \in \MMM\bigl(\hSigma,\sigma_{\hSigma}\bigr)$. By the ergodic decomposition theorem (see e.g.~\cite[Theorem~4.8]{EW11}), writing $\widehat{\cM}_1\=\MMM_{\erg}\bigl(\hSigma,\sigma_{\hSigma}\bigr)\cap \MMM(\Sigma,\sigma_{\Sigma})$ and $\widehat{\cM}_2\=\MMM_{\erg}\bigl(\hSigma,\sigma_{\hSigma}\bigr)\smallsetminus \MMM(\Sigma,\sigma_{\Sigma})$,
	\begin{equation*}\label{e_ergodic_decompostion_hSigma}
		\mu =\int_{\MMM_{\erg}(\hSigma,\sigma)}\! m \,\dd \alpha_\mu (m)=\int_{\widehat{\cM}_1}\! m \,\dd \alpha_\mu (m)+\int_{\widehat{\cM}_2}\! m \,\dd \alpha_\mu (m)  ,  
	\end{equation*}
	for some probability measure $\alpha_\mu$ on $\MMM_{\erg}\bigl(\hSigma,\sigma\bigr)$.
	Define
	$p\=\mu(\Sigma)$, $\mu_1\= \int_{\widehat{\cM}_1}\! m \,\dd \alpha_\mu (m)$, and $\mu_2\=\int_{\widehat{\cM}_2}\! m \,\dd \alpha_\mu (m)$.  
	Then by statement~\ref{l_properties_sigma_density_invariant_measures}, we conclude that
    \begin{enumerate}
    	\smallskip
        \item[(a)] when $p=0$, $\mu=\mu_2\in \bigl\{\mu\in\MMM\bigl(\hSigma,\sigma_{\hSigma}\bigr):\mu(\Sigma)=0 \bigr\}$,
	    \smallskip    
    	\item[(b)] when $p=1$, $\mu=\mu_1 \in \MMM(\Sigma,\sigma_{\Sigma})$,
	    \smallskip    
    	\item[(c)] when $p\in (0,1)$, $\mu = \mu_1+\mu_2$, with $\frac{\mu_1}{p}\in \MMM(\Sigma,\sigma_{\Sigma})$ and $\frac{\mu_2}{1-p}\in \mu\in \bigl\{\mu\in\MMM\bigl(\hSigma,\sigma_{\hSigma}\bigr):\mu(\Sigma)=0 \bigr\}$.
    \end{enumerate}
	 Statement~\ref{l_properties_sigma_decomposition_measure_2} now follows.
	\end{proof}

\begin{rem}
	The map $\sigma\: \bigl(\hSigma,d_{\hrho}\bigr)\to \bigl(\hSigma,d_{\hrho}\bigr)$ is not expansive: to see this note, for example, that if $A_n\=\overline{n}$ and $B\=\overline{\infty} $
    for $n\in \N$, then $d_{\hrho}\bigl(\sigma^k(A_n), \sigma^k(B)\bigr)=\sum_{i=1}^{+\infty}\frac{1}{n\theta^{2i}}=\frac{1}{(\theta^2-1)n}$ for each $k\in \N$.
\end{rem}

\subsection{Bounded continued fractions}\label{subsec_bounded_continued_fractions}

Recall the following notion from Diophantine approximation (see e.g.~\cite[Chapter~1]{Sc80}):
\begin{definition}[Badly approximable number]
	An irrational number $x$ is \emph{badly approximable} if there
	is a constant $c=c(x) >0$ such that
	$\Absbig{x-\frac{p}{q}} >\frac{c}{q^2}$ 
	for every rational number $\frac{p}{q}$.
\end{definition}

An irrational number $x\in I$ is badly approximable if and only if the partial
quotients in its continued fraction expansion are bounded
(see e.g.~\cite[Theorem~5F]{Sc80}).

Here we recall some properties for the set of bounded continued fractions.
For a nonempty subset $A\subseteq \N$,  let $\bcf_A$ denote the set of all irrational $x\in(0, 1)$ such that the digits $a_1(x)$, $a_2(x)$, $\ldots$ in the continued fraction expansion $x=[a_1(x),a_2(x),a_3(x),\dots]$ all belong to $A$.  

If $\card A<+\infty$, sets of the form $\bcf_A$ are said to be of bounded type and in particular they are Cantor sets. Of particular interest have been the sets $\bcf_n \= \bcf_{\{1,\,\dots,\,n\}}$.
The restriction $\Gau|_{\bcf_n}:\bcf_n \to \bcf_n$ is conjugate to the one-sided full shift on a finite alphabet $\{1,\,2,\,\dots,\,n\}$.

We recall the following notion (see e.g.~\cite[Chapter~4]{PU10}):
\begin{definition}[Distance-expanding map]
	For $ (X,d )$ a compact metric space, $T\:X\to X$ is called a \emph{distance-expanding map} if there exist constants $\lambda>1$ and $\eta>0$ such that for all $x, \, y \in X$ with $d (x,y )<2\eta$, 
	\begin{equation*}
		d (T (x ),T (y ) )\ge \lambda d (x,y )  . 
	\end{equation*}
\end{definition}

\begin{definition}\label{def_restricted_maximum_ergodic_average}
	Suppose $n\in \N$. For $\psi\in C(I)$, define the corresponding \emph{restricted ergodic supremum}
    $\mea_{n}(\Gau,\psi)$ by
	\begin{equation*}\label{e_def_restricted_mea}
		\mea_{n}(\Gau,\psi)\=\mea(\Gau|_{E_n},\psi|_{E_n})=\sup \{ \langle \mu,  \psi \rangle :\mu \in \MMM(I,\Gau),\,\supp \mu \subseteq E_n \}.
	\end{equation*}
\end{definition}
The following lemma collects some basic properties of the sets $\bcf_m$.
\begin{lemma}\label{l_basic_property_bounded_continued_fractions}
	Suppose $m\in \N$. If we write 
	\begin{align*}
		P& \=\bigl\{\Gau_{\ha}(0): \ha\in \hbN \bigr\}=\{0,\,1,\,1/2,\,\dots,\,1/n,\,\dots\}   ,   \\
		\eta_m&\=(m+2)^{-3}\big/2\in (0,1),\quad\lambda_m\=(1-\eta_m)^{-2}>1  ,   
	\end{align*}
	and denote the closed $\eta_m$-neighbourhood of $E_m$ by
	\begin{equation}\label{e_def_set_F_m}
		F_m\=\overline{B}^{\eta_m}_d(\bcf_m)
        =\{x\in I:d(x,\bcf_m)\leq \eta_m\}      , 
	\end{equation}
then the following hold:
	\begin{enumerate}[label=\rm{(\roman*)}]
		\smallskip
		\item $d(\bcf_m,P)>2\eta_m$.
		\smallskip
		\item $\Gau|_{F_m}$ is Lipschitz, and in particular $\Gau|_{\bcf_m}$ is Lipschitz.
		\smallskip
		
		\item If $x,\, y\in F_m$ with $\abs{x-y}<\eta_m$, then $\abs{\Gau(x)-\Gau(y)}\geq \lambda_m\abs{x-y}$, so $\Gau|_{\bcf_m}$ is distance-expanding.
		
		\item $\Gau|_{\bcf_m}$ is an open map.
	\end{enumerate}
\end{lemma}

\begin{proof} 
	(i) By Proposition~\ref{p_inverse_branches}~(iii),
	\begin{equation}\label{e_bound_of_bcf_m}
		\min \bcf_m= [\overline{m,1}]>[m,1]= 1/(m+1)\text{ and }
        \max \bcf_m= [\overline{1,m}]<[1,m+1]=(m+1)/(m+2)  . 
	\end{equation}
    But for each $1\le k\le m$, the map $\Gau_k$ is strictly decreasing  (see Proposition~\ref{p_inverse_branches}~(iii)),
   so (\ref{e_bound_of_bcf_m}) yields
	\begin{equation*}
		\Gau_k(E_m)\subseteq 
        (\Gau_k((m+1)/(m+2)),\Gau_k(1/(m+1)))
        \subseteq (1/(k+1),1/k)  , 
	\end{equation*}
and hence 
\begin{align*}
	d(P,\Gau_k(E_m))
    &>\min \{\Gau_k((m+1)/(m+2))-\Gau_k(1),\, \Gau_k(0)-\Gau_k(1/(m+1))\}\\
	&=\min \{\abs{\Gau_{k}'(z_1)} /(m+2),\,\abs{\Gau_{k}'(z_2)} /(m+1)\}
\end{align*}
    for some $z_1\in((m+1)/(m+2),1)$ and $z_2\in (0,1/(m+1))$ by the intermediate value theorem. 
   Since $\abs{\Gau_k'(x)} = (k+x)^{-2}\geq (m+1)^{-2}$ for all $x\in I$, we deduce that
   \begin{equation*}
	d(P,\Gau_k(E_m)) >\min \bigl\{ (m+1)^{-2}(m+2)^{-1},\,(m+1)^{-3}\bigr\} >2\eta_m   . 
\end{equation*}

Now $\bcf_m=\bigcup_{k=1}^m \Gau_k(E_m)$
(an immediate consequence of the definition of $\bcf_m$), 
so (i) follows. 

\smallskip
	
(ii) Assume that $x, \, y\in F_m$.
 When $\abs{x-y}<\eta_m$, by (i), we get $x,\,y\in (1/(k+1),1/k)$ for some $1\le k\le m$. Then 
 \begin{equation*}
 	\abs{\Gau(x)-\Gau(y)}=\abs{k+1/x-k-1/y}= z^{-2}\abs{x-y}\leq (m+1)^2\abs{x-y}  ,  
 \end{equation*}
	for some $z$ between $x$ and $y$ by the intermediate value theorem, where the last inequality follows from $z\geq 1/(m+1)$.	
	When $\abs{x-y}\geq \eta_m$, we have $\abs{\Gau(x)-\Gau(y)}<1<\frac{1}{\eta_m}\abs{x-y}$, so (ii) follows.
    
	\smallskip
	
	(iii) Assume that $x,\, y\in F_m$ with $\abs{x-y}<\eta_m$. By (i),  
    we get $x,\, y\in (1/(k+1),1/k)$ for some $1\le k\le m$, and then
	\begin{equation*}
		\abs{\Gau(x)-\Gau(y)}
        =\abs{k+1/x-k-1/y}
        = w^{-2}\abs{x-y}
        \geq \lambda_m \abs{x-y} 
	\end{equation*}
for some $w$ in between $x$ and $y$ by the intermediate value theorem, where the inequality follows from the fact that $w<1-\eta_m$.

    \smallskip
    
	(iv) Evidently $\Gau|_{\bcf_m}=(\pi|_{\Sigma_m})^{-1}\circ \sigma_{\Sigma_m} \circ\pi|_{\Sigma_m}$, so the result follows from the fact that $\pi|_{\Sigma_m}$ is a homeomorphism and  $\sigma_{\Sigma_m}$ is open.
\end{proof}

\section{Maximizing and limit-maximizing measures}\label{sec_maximizing_measures}

Here we introduce the notion of \emph{limit-maximizing measure}, which will be useful for a dynamical system, such as $\Gau$, whose set of invariant measures is not weak$^*$ compact.

\begin{notation}\label{pushforwards}
The pushforward 
$\pi_*\:\MMM(\Sigma,\sigma) \to \MMM(I\smallsetminus\Q,\Gau|_{I\smallsetminus\Q})$
of $\pi$ is defined by 
\begin{equation*}
	\pi_*(\mu)(V)\= \mu \bigl(\pi^{-1}(V)\bigr),
\end{equation*}
for all Borel subsets $V\subseteq I\smallsetminus\Q$.
Similarly,
the pushforward 
$\hpi_*\:\MMM\bigl(\hSigma,\sigma\bigr) \to \cP(I)$
is defined by 
\begin{equation*}
	\hpi_*(\mu)(W)\= \mu \bigl(\hpi^{-1}(W)\bigr) 
\end{equation*}
for all Borel subsets $W\subseteq I$.
\end{notation}

\begin{rem}
Since $\pi=\hpi\circ k$, Lemma~\ref{l_identification_2} allows us to abuse notation by writing 
\begin{equation*}
    \hpi_*|_{\MMM(\Sigma,\sigma_{\Sigma})}=\pi_*.
\end{equation*}
\end{rem}
 
 \begin{definition}\label{d_limit_max_measure}
 	Let $T \: X\to X$ be a Borel measurable map on a compact metric space $X$. For a Borel measurable function $\psi \: X\to \R$, a probability measure $\mu$ is called a \emph{$(T,\psi)$-limit-maximizing measure}, or simply a \emph{$\psi$-limit-maximizing measure}, if it is a weak$^*$ accumulation point of $\MMM(X,T)$ and $\int \!\psi \,\mathrm{d}\mu=\mea(T,\psi)$. We denote the set of $(T,\psi)$-limit-maximizing measures by $\Mmax^*(T,\psi)$. 
 \end{definition}
 
 Clearly, $\MMM_{\max}(T,\psi)\subseteq \MMM_{\max}^*(T,\psi)$. The following lemma collects some basic properties of $\pi$ and $\hpi$.
\begin{lemma}\label{l_properties_pi_and_pi_*}
	If $n\in \N$, then the following hold:
		\begin{enumerate}[label=\rm{(\roman*)}]		
			\smallskip
			\item \label{l_propertiess_pi_and_pi_*_preimage} $\hpi^{-1}(0)=\cC(\infty)$, $\hpi^{-1}(1)=\cC(1,\infty)$, and for each $x=[a_1,\dots,a_n]\in R_n$,
            \begin{equation*}
                \hpi^{-1}(x)=\cC(a_1,\dots,a_n,\infty)\cup\cC(a_1,\dots,a_{n-1},a_n-1,1,\infty).
            \end{equation*}

			\smallskip
			
			\item\label{l_properties_sigma_and_pi_commutes} If $\hfa=(\ha_1,\ha_2,\dots ,\ha_n)\in \hbN^n$, then the equality $\Gau_{\hfa}\circ \hpi \circ \sigma^n=\hpi$ holds on
            $\cC(\ha_1,\dots,
			\ha_n)$.
			\smallskip
			
		    \item\label{l_properties_pi_Lipschitz} 	The map $\hpi\:  \bigl(\hSigma,d_{\hrho}\bigr)\to (I,d) $ is Lipschitz. 
		    \smallskip
		
		    \item\label{l_properties_pi_*} $\pi_*\: \MMM(\Sigma,\sigma)\to \MMM_{\irr}(I,\Gau)$ is a homeomorphism.
		    \smallskip
		
		    \item\label{l_properties_hpi_*} $\hpi_*\: \MMM\bigl(\hSigma,\sigma\bigr)\to \overline{\MMM(I,\Gau)}$ is a continuous surjection.
		    \smallskip
		
		    \item\label{l_properties_hpi_*_Birkhoff_sum} If $\hfa=(\ha_1,\dots,\ha_n)\in \hbN^n$, $\phi\in C(I)$, $\hA=(\ha_i)_{i\in \N}\in \cC(\ha_1,\dots,\ha_n)$, and $y_n\=\hpi\bigr(\sigma^n\bigl(\hA\bigr)\bigr)$, then $S_{n,\hfa}\phi(y_n)=S_n^{\sigma}(\phi\circ \hpi)\bigl(\hA\bigr)$.
		\end{enumerate}
\end{lemma}

\begin{proof}
	\ref{l_propertiess_pi_and_pi_*_preimage}
	 $\hpi^{-1}(0)=\cC(\infty)$ and $\hpi^{-1}(1)=\cC(1,\infty)$ follow from (\ref{e_def_hpi}),
     while the equality
	  \begin{equation*}
        \hpi^{-1}(x)=\cC(a_1,\dots,a_n,\infty)\cup\cC(a_1,\dots,a_{n-1},a_n-1,1,\infty)
    \end{equation*}
    follows from (\ref{e_def_hpi}) and Lemma~\ref{l_continued_fraction_of_rational_number}.
     
	\smallskip
	
	\ref{l_properties_sigma_and_pi_commutes} Fix $\hfa=(\ha_1,\ha_2,\dots ,\ha_n)\in \hbN^n$. Consider an arbitrary $\hA=\ha_1\ha_2\ldots \ha_n\ha_{n+1}\ldots \in \cC[\ha_1,\ha_2,\dots,\ha_n]$. Write $k\= \iota\bigl(\hA\bigr)$. If $\hA$ belongs to $\Sigma$, it follows immediately from the definitions of $\pi$, $\Gau_{\hfa}$, and $\sigma$ that 
		\begin{equation*}
			\Gau_{\hfa}\bigl(\pi\bigl(\sigma^n\bigl(\hA\bigr)\bigr)\bigr)=\Gau_{\widehat{\bf{a}}}([\ha_{n+1},\ha_{n+2},\dots])=[\ha_1,\ha_2,\dots] =\hpi \bigl(\hA\bigr)  . 
		\end{equation*}
		If on the other hand $\hA\in \hSigma\smallsetminus\Sigma$, 
        in the case that $k=1$, then by Proposition~\ref{p_inverse_branches}~(ii),
		\begin{equation*}
			\Gau_{\hfa}\bigl(\hpi\bigl(\sigma^n\bigl(\hA\bigr)\bigr)\bigr)=0 =\hpi \bigl(\hA\bigr)  , 
		\end{equation*}
	while in the case that $2\leq k\leq n$, then by Proposition~\ref{p_inverse_branches}~(ii) and (\ref{e_def_hpi}), 
	\begin{equation*}
		\Gau_{\hfa}\bigl(\hpi\bigl(\sigma^n\bigl(\hA\bigr)\bigr)\bigr)=[\ha_1,\ha_2,\dots ,\ha_{k-1}] =\hpi \bigl(\hA\bigr)  . 
	\end{equation*}
	In the case that $k\geq n+1$, we obtain $\hpi\bigl(\sigma^n\bigl(\hA\bigr)\bigr)=[\ha_n,\dots,\ha_{k-1}]$ and 
	\begin{equation*}
	     \Gau_{\hfa}\bigl(\hpi\bigl(\sigma^n\bigl(\hA\bigr)\bigr)\bigr)=[\ha_1,\ha_2,\dots,\ha_n,\dots ,\ha_{k-1}] =\hpi \bigl(\hA\bigr)  . 
	\end{equation*}

	\smallskip
	
	\ref{l_properties_pi_Lipschitz} Assume that $\hB=\bigl(\hb_i\bigr)_{i\in \N}$, $\hC=(\hc_i)_{i\in \N}\in \hSigma$. 
	Let $k$ be the smallest integer such that $\hb_k \neq \hc_k$. If $k=1$, since $\hpi\bigl(\hB\bigr)\in I_{\hb_1}$, $\hpi\bigl(\hC\bigr)\in I_{\hc_1}$,  Lemma~\ref{l_est_rho_(a,b)}, together with the definition of $d_{\hrho}$ 
    (cf.~(\ref{e_def_hd_hrho})), implies that
	\begin{equation*}
		\Absbig {\hpi\bigl(\hB\bigr)-\hpi\bigl(\hC\bigr)}
        \leq 2\rho\bigl(\hb_1, \hc_1\bigr)\leq 2\theta^2 d_{\hrho} \bigl(\hB,\hC\bigr)  . 
	\end{equation*} 
	If $k\geq 2$, the definition of $\hpi$ guarantees that $\hpi \bigl(\hB\bigr),\, \hpi \bigl(\hC\bigr)\in I_{\hfa}$ for some $\hfa\=\ha_1\ha_2\dots\ha_{k-1}\in \hbN^{k-1}$. 
	Combining statement~\ref{l_properties_sigma_and_pi_commutes}, Proposition~\ref{p_inverse_branches}~(i), Lemma~\ref{l_est_rho_(a,b)}, (\ref{e_def_hd_hrho}), and the fact that $\hpi\bigl(\sigma^{k-1}\bigl(\hB\bigr)\bigr)\in I_{\hb_k}$ and $\hpi\bigl(\sigma^{k-1}\bigl(\hC\bigr)\bigr)\in I_{\hc_k}$, 
    we have
	\begin{align*}
		\Absbig{\hpi\bigl(\hB\bigr)-\hpi\bigl(\hC\bigr)}
        &=\Absbig{\Gau_{\hfa}\bigl(\hpi\bigl(\sigma^{k-1}\bigl(\hB\bigr)\bigr)\bigr) - \Gau_{\hfa}\bigl(\hpi\bigl(\sigma^{k-1}\bigl(\hC\bigr)\bigr)\bigr)} \\
		&\leq \frac{\abs{ \hpi (\sigma^{k-1}(\hB\bigr))-\hpi (\sigma^{k-1}(\hC))}}{c_0^{2}\theta^{2(k-1)}} 
		\leq \frac{2\rho(\hb_k,\hc_k)}{c_0^{2}\theta^{2(k-1)}} 
		  \leq \frac{2\theta^2d_{\rho}(A,B)}{c_0^{2}}  , 
	\end{align*} 
so statement~\ref{l_properties_pi_Lipschitz} follows.

\smallskip

\ref{l_properties_pi_*} Since $\pi \:\Sigma  \to I\smallsetminus\Q$ is a homeomorphism (cf.~\cite[Theorem~1.1]{Mi17}), the relation
(\ref{conjugacyequation}) implies that $\sigma\:\Sigma\to\Sigma$ and $\Gau \: I\smallsetminus\Q\to I\smallsetminus\Q$ are topologically conjugate, so their spaces of invariant probability measures are homeomorphic under $\pi_*$.

\smallskip

\ref{l_properties_hpi_*} Fix $\mu \in \overline{\MMM(I,\Gau)}$. By Lemma~\ref{l_Q_Q^c_invatiant}, there exists $\{\mu_n\}_{n\in \N}\subseteq \MMM_{\irr}(I,\Gau)$ such that $\mu_n$ converges to $\mu$ in the weak$^*$ topology as $n\to +\infty$. Let us denote $\nu_n\=\pi_*^{-1}(\mu_n)$ for each $n\in \N$. By the weak$^*$ compactness of $\MMM\bigl(\hSigma,\sigma_{\hSigma}\bigr)$, there exists an accumulation point $\nu$ of $\{\nu_n\}_{n\in \N}$. Since $\hpi_*$ is continuous (see statement~\ref{l_properties_pi_Lipschitz}), $\hpi_*(\nu)=\mu$. Consequently $\hpi_*$ is surjective.

\smallskip

\ref{l_properties_hpi_*_Birkhoff_sum} From the definition of $S_{n,\hfa}$ (see (\ref{e_S_n_bf_a})), it suffices to note that $\Gau_{\sigma^i(\hfa)}(y_n)=\hpi \bigl(\sigma^i(A)\bigr)$ for each $0\le i\le n-1$, by statement~\ref{l_properties_sigma_and_pi_commutes}. \end{proof}

\begin{rem}
   \begin{enumerate}[label=\rm{(\roman*)}]	
	\item The continuous map 
    $\pi^{-1} \: I\smallsetminus\Q\to\Sigma$ is not Lipschitz: for example if $x_n\=[3,\overline{n}]$ and $y_n \=[2,1,\overline{n}]$, then $\lim\limits_{n\to +\infty}\abs{x_n- y_n}=0$ but $d_{\hrho}\bigl(\pi^{-1}(x_n),\pi{}^{-1}(y_n)\bigr)\geq \theta^{-2}\hrho(3,2)=\frac{1}{6\theta^2}$.
	
    \smallskip
	
	\item The pushforward $\hpi_*$ of the extension $\hpi$ is not injective: for example let $\mu_1$ be the periodic measure supported on the periodic orbit of $\overline{1\infty2\infty}$, let $\mu_2$ be the periodic measure supported on the periodic orbit of $\overline{1\infty}$, and let $\mu_3$ be the periodic measure supported on the periodic orbit of $\overline{2\infty}$. Then by the definition of $\hpi_*$, we get $\hpi_*(\mu_1)=\frac{1}{4}\bigl(2\delta_0+\delta_{1}+\delta_{1/2}\bigr)$, $\hpi_*(\mu_2)=\frac{1}{2}(\delta_0+\delta_{1})$, and $\hpi_*(\mu_3)=\frac{1}{2}\bigl(\delta_0+\delta_{1/2} \bigr)$. Clearly, $\hpi_*(\mu_1)=\hpi_*\bigl(\frac{1}{2}(\mu_2+\mu_3)\bigr)$ but $\mu_1\neq \frac{1}{2}(\mu_2+\mu_3)$.
    
	\smallskip
	
	\item The dynamics of $\Gau$ and $\sigma$ are not intertwined by $\hpi$, in other words, $\hpi\circ \sigma \neq \Gau\circ \hpi$: to see this note, for example, that if $A=\infty \overline{1}$ then $\hpi (\sigma(A))=[\overline{1}]$, but $\Gau(\hpi(A))=0$. 
\end{enumerate}
\end{rem}

\begin{prop}\label{p_ergodic_op_relations}
	If $\phi \in C(I)$, then 	
    \begin{enumerate}[label=\rm{(\roman*)}]	
	\smallskip
	\item\label{p_ergodic_op_relation_mea_equal} $\mea\bigl(\sigma_{\hSigma}, \phi\circ \hpi\bigr)=\mea (\Gau,\phi)$ and
	\smallskip
	
	\item\label{p_ergodic_op_relation_maximizing_measures} $\hpi_*\bigl(\MMM_{\max}\bigl(\sigma_{\hSigma}, \phi\circ \hpi\bigr)\bigr)
    =\MMM^*_{\max} (\Gau,\phi)\neq \emptyset$.
	\end{enumerate}	
\end{prop}

\begin{proof}
	\ref{p_ergodic_op_relation_mea_equal} From the definition of $\mea\bigl(\sigma_{\hSigma}, \phi\circ \hpi \bigr)$, and the fact that $\MMM(\Sigma,\sigma_{\Sigma})$ is dense in $\MMM\bigl(\hSigma,\sigma_{\hSigma} \bigr)$, we see that
	\begin{equation*}
	     \mea\bigl(\sigma_{\hSigma}, \phi\circ \hpi\bigr)
         =\sup\{ \langle  \mu, \phi\circ \hpi \rangle: \mu\in \MMM(\Sigma,\sigma_{\Sigma})\}
         =\sup\{ \langle  \mu, \phi\circ \pi  \rangle: \mu\in \MMM(\Sigma,\sigma_{\Sigma})\}. 
	\end{equation*}
Since $\MMM_{\irr}(I,\Gau)$ is dense in $\MMM(I,\Gau)$ (see Lemma~\ref{l_Q_Q^c_invatiant}), we have
	\begin{equation*}
	\mea(\Gau, \phi)=\sup\{\langle\nu , \phi\ \rangle: \nu\in \MMM_{\irr}(I,\Gau)\}  . 
\end{equation*}
Combining the above two identities with Lemma~\ref{l_properties_pi_and_pi_*}~\ref{l_properties_pi_*}, 
\ref{p_ergodic_op_relation_mea_equal} follows.

\smallskip

\ref{p_ergodic_op_relation_maximizing_measures} The first identity comes from the fact that $\hpi_*$ is a surjection from $\MMM\bigl(\hSigma,\sigma_{\hSigma}\bigr)$ to $\overline{\MMM(I,\Gau)}$ (cf.~Lemma~\ref{l_properties_pi_and_pi_*}~\ref{l_properties_hpi_*}),
and $\MMM^*_{\max} (\Gau,\phi)\neq \emptyset$ follows from the fact that $\overline{\MMM(I,\Gau)}$ is compact with respect to the weak$^*$ topology, and the assumption that $\phi\in C(I)$.
\end{proof}

\begin{prop}\label{p_mea_equal_sup_time_average}
	If $\phi\in C(I)$, then 
		\begin{enumerate}[label=\rm{(\roman*)}]	
        \smallskip
		\item $\lim\limits_{m\to +\infty}\mea_m(\Gau,\phi)=\mea(\Gau,\phi)$ and
		\smallskip		
		\item $\mea(\Gau,\phi)=\sup \bigl\{ \liminf\limits_{n\to +\infty}\frac{S_n\phi(x)}{n} : x\in I\smallsetminus\Q\bigr\}$.
	\end{enumerate} 
\end{prop}

\begin{proof}
	(i) The set of periodic measures is known to be dense in $\MMM(\Sigma,\sigma)$ (see \cite[Theorem~3.8]{IV21}), and evidently $\pi^*$ gives a one-to-one correspondence between the set of periodic measures in $\MMM(\Sigma,\sigma)$ and the set of periodic measures in $\MMM_{\irr}(I,\Gau)$. 
	Combining this with Lemma~\ref{l_properties_pi_and_pi_*}~\ref{l_properties_pi_*}, it follows that the set of periodic measures is dense in $\MMM_{\irr}(I,\Gau)$,
    so for any $\epsilon>0$ there is a periodic measure $\mu\in \MMM_{\irr}(I,\Gau)$ with $\int\!\phi \, \mathrm{d}\mu\geq \mea (\Gau,\phi)-\epsilon$. Since $\mu$ is periodic, there exists $m\in \N$ such that $\mu \in \MMM(\bcf_m,\Gau|_{\bcf_m})$,
    and clearly $\mea_m(\Gau,\phi)\geq \int\!\phi \, \mathrm{d}\mu\geq \mea (\Gau,\phi)-\epsilon$. But $\epsilon>0$ was arbitrary, and $\{\mea_m(\Gau,\phi)\}_{m\in \N}$ is nondecreasing and bounded above by $\mea(\Gau,\phi)$, so (i) follows.
    
	\smallskip
	
	(ii) Now $\mea\bigl(\sigma_{\hSigma}, \phi\circ \hpi \bigr)=\mea (\Gau,\phi)$ by Proposition~\ref{p_ergodic_op_relations}~(i), the space $\bigl(\hSigma, d_{\hrho} \bigr)$ is compact, $\sigma:\hSigma\to\hSigma$ is continuous, and $\pi \:\Sigma  \to I\smallsetminus\Q$ is a homeomorphism (cf.~\cite[Theorem~1.1]{Mi17}), so \cite[Proposition~2.2]{Je19} gives
	\begin{align*}
		\mea(\Gau,\phi)
        &=\mea\bigl(\sigma_{\hSigma}, \phi\circ \hpi\bigr)
        = \sup \bigl\{ \liminf_{n\to +\infty}  n^{-1} S^{\sigma}_n(\phi\circ \hpi) (A)  : A\in\hSigma \bigr\} \\
		&\geq \sup \bigl\{ \liminf_{n\to +\infty} n^{-1} S^{\sigma}_n(\phi\circ \pi) (A) : A\in\Sigma \bigr\}
		= \sup \bigl\{  \liminf_{n\to +\infty} n^{-1} S_n\phi(x) : x\in I\smallsetminus\Q \bigr\} .  
	\end{align*}
For each $m\in\N$, the restriction $\Gau|_{\bcf_m}$ is a continuous map on the compact metric space $\bcf_m$,
so (i) and \cite[Proposition~2.2]{Je19} together give 	
\begin{equation*}
	\mea(\Gau,\phi)\leq \sup\limits_{ x\in \bigcup_{m\in \N}\bcf_m} \liminf_{n\to +\infty}\frac{S_n\phi(x)}{n}\leq \sup\limits_{ x\in I\smallsetminus\Q} \liminf_{n\to +\infty}\frac{S_n\phi(x)}{n}  , 
\end{equation*}
and (ii) follows.
\end{proof}

\section{\texorpdfstring{Structure of the closure of $\MMM(I,\Gau)$}{Structure of the closure of M(I,G)}}\label{sct_closure_invariant_measures}

\begin{definition}[Finite-continued-fraction measures and rational orbits]\label{d_rational_orbit_and_rational_measure}

	Fix $n\in\N$ and $\fa=(a_1,\dots,a_n)\in \N^n$. Let us denote 
    \begin{equation}  \label{e_def_p_k}
        \begin{aligned}
        p_0&=p_0(\fa)\=0 \quad\text{ and} \\ 
        p_k&=p_k(\fa)\=[a_{n-k+1},\dots,a_n] \quad\text{ for each } 1\le k\le n.
        \end{aligned} 
    \end{equation}
    We define the corresponding \emph{rational orbit} of length 
    \begin{equation}\label{la_defn}
    l_{\fa}\=n+1
    \end{equation}
    to be
	\begin{equation}\label{e_def_rational_orbit}
		\cO_{\bf a}\=
        \{p_0(\fa),\,p_1(\fa),\,p_2(\fa),\,\dots,\,p_n(\fa)\}
       = \{0,\,p_1,\,p_2,\,\dots,\,p_n\}  . 
	\end{equation}
    Note that each rational orbit $\cO_{\bf a}$
    contains only rational numbers.

	The corresponding \emph{finite-continued-fraction measure} (abbreviated as \emph{FCF measure}) $\mu_{\bf a}$ (of \emph{length} $ l_{\fa}$) is defined by
	\begin{equation}\label{e_def_rational_measure}
		\mu_{\bf a}\=
        \frac{1}{l_{\fa}}\bigl(\delta_{p_0(\fa)}+\delta_{p_1(\fa)}+\cdots+\delta_{p_n(\fa)}\bigr)
        = \frac{1}{n+1}\bigl(\delta_0+\delta_{p_1}+\cdots+\delta_{p_n}\bigr)  . 
	\end{equation}
    Note that $\mu_{\bf a}$ is a probability measure, but is never $\Gau$-invariant.

	Define $\RM$ to be the convex hull of $\{\delta_0\} \cup \{\mu_{\bf{a}}: n\in\N, \,\fa\in \N^n\}$.
    Define
	\begin{equation}\label{e_RM_*}
		\RM_{[0,1)} \= \{\mu \in \RM: \mu (\{1\})=0\}\subseteq \RM   . 
	\end{equation}
\end{definition}

\begin{rem}\label{r_rational_orbit}
    Fix $n\in\N$ and $\fa\in \N^n$. When $\fa\in \cA^n$, the rational orbit $\cO_{\fa}=\cO(p_n(\fa))$ is a $\Gau$-orbit that is eventually fixed, in the sense that $\Gau(p_0(\fa))=p_0(\fa)$. 
    
    When $\fa\in \cB^n$, the rational orbit $\cO_{\fa}=\cO(p_n(\fa))\cup\{1\}$ is not a $\Gau$-orbit, but 
    is an orbit under the map that is equal to $1$ on $R_1$, and equal to $\Gau$ elsewhere.
\end{rem}

\begin{lemma}\label{l_rational_orbit/measure_close_to_periodic}
	Every FCF measure is the limit of a sequence of periodic measures. Moreover,
    \begin{equation}\label{rational_contained_closure}
    \RM\subseteq \overline{\MMM(I,\Gau)} =\overline{\MMM_{\irr}(I,\Gau)}  . 
    \end{equation}
\end{lemma}

\begin{proof}
	Fix $n\in \N$ and $\fa=(a_1,\dots,a_n)\in \N^n$.
    For each $m\in \N$, define 
	\begin{equation}
		r_m\=[\overline{a_1,\dots,a_n,m}]\in \Fix\bigl(\Gau^{n+1}\bigr)  . 
	\end{equation}
	Evidently $\lim\limits_{m\to +\infty}\Gau^{n-k}(r_m)=p_k$ for each $0\le k\le n$, hence the sequence $\{\mu_{\cO(r_m)}\}_{m\in\N}$ of periodic measures converges to the FCF measure $\mu_{\fa}$, as required.

	The sequence $\{\mu_{\cO(r_m)}\}_{m\in\N}$ is contained in $\MMM(I,\Gau)$, so $\mu_\mathbf{a}\in \overline{\MMM(I,\Gau)}$. But $\overline{\MMM(I,\Gau)}$ is convex, and $\fa\in\N^n$ was arbitrary, so
    $\RM\subseteq \overline{\MMM(I,\Gau)}$. By Lemma~\ref{l_Q_Q^c_invatiant}, we get $\overline{\MMM(I,\Gau)} =\overline{\MMM_{\irr}(I,\Gau)}$. Therefore, we obtain (\ref{rational_contained_closure}).
\end{proof}

The sets $R_n$, defined in Notation~\ref{n_Rn_g_f}, have the following simple properties:

\begin{lemma}\label{l_Rn_disjoint_union}
    $\bigcup_{n\in \N}R_n=(0,1)\cap \Q$, and if $m\neq n$ then $R_n\cap R_m=\emptyset$.
\end{lemma}

\begin{proof}
    By (\ref{e_def_sets_R_m}), $\bigcup_{n\in \N}R_n\subseteq (0,1)\cap\Q$. By Lemma~\ref{l_continued_fraction_of_rational_number},  $(0,1)\cap\Q\subseteq \bigcup_{n\in \N}R_n$ and $R_n\cap R_m = \emptyset$ when $n\neq m$.
\end{proof}

The following lemma collects some basic properties of measures in $\RM$.

\begin{lemma}\label{l_basic_properties_rational_measures}
	Suppose $\mu\in \RM$, $n\in\N$, $\fa=(a_1,\dots,a_n)\in \cA_n$, and $\fb=(b_1,\dots,b_n)\in \cB_n$.
	\begin{enumerate}[label=\rm{(\roman*)}]
		\smallskip
		
		\item Then $\mu(\{1\})\leq \frac{1}{2}$.
		\smallskip
		
		\item If $\mu \in \RM_{[0,1)}$, then $\mu(R_n)\leq\frac{1}{n+1}$.
		\smallskip
		
		\item $(l_{\fa}+1)\mu_{f(\fa)}=l_{\fa}\mu_{\fa}+\delta_1$.
		\smallskip
		
		\item The map $g_n \: \cA_n \to R_n$ is bijective.
		\smallskip	
		
		\item For every $y\in I$, we have $\mu_{\fa}(\{y\})= 1 / l_{\fa}$ if $y\in \cO(g_n(\fa))$, and $\mu_{\fa}(\{y\})= 0$ otherwise.
        \smallskip
        
        \item $\mu_\fa(\{1\})=0$ and $\mu_\fa(\{0\})= \mu_\fa (R_1)$.
        \smallskip

        \item For all $x\in (0,1)\cap \Q$, $\mu_\fa (\{x\})\geq \mu_\fa \bigl(\Gau^{-1}(x) \bigr)$.
        \smallskip

        \item $\mu_\fb(\{1\})=1/l_{\fb}$ and $\mu_{\fb}(\{0\})\geq \mu_{\fb}(R_1)$.
        \smallskip

        \item For all $x\in (0,1)\cap \Q$, $\mu_\fb (\{x\})\geq \mu_\fb \bigl(\Gau^{-1}(x) \bigr)$.
	\end{enumerate}
\end{lemma}

\begin{proof}
	(i) follows immediately from the fact that $\mu_{\fb} (\{1\})\leq \frac{1}{2}$ for all $\fb\in \N^*$ (see (\ref{e_def_rational_measure})), the fact that $\delta_0(\{1\})=0$, and the definition of $\RM$ (cf.~Definition~\ref{d_rational_orbit_and_rational_measure}).
    
	\smallskip
	
    (ii) This follows from the fact that $\mu_\fb (\{R_n\})\leq \frac{1}{n+1}$ for each $ \fb\in \cA$ (see (\ref{e_def_rational_measure}) and (\ref{e_def_sets_R_m})), the fact that $\delta_0(\{R_n\})=0$, and the definition of $\RM_{[0,1)}$ (see (\ref{e_RM_*})).
    
	\smallskip
	
	(iii) Since $[a_{n-k+1},\dots,a_n]=[a_{n-k+1},\dots,a_n-1,1]$ for each $1\le k\le n$, using (\ref{e_def_rational_measure}) and (\ref{e_def_p_k}) we get
	$(l_{\fa}+1)\mu_{f(\fa)}
        =\delta_0+\delta_1+\sum_{k=1}^n\delta_{p_k}
        =l_{\fa}\mu_{\fa}+\delta_1$.

	\smallskip
	
	(iv) By the definition of $\cA_n$ (see (\ref{e_def_cA_n})) and the definition of $g_n$ (see (\ref{e_def_g_n})), we have $\Gau^{n-1}(g_n(\fa))=\frac{1}{a_n}\in R_1$. By Lemma~\ref{l_continued_fraction_of_rational_number}, $g_n$ is injective. By the definition of $R_n$ (see (\ref{e_def_sets_R_m})), $g_n$ maps $\cA_n$ surjectively to $R_n$.
    
	\smallskip
	
	(v) For $\fa =(a_1,a_2,\dots,a_n)\in \cA_n$,  denote $x\= [a_1,a_2,\dots,a_n]=g_n(\fa)$. The definition of the rational orbit $\cO_{\fa}$ 
    (cf.~(\ref{e_def_rational_orbit})) gives that $\cO_{\fa}=\cO(x)$, and then (v) follows from the definition of $\mu_{\fa}$ 
    (cf.~(\ref{e_def_rational_measure})).
    
    \smallskip

    (vi) We have $\mu_{\fa}(\{1\})=0$ since the support of $\mu_{\fa}$ is $\cO_\fa$ contained in $[0,1)\cap\Q$ (cf.~(\ref{e_def_rational_orbit})). 
    The point $0$ is an atom of $\mu_{\fa}$, with $\mu_{\fa}(\{0\})=1 / l_{\fa}$, and precisely one element of $R_1$, namely the point $1/a_n$, is an atom of $\mu_{\fa}$, also with weight
    $\mu_{\fa}(\{1/a_n\})=1 / l_{\fa}$, so in particular $\mu_{\fa}(\{R_1\})=1 / l_{\fa} = \mu_{\fa}(\{0\})$.

    \smallskip

    (vii) Note that the support of $\mu_{\fa}$ is an eventually fixed $\Gau$-orbit, so if $x$ is not an atom of $\mu_{\fa}$ then nor is any element of $\Gau^{-1}(x)$, so $\mu(\{x\})=0=\mu\bigl(\Gau^{-1}(x)\bigr)$. If $x$ is an atom of $\mu_{\fa}$ then $x=p_k(\fa)$ for some $0\le k\le n$: if $k=n$ then $\Gau^{-1}(x)$ does not contain any atoms of $\mu_{\fa}$, so $\mu_\fa (\{x\}) = 1 / l_{\fa}> 0= \mu_\fa\bigl(\Gau^{-1}(x)\bigr)$,
    while if $k<n$ then $\Gau^{-1}(x)$ contains precisely one atom of $\mu_{\fa}$, namely $p_{k+1}(\fa)$, so
    $\mu_\fa (\{x\}) = 1 / l_{\fa}= \mu_\fa\bigl(\Gau^{-1}(x)\bigr)$, therefore in both cases we see that (vii) holds.
    \smallskip

    (viii) We have $\mu_{\fb}(\{1\})=1/l_\fb$ since the support of $\mu_{\fb}$ is $\cO_\fb$ containing $1$ (cf.~(\ref{e_def_rational_orbit})). When $n=1$, $\mu_{\fb}=1/2(\delta_0+\delta_1)$ and $\mu_\fb(\{0\})\geq 0=\mu_\fb(R_1)$. When $n\geq 2$, note that $\cO_{\fb}=\cO_{f^{-1}(\fb)}\cup\{1\}$ and $f^{-1}(\fb)\in \cA^{n-1}$, so by (iii) and (vi) we get $\mu_\fb(\{0\})\geq\mu_\fb(R_1)$.
    \smallskip

    (ix) Note that the support of $\mu_{\fb}$ is the union of an eventually fixed $\Gau$-orbit and $1$ (see Remark~\ref{r_rational_orbit}), so by the fact that $\Gau^{-1}(1)=\emptyset$, if $x$ is not an atom of $\mu_{\fb}$ then nor is any element of $\Gau^{-1}(x)$, so $\mu_{\fb}(\{x\})=0=\mu_{\fb}\bigl(\Gau^{-1}(x)\bigr)$. If $x\in (0,1)$ is an atom of $\mu_{\fb}$, then $x=p_k(\fb)$ for some $2\le k\le n$: if $k=n$ then $\Gau^{-1}(x)$ does not contain any atoms of $\mu_{\fb}$, so $\mu_\fb (\{x\}) = 1 / l_{\fb}> 0= \mu_\fa\bigl(\Gau^{-1}(x)\bigr)$,
    while if $k<n$ then $\Gau^{-1}(x)$ contains precisely one atom of $\mu_{\fb}$, namely $p_{k+1}(\fb)$, so
    $\mu_\fb (\{x\}) = 1 / l_{\fb}= \mu_\fb\bigl(\Gau^{-1}(x)\bigr)$, therefore in both cases we see that (ix) holds.
\end{proof}

\begin{lemma}\label{l_relation_RM_and_RM_*}
    If $\nu \in \RM_{[0,1)}$ and $r\in [0,1]$ satisfy $(1-r)\nu(\{0\})\geq r$, then $(1-r)\nu +r\delta_1\in \RM$.
\end{lemma}

\begin{proof}
    Since $\nu\in\RM_{[0,1)}$, 
    we can write
    \begin{equation}\label{e_decomposition_nu}
		\nu=r_0\delta_0+\sum\limits_{ \fa \in \cA}r(\fa)\mu_{\fa},
	\end{equation}
    where $r_0\ge0$, and $r(\fa)\ge 0$ for each $\fa\in \cA$, and 
	\begin{equation}
		r_0+\sum\limits_{ \fa \in \cA}r(\fa)=1.\label{e_decomposition_nu_constants}
\end{equation}
        
Combining (\ref{e_decomposition_nu}),
(\ref{la_defn}), and (\ref{e_def_rational_measure}), we see that
	\begin{equation*}
		\nu(\{0\})=
        r_0 + \sum\limits_{ \fa \in \cA} r(\fa) \mu_{\fa}(\{0\})=
        r_0+\sum\limits_{ \fa \in \cA}\frac{r(\fa)}{l_{\fa}},
	\end{equation*}
and combining this with the assumption that $ (1-r)\nu(\{0\})\geq r$
gives
\begin{equation}\label{e_mu0_equality}
r_0+\sum\limits_{ \fa \in \cA}\frac{r(\fa)}{l_{\fa}}
\ge \frac{r}{1-r}.
\end{equation}
    
	Let us denote 
	$\lambda\= \frac{r}{(1-r)\nu (\{0\})}\in[0,1]$.
    Then by (\ref{e_decomposition_nu}), (\ref{e_mu0_equality}), and  (\ref{e_decomposition_nu_constants}),
	\begin{equation*}\label{e_proof_decomposition1}
		\begin{aligned}
			\lambda(1-r)\nu+r\delta_1
			&=\frac{r}{\nu(\{0\})}\Bigl(r_0\delta_0+\sum\limits_{ \fa \in \cA}r(\fa)\mu_{\fa}\Bigr)+
			\frac{r}{\nu(\{0\})}\Bigl(r_0+\sum\limits_{ \fa \in \cA}\frac{r(\fa)}{l_{\fa}}\Bigr)\delta_1\\
			&=\frac{r}{\nu(\{0\})}\Bigl(r_0(\delta_0+\delta_1)+\sum\limits_{ \fa \in \cA}\frac{r(\fa)}{l_\fa}(l_\fa\mu_{\fa}+\delta_1)\Bigr)  . 
		\end{aligned}	
	\end{equation*}
	This, together with $l_\fa\mu_{\fa}+\delta_1 = (l_\fa+1)\mu_{f(\fa)}$ (see Lemma~\ref{l_basic_properties_rational_measures}~(iii)), gives us 
	\begin{align*}
		(1-r)\nu+r\delta_1&=(1-\lambda)(1-r)\nu+\lambda(1-r)\nu+r\delta_1\\
		&=(1-\lambda)(1-r)\nu+\frac{r}{\nu(\{0\})}\Bigl(r_0(\delta_0+\delta_1)+\sum\limits_{ \fa \in \cA}\frac{r(\fa)}{l_{\fa}}(l_\fa+1)\mu_{f(\fa)}\Bigr)  .
	\end{align*}
	Then since $\RM$ is convex, and each of the measures $\nu$, $\frac{\delta_0+\delta_1}{2}$, and $\mu_{f(\fa)}$
    belongs to $\RM$, 
    we conclude that $(1-r)\nu+r\delta_1\in \RM$, as required.
\end{proof}

\begin{prop}\label{l_RM*_equivalent_def}
	Suppose $\mu \in \cP(I)$ with $\mu (I\cap\Q)=1$. Then $\mu\in \RM_{[0,1)}$ if and only if $\mu$ satisfies the following conditions:
	\begin{enumerate}[label=\rm{(\roman*)}]
		\smallskip
		\item $\mu(\{1\})=0$ and $\mu(\{0\})\geq \mu (R_1)$.
		\smallskip
		
		\item $\mu (\{x\})\geq \mu\bigl(\Gau^{-1}(x)\bigr)$ for all $x\in (0,1)\cap \Q$. 
	\end{enumerate}
\end{prop}

\begin{proof}
	First we assume that $\mu\in \RM_{[0,1)}$, and will show that $\mu$ satisfies conditions~(i) and~(ii). Note that these properties are closed under convex combination.
    From the definition of $\RM_{[0,1)}$ (see (\ref{e_RM_*})), and the fact that $\delta_0$ satisfies conditions~(i) and~(ii),
    to prove that $\mu$ satisfies conditions~(i) and~(ii)
    it suffices to note that $\mu_{\fa}$ satisfies conditions~(i) and~(ii) for all $\fa\in \cA$, by~(vi) and~(vii) of
	Lemma~\ref{l_basic_properties_rational_measures}.

	Now we assume that $\mu$ satisfies conditions~(i) and~(ii), and will show that this implies that $\mu\in\RM_{[0,1)}$. 
	Define a function $\phi_\mu \: [0,1)\cap \Q\to [0,1]$ by 
	\begin{align}
		\phi_{\mu}(0)&\= \mu(\{0\})-\mu (R_1)  , \label{e_phi_mu_0}\\
		\phi_\mu (x)&\=(m+1)\bigl(\mu(\{x\})-\mu\bigl(\Gau^{-1}(x)\bigr)\bigr)\quad \text{ for } x\in R_m,\, m\in \N \label{e_phi_mu}  , 
	\end{align}
    noting that conditions~(i) and~(ii) ensure that $\phi_\mu$ is everywhere nonnegative.
	Define measures
	\begin{equation} \label{e_proof_nu_n}\\
		\nu_0\=\phi_\mu(0)\delta_0  , \qquad
		\nu_n  \=\sum\limits_{x\in R_n}\phi_{\mu}(x)\mu_{g^{-1}(x)}  , \qquad 
        \nu \=\sum\limits_{n=0}^{+\infty}\nu_n   .   
	\end{equation}

       We note that by the constructions in (\ref{e_proof_nu_n}), $\nu$ is a sum of positive measures. 
       Thus, $\nu$ is a nonnegative combination of the base elements $\delta_0$ and $\bigcup_{n\in\N}\{\mu_{g^{-1}(x)} : x\in R_n\}$ of $\RM_{[0,1)}$, since $g^{-1}(x) \in \cA$ (as $g_n\: \cA_n \to R_n$ is a bijection (cf.~Lemma~\ref{l_basic_properties_rational_measures}~(iv)), $x \in R_n$ guarantees $g^{-1}(x) \in \cA_n \subseteq \cA$, and thus $\mu_{g^{-1}(x)} \in \RM_{[0,1)}$). 
    
	We claim that $\mu =\nu$, so in particular
     $\nu$ is a probability measure. From the above it will therefore follow that $\nu \in \RM_{[0,1)}$, by the convexity of $\RM_{[0,1)}$, and hence the required result that $\mu\in\RM_{[0,1)}$.
    
    To verify that $\mu =\nu$ it suffices to show that
	\begin{equation*}
		\mu(\{y\})=\nu(\{y\})\quad\text{for all } y \in I\cap\Q   . 
	\end{equation*}
	For all $x\in (0,1)\cap \Q$, from the fact that $g^{-1}(x)\in \cA$, the definition of $\cA$ (see (\ref{e_def_cA_n})), and the definition of FCF measures (see (\ref{e_def_rational_measure})), we obtain that $\mu_{g^{-1}(x)}(\{1\})=0$, and hence that   
	\begin{equation}\label{e_proof_nu_equal_mu_1}
		\nu(\{1\})=0=\mu(\{1\})  . 
	\end{equation}
	Now suppose $y\in (0,1)\cap \Q$; by Lemma~\ref{l_Rn_disjoint_union}, there exists $n\in \N$ such that $y\in R_n$. By Lemma~\ref{l_basic_properties_rational_measures}~(v), for each $m\in \N$ and $x\in R_m$, we have
	\begin{equation*}
		(m+1)\mu_{g^{-1}(x)}(\{y\})=\begin{cases}
			1 &\text{if }y\in \cO(x)  , \\
			0 & \text{otherwise}  . 
		\end{cases}
	\end{equation*}
	Combining this with (\ref{e_phi_mu}) gives
	\begin{equation}\label{e_proof_phi_mu_xy}
		\phi_{\mu}(x)\mu_{g^{-1}(x)}(\{y\})=\begin{cases}
			\mu(\{x\})-\mu\bigl(\Gau^{-1}(x)\bigr) & \text{if } y\in \cO(x)  , \\
			0 &\text{otherwise}  . 
		\end{cases}
	\end{equation}	
	Fix $m\in \N$. When $m<n$, we have $y\notin \cO(x)$ for all $x\in R_m$, and combining this with (\ref{e_proof_nu_n}) and (\ref{e_proof_phi_mu_xy}) gives
	\begin{equation}\label{e_nu_m<n_y}
		\nu_m(\{y\})=\sum\limits_{x\in R_m}\phi_{\mu}(x)\mu_{g^{-1}(x)}(y)=0  . 
	\end{equation}
	When $m\geq n$, for each $x\in R_m$, we have $y\in \cO(x)$ if and only if $\Gau^{m-n}(x)=y$, and combining this with (\ref{e_proof_nu_n}) and (\ref{e_proof_phi_mu_xy}) gives
	\begin{equation}\label{e_nu_m>=n_y}
		\begin{aligned}
			\nu_m(\{y\})=\sum\limits_{x\in R_m}\phi_{\mu}(x)\mu_{g^{-1}(x)}(y)
			&=\sum\limits_{x\in\Gau^{-(m-n)}(y)}\mu(\{x\})-\mu\bigl(\Gau^{-1}(x)\bigr) \\
			&=\mu\bigl(\Gau^{-(m-n)}(y)\bigr)-\mu\bigl(\Gau^{-(m-n+1)}(y)\bigr)  . 
		\end{aligned}
	\end{equation}
	By (\ref{e_proof_nu_n}), (\ref{e_nu_m<n_y}), and (\ref{e_nu_m>=n_y}),
	\begin{equation*} 
		\begin{aligned}
			\nu(\{y\})
            =\sum\limits_{m=n}^{+\infty}\nu_{m}(\{y\})
			&=\sum\limits_{m=n}^{+\infty} \bigl( \mu\bigl(\Gau^{-(m-n)}(y)\bigr)-\mu\bigl(\Gau^{-(m-n+1)}(y)\bigr) \bigr)\\
			&=\sum\limits_{j=0}^{+\infty}\bigl(\mu \bigl(\Gau^{-j}(y)\bigr)-\mu \bigl(\Gau^{-(j+1)}(y)\bigr)\bigr)  . 
		\end{aligned}
	\end{equation*}
    By condition~(ii), $\mu \in \PPP(I)$, and the fact that $G^{-j}(y) \cap G^{-k}(y) = \emptyset$ if $0\leq j < k$, the series on the right-hand side of the above has nonnegative entries and is convergent. Thus by telescoping, we get
	\begin{equation}\label{e_nu_=_mu_x}
		 	\nu(\{y\})=\mu(\{y\})  . 
	\end{equation}    
	Combining (\ref{e_proof_nu_n}) and (\ref{e_proof_phi_mu_xy}), for each $m\in \N$, we get
	\begin{equation*}\label{e_nu_m_0}
			\nu_m(\{0\})=\sum\limits_{x\in R_m}\phi_{\mu}(x)\mu_{g^{-1}(x)}(\{0\})
			=\sum\limits_{x\in R_m}\mu(\{x\})-\mu\bigl(\Gau^{-1}(x)\bigr) 
			=\mu(R_m)-\mu(R_{m+1})  . 
	\end{equation*}
	Combining this with (\ref{e_proof_nu_n}), (\ref{e_phi_mu_0}), and a similar argument as above on the convergence of the series gives
	\begin{equation}\label{e_nu_0}
		\nu(\{0\})
        =\sum\limits_{n=0}^{+\infty}\nu_n(\{0\})
        =\mu(\{0\})-\mu(R_1)+\sum\limits_{n=1}^{+\infty}(\mu(R_n)-\mu(R_{n+1}))
        =\mu(\{0\})  . 
	\end{equation}
	Then by (\ref{e_proof_nu_equal_mu_1}), (\ref{e_nu_=_mu_x}), and (\ref{e_nu_0}), we conclude that $\mu=\nu \in \RM_{[0,1)}$, as required.
\end{proof}

Combining Lemma~\ref{l_relation_RM_and_RM_*} and Proposition~\ref{l_RM*_equivalent_def} together gives the following corollary, which is important in the proof of Theorem~\ref{t_weak_*_closure_M(I,T)}.

\begin{cor}\label{c_RM_equivalent_def}
	Suppose $\mu \in \cP(I)$ with $\mu (I\cap\Q)=1$. Then $\mu\in \RM$ if and only if 
    $\mu$ satisfies the following conditions:
	\begin{enumerate}[label=\rm{(\roman*)}]
		\smallskip
		\item $\mu(\{0\})\geq \mu(\{1\})$ and $\mu(\{0\})\geq \mu (R_1)$,
		\smallskip
		
		\item $\mu (\{x\})\geq \mu\bigl(\Gau^{-1}(x)\bigr)$ for all $x\in (0,1)\cap \Q$.
	\end{enumerate}
\end{cor}

\begin{proof}
	First assume that $\mu\in \RM$.  
    Using the fact that conditions~(i) and~(ii) are closed under convex combination, 
    the definition of $\RM$ (see Definition~\ref{d_rational_orbit_and_rational_measure}), and the fact that $\delta_0$ 
    satisfies conditions~(i) and~(ii), 
    it suffices to show that $\mu_{\fa}$ 
    satisfies conditions~(i) and~(ii) 
    for all $\fa\in \N^*$. 
    When $\fa\in \cA$, it follows from Lemma~\ref{l_basic_properties_rational_measures}~(vi) and (vii) that $\mu_{\fa}$ satisfies conditions~(i) and~(ii).
   When $\fa \in \cB$, it follows from Lemma~\ref{l_basic_properties_rational_measures}~(viii) and (ix) that $\mu_{\fa}$ satisfies conditions~(i) and~(ii).
	\smallskip
	
	To prove the converse, let us assume that $\mu$ satisfies conditions~(i) and~(ii). 
    Denoting $r\= \mu(\{1\}) \in [0,1/2]$ by condition~(i), define
	\begin{equation}\label{e_proof_c_rational_measures_nu}
		\nu\=  (\mu-r\delta_1) / (1-r)  , 
	\end{equation}
	and note that 
    (\ref{e_proof_c_rational_measures_nu}) gives
    $\nu(\{1\})=0\le \nu(\{0\})$ and $\nu(\{0\})=\mu(\{0\})/(1-r)\geq \mu(\{R_1\})/(1-r)=\nu(R_1)$, which is condition~(i) of Proposition~\ref{l_RM*_equivalent_def}, and 
    condition~(ii) gives, for all $x\in (0,1)\cap \Q$,  
    \begin{equation*}
        \nu(\{x\})
        = \mu (\{x\}) /(1-r)
        \geq  \mu\bigl(\Gau^{-1}(x) \bigr) \big/ (1-r)
        =\nu\bigl(\Gau^{-1}(x)\bigr), 
    \end{equation*}
    which is condition (ii) of Proposition~\ref{l_RM*_equivalent_def}.
    From Proposition~\ref{l_RM*_equivalent_def}
    it follows 
    that $\nu\in \RM_{[0,1)}$. 
    
    Now (\ref{e_proof_c_rational_measures_nu}) can be written as
		$(1-r)\nu = \mu -r\delta_1$,
	so by condition~(i),
	\begin{equation}\label{(vii)_can_be_used}
		(1-r)\nu(\{0\})=\mu(\{0\})\geq \mu(\{1\})=r   . 
	\end{equation}
	But (\ref{(vii)_can_be_used}) means, by Lemma~\ref{l_relation_RM_and_RM_*}, that $\mu =(1-r)\nu +r\delta_1\in \RM$, as required.
\end{proof}

Finally, we are able to prove our first main theorem.

\begin{proof}[\bf Proof of Theorem~\ref{t_weak_*_closure_M(I,T)}]
By Lemma~\ref{l_properties_sigma_and_Sigma}~(iv),
$ \MMM\bigl(\hSigma,\sigma\bigr)$ is the convex hull of 
        \begin{equation*}
            \MMM(\Sigma,\sigma) \cup \bigl\{\nu\in\MMM\bigl(\hSigma,\sigma_{\hSigma}\bigr):
        \nu(\Sigma)=0 \bigr\}.
        \end{equation*}
        But $\pi_*\: \MMM(\Sigma,\sigma)\to \MMM_{\irr}(I,\Gau)$ is a homeomorphism,
 according to Lemma~\ref{l_properties_pi_and_pi_*}~\ref{l_properties_pi_*},
and the pushforward 
$\hpi_*\:\MMM\bigl(\hSigma,\sigma\bigr) \to \cP(I)$
(cf.~Notation~\ref{pushforwards})
is affine,
so by Lemma~\ref{l_properties_pi_and_pi_*}~, it suffices to show that $\hpi_*(\mu)\in \RM$ for all $\mu \in \bigl\{\nu\in \MMM \bigl(\hSigma,\sigma_{\hSigma} \bigr):\nu(\Sigma)=0 \bigr\}$. 
	
	Fix $\mu \in \bigl\{\nu\in \MMM \bigl(\hSigma,\sigma_{\hSigma} \bigr):\nu(\Sigma)=0 \bigr\}$. We want to apply Corollary~\ref{c_RM_equivalent_def} to $\hpi_*(\mu)$.
	
	Fix $x= [a_1,a_2,\dots,a_n]$ with $a_1,\,\dots,\,a_{n}\in \N$, $a_n\geq 2$. By Lemma~\ref{l_properties_pi_and_pi_*}~\ref{l_propertiess_pi_and_pi_*_preimage}, we obtain
	\begin{align*}
		\hpi^{-1}(0)&=\cC(\infty)  , \qquad
		\hpi^{-1}(1)=\cC(1,\infty)  , \\
		\hpi^{-1}(x)&=\cC(a_1,\dots,a_n,\infty)\cup \cC(a_1,\dots,a_n-1,1,\infty)  . 
	\end{align*}
	Thus, 
	\begin{align*}
		\hpi_*(\mu)(\{0\})&=\mu\bigl(\hpi^{-1}(0)\bigr)=\mu(\cC(\infty))  , \qquad
		\hpi_*(\mu)(\{1\}) =\mu\bigl(\hpi^{-1}(1)\bigr)=\mu(\cC(1,\infty))  , \\
		\hpi_*(\mu)(\{x\})&=\mu\bigl(\hpi^{-1}(x)\bigr)=\mu(\cC(a_1,\dots,a_n,\infty))+\mu(\cC(a_1,\dots,a_n-1,1,\infty))  . 
	\end{align*}
	As a consequence, 
    $\hpi_*(\mu)(R_1)=\sum_{n=2}^{+\infty}\mu(\cC(n,\infty)) + \sum_{n=1}^{+\infty}\mu (\cC(n,1,\infty))$. Hence, from the fact that $\mu \in \MMM\bigl(\hSigma,\sigma_{\hSigma}\bigr)$ and $\mu(\Sigma) = 0$, we get
	\begin{align*}
		\hpi_*(\mu)(I\smallsetminus\Q)&=\mu\bigl(\hSigma\smallsetminus \Sigma\bigr)=1  , \\
		\hpi_*(\mu)(\{0\})&=\mu(\cC(\infty))\geq \mu(\cC(1,\infty))=\hpi_*(\mu)(\{1\})  , \\
		\hpi_*(\mu)(\{0\})&=\mu(\cC(\infty))
        \geq \sum\limits_{n=1}^{+\infty}\mu(\cC(n,\infty)) \\
        &\geq \sum\limits_{n=1}^{+\infty}\mu (\cC(n,1,\infty))+\sum\limits_{n=2}^{+\infty}\mu(\cC(n,\infty))=\hpi_*(\mu)\bigl(\{1/n\}_{n=2}^{+\infty} \bigr) 	,\\
		\hpi_*(\mu)(\{x\})&=\mu(\cC(a_1,\dots,a_n,\infty))+\mu(\cC(a_1,\dots,a_n-1,1,\infty))\\
		&\geq \sum\limits_{m=1}^{+\infty}\mu(\cC(m,a_1,\dots,a_n,\infty))+\sum\limits_{m=1}^{+\infty}\mu(\cC(m,a_1,\dots,a_n-1,1,\infty))
		= \hpi_*(\mu)\bigl(\Gau^{-1}(x)\bigr)  . 	
	\end{align*}
	The last identity above follows from (\ref{explicitpreimageset}). Therefore, applying Corollary~\ref{c_RM_equivalent_def} to $\hpi_*(\mu)$, we conclude $\hpi_*(\mu) \in \RM$, as required.
\end{proof}

	\section{The Ma\~n\'e lemma}\label{Sec_mane_lemma}

In this section we will prove a version of the Ma\~n\'e
lemma for the Gauss map $\Gau$,
and then use this to derive a revelation theorem.
The approach, by analogy with \cite{Bou00}, will involve a certain nonlinear operator which can be shown (cf.~Proposition~\ref{p_calibrated_sub-action_exists}) to have a fixed point function
(a so-called \emph{calibrated sub-action}, in the terminology of \cite{GLT09})
with certain regularity properties.

For a Borel measurable map $T \: I \to I$, and bounded Borel measurable function $\psi \: I\to\R$, to study the $(T,\psi)$-maximizing measures it is convenient, whenever possible, to consider a cohomologous function $\tpsi$ satisfying $\tpsi\leq \mea(T,\psi)$.
We recall the following (cf.~\cite[p.~2601]{Je19}):

\begin{definition}\label{d_normalisation}
	Suppose $T\: I\to I$ is Borel measurable, and $\psi\in C(I)$.
	If $\psi\le Q(T,\psi)$ and $\psi^{-1}(Q(T,\psi))$ contains $\supp \mu$ for some $\mu\in\MMM(I,T)$, then $\psi$ is said to be \emph{revealed}.
	If $Q(T,\psi)=0$ then $\psi$ is said to be \emph{normalised}; in particular, 
	a normalised function $\psi$ is revealed if and only if $\psi\le 0$ and $\psi^{-1}(0)$ contains $\supp \mu$ for some $\mu\in\MMM(I,T)$.
\end{definition}
\begin{lemma}\label{l_normalised_cohomologous}
	Suppose $T\: I\to I$ is Borel measurable, $\phi \: I\to\R$ is bounded and Borel measurable, and 
	$\Mmax(T,  \phi)\neq\emptyset$.
	Denote $\overline{\phi}\=\phi - Q(T,\phi)$, and suppose $\tphi\=\overline{\phi}+u-u\circ T$
	for some bounded Borel measurable function $u \: I\to\R$.
	Then:
	\begin{enumerate}[label=\rm{(\roman*)}]
		\smallskip
		\item	$Q\bigl(T,\tphi\bigr)=Q\bigl(T,\overline{\phi}\bigr)=0$.
		\smallskip
		
		\item	$\Mmax(T,  \phi)=\Mmax\bigl(T,  \overline{\phi}\bigr)=\Mmax\bigl(T,  \tphi\bigr)$.
		\smallskip

	\end{enumerate}	
\end{lemma}
\begin{proof}
	(i) and (ii) follow from (\ref{e_ergodicmax}), (\ref{e_setofmaximizngmeasures}), and the fact that
	$\int \!\tphi\,\mathrm{d}\mu=\int \!(\overline{\phi}+u-u\circ T)\,\mathrm{d}\mu=\int \!\overline{\phi}\,\mathrm{d}\mu$ for all $\mu\in \MMM(I,T)$.
\end{proof}
\subsection{Bousch operator}
The following operator $\cL_\psi$ is an analogue of the one used by Bousch in \cite{Bou00}. Instead of preimages used by Bousch, we use inverse branches in the definition to address the irregular behaviour of the Gauss map at
the points $0$ and $1$.

\begin{definition} \label{d_Def_BouschOp}
	Let $\psi \: I\to \R$ be bounded and Borel measurable.
	Define 
	$\cL_\psi\:B(I)\to B(I)$
	by
	\begin{equation*}
		\cL_\psi(u)(x) 
        \=\sup \{(u+\psi)(\Gau_a(x)):  a \in \N\}
        =\sup \{u ( 1/(a+x))+\psi( 1/(a+x)) : a\in \N\}  . 
	\end{equation*}
\end{definition}

Since $\psi$ and $u$ are bounded, $\cL_\psi (u)$ is well-defined. If $\psi$ and $u$ are continuous, by Proposition~\ref{p_inverse_branches}~(v) we have
\begin{equation}  \label{e_Def_BouschOp_2}
	\cL_\psi(u)(x)\=\max \bigl\{(u+\psi)(\Gau_{\ha}(x)) : \ha\in \widehat{\N} \bigr\}  . 
\end{equation}

\begin{lemma}\label{l_property_Bousch_operator}
If
 $\psi\in C(I)$
	and	$\overline{\psi} \= \psi - \mea (\Gau, \psi)$,
then the following hold:
	\begin{enumerate}[label=\rm{(\roman*)}]
		\smallskip
		\item\label{l_property_Bousch_operator_plus} If $x\in I$ and $u\in C(I)$, then $\mathcal{L}_\psi(u+c)=c+\mathcal{L}_\psi(u)$.
		
		\smallskip
		\item\label{l_property_Bousch_operator_iteration} If $x\in I$, $n\in\N$, and $u\in C(I)$, then	
		\begin{equation*}
			\cL^n_{\overline{\psi}}(u)(x)+n\mea(\Gau,\psi)
			=\cL^n_\psi(u)(x)
			=\sup \{u(\Gau_\mathbf{a}(x))+S_{n,\mathbf{a}}\psi(x): \mathbf{a}\in \N^n\}  . 
		\end{equation*} 
		
		\item\label{l_property_Bousch_operator_sup_equal_max} If $x\in I$, $n\in\N$, and $u\in C(I)$, then	
		\begin{equation*}
			\cL^n_\psi(u)(x)
			=\max \bigl\{u(\Gau_{\hfa}(x))+S_{n,\hfa}\psi(x): \hfa\in \hbN^n \bigr\}  . 
		\end{equation*} 
		
		\item\label{l_property_Bousch_operator_commutes_with_sup} $\mathcal{L}_\psi(\sup_{v\in\cH} v)=\sup_{v\in \cH}\mathcal{L}_\psi(v)$ for any collection $\cH$ of bounded real-valued functions on $I$.
		\smallskip
		
		\item\label{l_property_Bousch_operator_commutes_with_limit} 
		If $\{u_n\}_{n\in \mathbb{N}}$ is a pointwise convergent sequence of equicontinuous functions on $I$,
		then the identity $\lim\limits_{n\rightarrow +\infty}\mathcal{L}_\psi(u_n)=\mathcal{L}_\psi\bigl(\lim\limits_{n\rightarrow +\infty}u_n\bigr)$ holds, 
		where the limits are pointwise.
	\end{enumerate}		
\end{lemma}

\begin{proof}
	\smallskip
	\ref{l_property_Bousch_operator_plus} By Definition~\ref{d_Def_BouschOp}, for any $x\in I$ and $u\in C(I)$,
	\begin{equation*}
			\cL_\psi(u+c)(x)
            =\sup \{\psi(\Gau_{a}(x))+u(\Gau_{a}(x))+c : a\in \N\}
			=\cL_\psi(u)(x)+c   . 
	\end{equation*}

\smallskip
    \ref{l_property_Bousch_operator_iteration} The first identity immediately follows from the second identity. We use induction to prove the second identity: the case $n=1$ follows from Definition~\ref{d_Def_BouschOp}, and assuming  
it is satisfied for some $n=m\in \N$, then
\begin{equation*}
	\begin{aligned}
	    \cL_{\psi}^{m+1}(u)(x)
        &=\sup \bigl\{\psi(\Gau_a(x))+	\cL_{\psi}^{m}(u)(\Gau_a(x)) : a\in \N\bigr\}\\
		&=\sup \{\psi(\Gau_a(x))+	\sup\{u(y)+S_{m,\mathbf{a}}\psi(\Gau_a(x)) : y= \Gau_{\mathbf{a}}(\Gau_a(x)),\,\mathbf{a}\in \N^m\} : a\in \N\}\\
		&=\sup \bigl\{u(\Gau_\mathbf{b}(x))+S_{m+1,\mathbf{b}}\psi(\Gau_\mathbf{b}(x)) : \mathbf{b}\in \N^{m+1} \bigr\}  . 
	\end{aligned}
\end{equation*}

\smallskip

\ref{l_property_Bousch_operator_sup_equal_max} By Proposition~\ref{p_inverse_branches}~(iv), and the fact that $u,\,\psi\in C(I)$, if $x\in I$ then $\hfa \mapsto u(\Gau_{\hfa}(x))+S_{n,\hfa}\psi(x)$ can be seen as a continuous function on $\hbN^n$; so \ref{l_property_Bousch_operator_sup_equal_max} follows from \ref{l_property_Bousch_operator_iteration} and the fact that $\N^n$ is dense in $\hbN^n$.

\smallskip
\ref{l_property_Bousch_operator_commutes_with_sup} follows readily from the definition.

\smallskip
\ref{l_property_Bousch_operator_commutes_with_limit} Let $v$ be the pointwise limit of $\{u_n\}_{n\in\N}$ as $n$ tends to infinity. Fix arbitrary $x\in I$ and $\epsilon>0$. Since $\{u_n\}_{n\in\N}$ is equicontinuous, there exists $\delta\in (0,1)$ such that for each $y\in[0,\delta)$ and each $n\in\N$, 
\begin{equation}\label{e_Bousch_op_commutes_limit_1}
	\abs{u_n(y)-u_n(0)}< \epsilon / 3  . 
\end{equation}
Letting $n$ tend to infinity, we have 
\begin{equation}\label{e_Bousch_op_commutes_limit_2}
	\abs{v(y)-v(0)}\leq \epsilon/3  . 
\end{equation}

We can find $N_1\in \N$ such that if $n>N_1$ then $\abs{u_n(0)-v(0)}<\epsilon/3$. When $n> N_1$, for each $y\in \{\Gau_a(x)\}_{a\in \N}\cap[0,\delta)$, by (\ref{e_Bousch_op_commutes_limit_1}) and (\ref{e_Bousch_op_commutes_limit_2}), we obtain
\begin{equation}\label{e_Bousch_op_commutes_limit_close_to_0}
		\abs{u_n(y)-v(y)}
        \leq\abs{u_n(y)-u_n(0)}+\abs{u_n(0)-v(0)}+\abs{v(0)-v(y)}
		< 3 \cdot (\epsilon / 3)
        <\epsilon  . 
\end{equation}
Since $\{\Gau_a(x) : a\in \N\} \cap[\delta,1]$ is finite, we can find $N_2\in \N$ such that for each $n\geq N_2$ and each $y\in \{\Gau_a(x) : a\in \N\} \cap[\delta,1]$, we obtain 
\begin{equation}\label{e_Bousch_op_commutes_limit_away_0}
	\abs{u_n(y)-v(y)}<\epsilon  . 
\end{equation}
Let $N\=\max\{N_1,\,N_2\}$. For each integer $n>N$, by (\ref{e_Bousch_op_commutes_limit_close_to_0}) and (\ref{e_Bousch_op_commutes_limit_away_0}), we have $\abs{u_n(y)-v(y)}<\epsilon$ for each $y\in \{\Gau_a(x) : a\in \N\}$. Fix an arbitrary integer $n>N$. We choose $z_1,z_2\in \{\Gau_a(x) : a\in \N\}$ satisfying $\mathcal{L}_{\psi}(u_n)(x)<\psi(z_1)+u_n(z_1)+\epsilon$ and $\mathcal{L}_{\psi}(v)(x)<\psi(z_2)+v(z_2)+\epsilon$. Then by Definition~\ref{d_Def_BouschOp},
\begin{align*}
	\mathcal{L}_{\psi}(u_n)(x)-\mathcal{L}_{\psi}(v)(x)&<\psi(z_1)+u_n(z_1)+\epsilon-\psi(z_1)-v(z_1)=u_{n}(z_1)-v(z_1)+\epsilon<2\epsilon,\\
	\mathcal{L}_{\psi}(u_n)(x)-\mathcal{L}_{\psi}(v)(x)&>\psi(z_2)+u_n(z_2)-\psi(z_2)-v(z_2)-\epsilon=u_{n}(z_2)-v(z_2)-\epsilon>-2\epsilon.
\end{align*}
Statement~\ref{l_property_Bousch_operator_commutes_with_limit} now follows.
\end{proof}

\begin{lemma} \label{l_Bousch_Op_preserve_space}
	Suppose $\alpha\in(0,1]$ and $\phi \in \Holder{\alpha}(I)$. Then for each $u\in \Holder{\alpha}(I)$ and each $n\in \N$, we have $\mathcal{L}_{\phi}^n(u)\in \Holder{\alpha}(I)$ and 
	\begin{equation}\label{e_Bousch_Op_Holder_seminorm_bound}
		\Hseminormbig{\alpha}{ \cL_\phi^n (u) } 
		\leq K_\alpha ( \Hseminorm{\alpha}{ \phi}  +  \Hseminorm{\alpha}{u} ).
	\end{equation}
\end{lemma}

\begin{proof}
	Suppose $u\in \Holder{\alpha}(I)$
	and $x,\,y\in I$. Fix $\epsilon>0$.
	By Lemma~\ref{l_property_Bousch_operator}~(ii), there exists $\fa\in \N^n$
	such that
	\begin{equation}\label{starequality}  
		\cL_{\phi}^n(u)(x)<u(\Gau_\mathbf{a}(x))+S_{n,\mathbf{a}}\phi(x)+\epsilon  . 
	\end{equation}
	By Lemma~\ref{l_property_Bousch_operator}~(ii), we have
	\begin{equation}\label{2starinequality}  
		\mathcal{L}_{\phi}^n(u)(y)\geq u(\Gau_\mathbf{a}(y))+S_{n,\mathbf{a}}\phi(y)  . 
	\end{equation}
	Combining (\ref{starequality}) and (\ref{2starinequality}) gives
	\begin{equation}\label{3starcombined}
		\cL_{\phi}^n(u)(x)-\cL_{\phi}^n(u)(y)\leq S_{n,\mathbf{a}}\phi(x)+u(\Gau_\mathbf{a}(x))-S_{n,\mathbf{a}}\phi(y)-u(\Gau_\mathbf{a}(y))+\epsilon  . 
	\end{equation}
 Lemma~\ref{l_Distortion} and (\ref{e_def_K_alpha}) gives
	\begin{equation}\label{4star}
		S_{n,\mathbf{a}}\phi(x)-S_{n,\mathbf{a}}\phi(y)\leq K_\alpha  \Hseminorm{\alpha}{\phi} \abs{x-y}^{\alpha}.
	\end{equation}
	From the fact that $u\in \Holder{\alpha}(I)$, the intermediate value theorem, and Proposition~\ref{p_inverse_branches}~(i), there exists $\xi$ in between $x$ and $y$ such that
	$$u(\Gau_\mathbf{a}(x))-u(\Gau_\mathbf{a}(y))\leq \Hseminorm{\alpha}{u}\abs{\Gau_\mathbf{a}(x)-\Gau_\mathbf{a}(y)}^\alpha=\Hseminorm{\alpha}{u}\abs{x-y}^\alpha\abs{\Gau'_\mathbf{a}(\xi)}^\alpha\le c_0^{-2\alpha} \theta^{-2n\alpha}\Hseminorm{\alpha}{u}\abs{x-y}^\alpha.$$ 
 	But $c_0^{-2\alpha} \theta^{-2n\alpha}<K_{\alpha}$ (see (\ref{e_def_K_alpha})), so
	\begin{equation}\label{5star}
		u(\Gau_\mathbf{a}(x))-u(\Gau_\mathbf{a}(y)) \leq  K_\alpha \Hseminorm{\alpha}{u} \abs{x-y}^\alpha .
	\end{equation}
	Combining (\ref{3starcombined}), (\ref{4star}), and (\ref{5star}) gives
	\begin{equation}\label{e_e_Bousch_Op_Holder_seminorm_upper_bound}
			\cL_{\phi}^n(u)(x)-\cL_{\phi}^n(u)(y)\leq K_\alpha(   \Hseminorm{\alpha}{ \phi}  +  \Hseminorm{\alpha}{u}  )\abs{x-y}^\alpha+\epsilon .
	\end{equation}

	Since (\ref{e_e_Bousch_Op_Holder_seminorm_upper_bound}) is satisfied for all $x,\, y \in I$, by swapping the positions of $x$ and $y$, we obtain
	\begin{equation}\label{e_e_Bousch_Op_Holder_seminorm_lower_bound}
			\cL_{\phi}^n(u)(x)-\cL_{\phi}^n(u)(y)\geq -K_\alpha(\Hseminorm{\alpha}{\phi}+\Hseminorm{\alpha}{u})\abs{x-y}^\alpha -\epsilon  . 
	\end{equation}
Finally, $\cL^n_\phi(u)\in \Holder{\alpha}(I)$ and (\ref{e_Bousch_Op_Holder_seminorm_bound}) follows from (\ref{e_e_Bousch_Op_Holder_seminorm_upper_bound}), (\ref{e_e_Bousch_Op_Holder_seminorm_lower_bound}), and the fact that $\epsilon>0$ was arbitrary.
\end{proof}

We are now able to find a fixed point  
$u_\phi$ of the normalised Bousch 
operator $\cL_{\overline\phi}$:

\begin{prop}\label{p_calibrated_sub-action_exists}
	Suppose $\alpha\in(0,1]$ and $\phi \in \Holder{\alpha}(I)$.
	Then the function $u_\phi \: I\rightarrow\R$ given by
	\begin{equation}  \label{e_calibrated_sub-action_exists}
		u_\phi(x)\=\limsup\limits_{n\rightarrow +\infty}\cL_{\overline{\phi}}^n(\mathbbold{0})(x),\quad\text{ for } x\in I,
	\end{equation}
	where $\overline{\phi} \= \phi - \mea \bigl(\Gau, \phi\bigr)$,
	satisfies the following properties:
	\begin{enumerate}[label=\rm{(\roman*)}]
		\smallskip
		\item $\abs{u_\phi(x)}\leq K_{\alpha} \Hseminorm{\alpha}{\phi}$ for each $x\in I$,
		\smallskip
		
		\item $u_\phi\in \Holder{\alpha}(I)$ with $\Hseminorm{\alpha}{u_\phi}\leq K_\alpha\Hseminorm{\alpha}{\phi}$,
		\smallskip
		
		\item $\cL_{\overline{\phi}}(u_\phi)=u_\phi$.
	\end{enumerate}	
\end{prop}

\begin{proof}
		For each $n\in\N$ and each $x\in I$, we write 
	\begin{equation}\label{rnsn_def}
		r_n(x)\=\cL_{\overline{\phi}}^n(\mathbbold{0})(x)
		\qquad \text{and} \qquad
		s_n(x)\=\sup \{ r_m(x) : m\ge n \} . 
	\end{equation}
	Note that, 
	for each $x\in I$, the sequence $\{s_n(x)\}_{n\in\N}$ is nonincreasing and by (\ref{e_calibrated_sub-action_exists}) and (\ref{rnsn_def}),  
	\begin{equation*}
		u_\phi(x)=\lim\limits_{n\rightarrow +\infty}s_n(x)=\limsup\limits_{n\rightarrow +\infty}r_n(x).
	\end{equation*}
	
	\smallskip
	(\romannumeral1) Fix $x\in I$ and $n\in \N$. For each $\fa=(a_1,\dots,a_n)\in \N^n$, denote $p_\fa\=[\overline{a_1,\dots,a_n}]$ which satisfies $\Gau_\fa(p_\fa)=\Gau^n(p_\fa)=p_\fa$. By Proposition~\ref{p_inverse_branches}~(vi), (\ref{e_S_n_bf_a}), and the fact that $\mea\bigl(\Gau,\overline{\phi}\bigr)=0$, we get
	\begin{equation}\label{e_est_upper_S_n_a_periodic_points} 
		S_{n,\fa}\overline{\phi}(p_\fa)=S_n\overline{\phi}(p_\fa)
        =n\int_I \!\overline{\phi}\,\mathrm{d}\mu_{\cO(p_\fa)}\leq 0
	\end{equation}
	Combining this with Lemma~\ref{l_Distortion} gives
	\begin{equation}\label{e_est_upper_S_n_a}
			S_{n,\fa}\overline{\phi}(x)= 	S_{n,\fa}\overline{\phi}(x)-	S_{n,\fa}\overline{\phi}(p_\fa)+	S_{n,\fa}\overline{\phi}(p_\fa)\\
			\leq K_\alpha\Hseminorm{\alpha}{\phi}.
	\end{equation}
So by (\ref{rnsn_def}), Lemma~\ref{l_property_Bousch_operator}~(ii), and (\ref{e_est_upper_S_n_a}), we have 
	\begin{equation*}
		r_n(x)
        =\cL^n_{\overline{\phi}}(\mathbbold{0})(x)
        =\sup \bigl\{ S_{n,\fa}\overline{\phi}(x) : \fa\in \N^n \bigr\}
        \leq K_\alpha\Hseminorm{\alpha}{\phi}.
	\end{equation*}
Combining this with (\ref{rnsn_def}) and (\ref{e_calibrated_sub-action_exists}) gives 
	\begin{equation*}\label{e_u_phi_upper_bound}
		u_{\phi}(x)=\limsup\limits_{n\rightarrow +\infty}r_n(x)\leq K_\alpha \Hseminorm{\alpha}{\phi}  . 
	\end{equation*}
	Next, we will show that $u_\phi(x)\geq -K_\alpha \Hseminorm{\alpha}{\phi}$. 
	Fix $n\in \N$. We choose a point $A_n=(a_{i})_{i\in \N}\in \hSigma$ on which $S_n^{\sigma }\bigl(\overline{\phi}\circ \hpi\bigr)$
	attains its maximum value. Denote $y_n\=\hpi(\sigma^n(A_n))$ and  $\fa\=(a_1,a_2,\dots,a_n)$. By Lemma~\ref{l_properties_pi_and_pi_*}~\ref{l_properties_hpi_*_Birkhoff_sum}, we get 
	\begin{equation}\label{e_proof_Birkhoff_sum_max_preimage}
		S_{n,\hfa}\overline{\phi}(y_n)=S_n^{\sigma }\bigl(\overline{\phi}\circ \hpi\bigr)(A_n)  . 
	\end{equation}
	Choose $\mu\in \MMM_{\max}\bigl(\sigma_{\hSigma},\overline{\phi}\circ \hpi\bigr)$. By Proposition~\ref{p_ergodic_op_relations}~(i), we have $\mea\bigl(\sigma_{\hSigma},\overline{\phi}\circ \hpi\bigr)=\mea(T,\overline{\phi})=0$, and then we have $\int_{\hSigma}\! S_n^\sigma\bigl(\overline{\phi}\circ \hpi\bigr)\, \mathrm{d}\mu=0$.
	So for all $x\in I$, combining (\ref{rnsn_def}), Lemma~\ref{l_Bousch_Op_preserve_space}, and (\ref{e_proof_Birkhoff_sum_max_preimage}) gives
	\begin{align*}
		r_n(x)&\geq r_n(y_n)-K_\alpha\Hseminorm{\alpha}{\phi}
		\geq S_{n,\hfa}\overline{\phi}(y_n)-K_\alpha\Hseminorm{\alpha}{\phi}
		=S_n^{\sigma }\bigl(\overline{\phi}\circ \hpi\bigr)(A_n)-K_\alpha\Hseminorm{\alpha}{\phi}\\
		&\geq \int_{\hSigma}\! S_n^\sigma\bigl(\overline{\phi}\circ \hpi\bigr)\, \mathrm{d}\mu-K_\alpha\Hseminorm{\alpha}{\phi}
		=-K_\alpha\Hseminorm{\alpha}{\phi}  . 
	\end{align*}
	Combining this with (\ref{rnsn_def}) and (\ref{e_calibrated_sub-action_exists}) gives $u_\phi(x)\geq -K_\alpha\Hseminorm{\alpha}{\phi}$ for all $x\in I$, so
    (\romannumeral1) follows.
    
	\smallskip
	
   (\romannumeral2) Suppose $x,\,y\in I$ and fix $\epsilon>0$. By (\ref{rnsn_def}) and (\ref{e_calibrated_sub-action_exists}), there exists $N\in \mathbb{N}$ such that 
    $\abs{r_N(x)-u_\phi(x)}<\epsilon$ and
    $s_N(y)-u_{\phi}(y)<\epsilon$. So by (\ref{rnsn_def}) and Lemma~\ref{l_Bousch_Op_preserve_space},
\begin{equation}\label{uphixydifference1}	
	u_\phi(x)-u_\phi(y) < r_N(x)-s_N(y)+2\epsilon
	\leq r_N(x)-r_N(y)+2\epsilon
	\leq K_{\alpha}\Hseminorm{\alpha}{\phi}\abs{x-y}^\alpha+2\epsilon  , 
\end{equation}
where the final inequality uses (\ref{e_Bousch_Op_Holder_seminorm_bound}).
Similarly, there exists $M\in \N$ such that
$\abs{r_M(y)-u_\phi(y)}<\epsilon$ and
$s_M(x)-u_{\phi}(x)<\epsilon$, and an analogous calculation gives
\begin{equation}\label{uphixydifference2}
	u_\phi(x)-u_\phi(y) \ge  -K_{\alpha}\Hseminorm{\alpha}{\phi} \abs{x-y}^\alpha-2\epsilon.
\end{equation}

Since $\epsilon>0$ was arbitrary, (\romannumeral2) follows from (\ref{uphixydifference1}) and  (\ref{uphixydifference2}).

	\smallskip
	
	(\romannumeral3) First we prove that $\{s_n\}_{n\in\N}$ is equicontinuous. Fix arbitrary $\epsilon>0$ and $m\in\N$. By (\ref{e_Bousch_Op_Holder_seminorm_bound}) and (\ref{rnsn_def}), $\{r_n\}_{n\in\N}$ is equicontinuous. Hence there exists $\delta>0$ such that if $\abs{x-y}<\delta$, we have
	\begin{equation*}
		\abs{r_n(x)-r_n(y)}< \epsilon / 2
	\end{equation*}
for all $n\in \N$. Then fix arbitrary $x,\,y \in I$ satisfying $\abs{x-y}<\delta$.

	Since $s_m(x)=\sup_{k\geq m}\{r_k(x)\}$, we can find $N_1>m$ such that $s_m(x)<r_{N_1}(x)+\frac{\epsilon}{2}$. 
	Then we have
	\begin{equation*}
		\begin{aligned}
			s_m(x)-s_m(y)&< r_{N_1}(x)+\frac{\epsilon}{2}-s_m(y)
			\leq r_{N_1}(x)+\frac{\epsilon}{2}-r_{N_1}(y)
			\leq \frac{\epsilon}{2}+\frac{\epsilon}{2}=\epsilon  . 
			\end{aligned}
	\end{equation*}

	Similarly, we can find $N_2>m$ such that $s_m(y)<r_{N_2}(y)+\frac{\epsilon}{2}$.
    Then we have 
	\begin{equation*}
		\begin{aligned}
			s_m(x)-s_m(y)&> s_m(x)-r_{N_2}(y)-\frac{\epsilon}{2}
			\geq r_{N_2}(x)-r_{N_2}(y)-\frac{\epsilon}{2}
			\geq -\frac{\epsilon}{2}-\frac{\epsilon}{2}=-\epsilon  . 
		\end{aligned}
	\end{equation*}
	Therefore, $\{s_n\}_{n\in\N}$ is equicontinuous.
	
	If $x\in I$, then by Lemma~\ref{l_property_Bousch_operator}~(iv), (v), and (\ref{rnsn_def}),
\begin{align*}
	\mathcal{L}_{\overline{\phi}}(u_\phi)(x)
	&=\mathcal{L}_{\overline{\phi}}\bigl(\lim\limits_{n\to +\infty}s_n\bigr)(x)	=\lim\limits_{n\to+\infty}\mathcal{L}_{\overline{\phi}} \bigl( \sup \bigl\{ \mathcal{L}_{\overline{\phi}}^m(\mathbbold{0})(x) : m\ge n \bigr\}\bigr)\\
	&=\lim\limits_{n\to+\infty} \bigl( \sup \bigl\{ \mathcal{L}_{\overline{\phi}}^{m+1}(\mathbbold{0})(x) : m\ge n \bigr\} \bigr)
	=\lim\limits_{n\rightarrow+\infty}s_{n+1}(x)
	=u_{\phi}(x)  . \qedhere
\end{align*}
\end{proof}

\begin{definition}
    Suppose $\phi\in\Holder{\alpha}(I)$, and $u_\phi$ is the calibrated sub-action defined by (\ref{e_calibrated_sub-action_exists}). We define the \emph{revealed version} $\tphi$ by
	 \begin{equation*}\label{e_Def_tphi}
		\tphi\= \overline{\phi}+u_\phi-u_\phi\circ \Gau   . 
	\end{equation*}
\end{definition}

We are now able to prove Theorem~\ref{t_mane}, a Ma\~n\'e lemma for the Gauss map, which resembles the form of \cite[Lemma~A]{Bou00}.

\begin{proof}[\bf Proof of Theorem~\ref{t_mane}]
	This follows immediately from Definition~\ref{d_Def_BouschOp} and Proposition~\ref{p_calibrated_sub-action_exists}.
\end{proof}

\subsection{Maximizing set}

Now, by analogy with the set $Z'$ of \cite[p.~495]{Bou00}, and the admissible words 
defined in \cite[Definition~2.3]{BM02}, we
wish to relate each $\phi\in \Holder{\alpha}(I)$ to a 
\emph{maximizing set} $\cK(\phi)\subseteq \hSigma$. 

\begin{definition}\label{Maximizing set}
	Given $\alpha\in(0,1]$, $\phi \in \Holder{\alpha}(I)$, and $u_\phi$ 
    the calibrated sub-action defined by (\ref{e_calibrated_sub-action_exists}), 
	denote
	\begin{align}
		\Phi&\=\phi\circ \hpi,\label{e_def_Phi}\\
		\overline{\Phi}&\= \Phi -\mea(\Gau,\phi)=\overline{\phi}\circ \hpi,\label{e_def_bar_Phi} \\
		U_\Phi&\=u_\phi\circ \hpi,\text{ and }\label{e_def_U_Phi}\\
		\Psi&\=\Phi- \mea (\Gau,\phi)+U_\Phi-U_\Phi\circ \sigma  , \label{e_def_Psi}
	\end{align}
and define the \emph{maximizing set} for $\phi$ by $\cK(\phi)\=\bigcap_{n=1}^{+\infty}\sigma_{\hSigma}^{-n}\bigl( \Psi^{-1}(0) \bigr)$.
\end{definition}

\begin{lemma}\label{l_maximizing set}
	If $\phi \in \Holder{\alpha} (I)$ with $\alpha \in (0,1]$, then the following hold:
		\begin{enumerate}[label=\rm{(\roman*)}]
		\smallskip
		\item\label{l_maximizing set_Phi_U_Holder} $\Phi,\, U_\Phi,\, \Psi\in \Holder{\alpha}\bigl(\hSigma,d_{\hrho}\bigr)$.
		\smallskip
		
		\item\label{l_maximizing set_mane_equation} $U_\Phi$ satisfies the functional equation
		\begin{equation}\label{e_mane_equation_shift_space}
		    U_\Phi(A)=\max \bigl\{\overline{\Phi}(B)+U_\Phi(B) : B\in\sigma_{\hSigma}^{-1}(A) \bigr\},\quad \text{ for all } A\in \hSigma  ,
		\end{equation}
	and consequently $\Psi\leq 0$.
		\smallskip
		
		\item\label{l_maximizing set_basic_property} $\cK(\phi)$ is a nonempty compact closed subset of $\hSigma$, with $\sigma(\cK(\phi))\subseteq \cK(\phi)$.
		\smallskip
		
		\item\label{l_maximizing set_supported_maximizing_measures} $\MMM_{\max} \bigl(\sigma_{\hSigma},\Phi\bigr)=\MMM_{\max} \bigl(\sigma_{\hSigma},\Psi\bigr)=\bigl\{\mu \in \MMM\bigl(\hSigma,\sigma_{\hSigma}\bigr):\supp \mu \subseteq \cK(\phi)\bigr\}\neq \emptyset$.
	\end{enumerate}	
\end{lemma}

\begin{proof}
	\ref{l_maximizing set_Phi_U_Holder} follows immediately from the fact that $\phi,\, u_\phi\in \Holder{\alpha}(I)$, and since $\hpi$, $\sigma$ are Lipschitz (see Lemmas~\ref{l_properties_pi_and_pi_*}~\ref{l_properties_pi_Lipschitz} and~\ref{l_properties_sigma_and_Sigma}~(i)).
    
    \smallskip
	
    \ref{l_maximizing set_mane_equation} Fix $A\in \hSigma$. By definition of $\Phi$ and $U_\Phi$ (see (\ref{e_def_Phi}) and (\ref{e_def_U_Phi})), Proposition~\ref{p_calibrated_sub-action_exists}~(iii), (\ref{e_Def_BouschOp_2}), and Lemma~\ref{l_properties_pi_and_pi_*}~\ref{l_properties_sigma_and_pi_commutes}, we obtain 
\begin{align*}
    U_\Phi(A) = u_\phi(\hpi(A))
    &= \max \bigl\{\overline{\phi}(\Gau_{\ha}(\hpi(A))) + u_\phi(\Gau_{\ha}(\hpi(A))) : \ha\in \hbN \bigr\}\\
    &= \max \bigl\{\overline{\phi}(\Gau_{\ha}(\hpi(\sigma(\ha A)))) + u_\phi(\Gau_{\ha}(\hpi(\sigma(\ha A)))) : \ha\in \hbN \bigr\}\\
    &= \max \bigl\{\overline{\phi}(\hpi(\ha A)) + u_\phi(\hpi(\ha A)) : \ha\in \hbN \bigr\}\\
    &= \max \bigl\{\overline{\Phi}(B) + U_\Phi(B) : B\in\sigma_{\hSigma}^{-1}(A)\bigr\}.
\end{align*}

\smallskip

\ref{l_maximizing set_basic_property} By definition of $\cK(\phi)$, it is immediate from the continuity of $\sigma$ and $\Psi$ that $\cK(\phi)$ is compact. By  definition of $\cK (\phi)$, it is also clear that $\sigma(\cK(\phi))\subseteq \cK(\phi)$. The fact that $\cK(\phi)$ is nonempty will follow directly from \ref{l_maximizing set_supported_maximizing_measures} and the fact that $\Mmax\bigl(\sigma_{\hSigma},\Phi\bigr)$ is nonempty.

\smallskip

\ref{l_maximizing set_supported_maximizing_measures}  The first identity follows from (\ref{e_def_Psi}), (\ref{e_def_Phi}), and Proposition~\ref{p_ergodic_op_relations}~(i). 
To establish the second identity, we first note that by the first identity, (\ref{e_setofmaximizngmeasures}), and \ref{l_maximizing set_mane_equation}, every $\mu \in \MMM\bigl(\hSigma,\sigma_{\hSigma}\bigr)$ with $\supp \mu \subseteq \cK(\phi) \subseteq \Psi^{-1} (0)$ is in $\Mmax\bigl( \sigma_{\hSigma}, \, \Psi  \bigr)$. 
Conversely, by \ref{l_maximizing set_mane_equation} and Proposition~\ref{p_ergodic_op_relations}~(i), every $\mu \in \Mmax \bigl(\sigma_{\hSigma}, \Psi \bigr)$ satisfies $\int_{\hSigma}\! \Psi \, \mathrm{d} \mu = 0$. 
By \ref{l_maximizing set_mane_equation}, $\supp \mu$ is a subset of the compact set $\Psi^{-1} (0)$. It now follows from the $\sigma_{\hSigma}$-invariance of $\mu$ that $\supp \mu \subseteq \bigcap_{n=0}^{+\infty} \sigma^{-n}  \bigl( \Psi^{-1} ( 0 ) \bigr) = \cK(\phi)$. The inequality follows from the fact that $\MMM\bigl(\hSigma,\sigma_{\hSigma}\bigr)$ is compact with respect to the weak$^*$ topology.
\end{proof}

\section{Typical finite optimization}\label{sec_TFO}

It will be convenient to classify the potential
functions $\phi\in \Holder{\alpha}(I)$ 
as follows:

\begin{definition}[Classification of potentials]\label{d_classification_of_potential}
	For $\alpha\in(0,1]$, a function $\phi\in \Holder{\alpha}(I)$ is said to be
	\begin{enumerate}[label=\rm{(\roman*)}]
		\smallskip
		\item \emph{essentially compact}\footnote{The terminology follows \cite{JMU06, JMU07}.}
        if $\mea(\Gau,\phi)=\mea_m(\Gau,\phi)$ (cf.~Definition~\ref{def_restricted_maximum_ergodic_average}) for some $m\in\N$; let $\bad^\alpha(\Gau)$ denote the set of $\alpha$-H\"older essentially compact functions;
        
		\smallskip
		\item \emph{rationally maximized} if there exists $n\in \N$ and $\fa\in \N^n$ such that $\int_I\!\phi\,\mathrm{d}\mu_\mathbf{a}=\mea(\Gau,\phi)$ or $\phi(0)=\mea(\Gau,\phi)$; let $\ral^\alpha(\Gau)$ denote the set of $\alpha$-H\"older rationally maximized functions;
        
		\smallskip
		\item  let $\well^\alpha(I)$ denote the set of $\alpha$-H\"older functions satisfying neither (i) nor (ii) above. 
	\end{enumerate}
\end{definition}

Fix $\alpha\in (0,1]$. It is easy to check that $\psi\=-d(\cdot,\bcf_m)\in \bad^\alpha(\Gau)$ for all $m\in \N$ (cf.~\cite{JMU07} for a sufficient condition for a similar notion of essential compactness in a symbolic dynamical setting). See Example~\ref{ex_rational_maximized} for the construction of a function $\phi \in \ral^\alpha (\Gau)$.

In this section we will
firstly establish
the \emph{typical finite optimization} for rationally maximized potentials (see Theorem~\ref{t_rationl_locking_property} and its slightly stronger form Theorem~\ref{t_rationl_locking_property_locking}), 
secondly use Theorem~\ref{t_rationl_locking_property} together with an example of a rationally maximized potential to demonstrate the failure of the typical periodic optimization conjecture for $\alpha$-H\"older potentials (see Theorem~\ref{t:FailTPO}), and finally prove that the set 
 $\bad^\alpha(\Gau)$
of essentially compact functions is contained in the closure of $\Lock^\alpha(\Gau)$ (see Theorem~\ref{t_main_theorem_ess_compact} and its slightly stronger form Theorem~\ref{t_main_theorem_ess_compact_locking}).

Let $T\: X \rightarrow X$ be a Borel measurable map on a metric space $X$, and $\alpha\in(0,1]$.
We define 
    $\sP(T)$
to be the set of those continuous functions $\phi\:X\to\R$ with a $(T,\phi)$-maximizing measure supported on a periodic orbit, and define
$
	\sP^\alpha(T) 
$ to be those $\alpha$-H\"older functions in 
$\sP(T)$.

 If a function $\phi \in \sP^\alpha(T)$ satisfies $\card \Mmax(T,  \phi) = 1$ and $\Mmax(T,   \phi) = \Mmax(T,  \psi)$ for all $\psi \in  \Holder{\alpha}(X)$ sufficiently close to $\phi$ in $\Holder{\alpha}(X)$, we say that $\phi$ has the \emph{(periodic) locking\footnote{The terminology follows \cite{Boc19, BZ15} (see also e.g.~\cite{Bou00, Je00}).} property} in $\Holder{\alpha}(X)$ (with respect to $T$). The set $\Lock^\alpha(T)$ is defined to consist of all $\phi \in \sP^\alpha(T)$ satisfying the periodic locking property in $\Holder{\alpha}(X)$.

Similarly, if a function $\phi \in \Holder{\alpha}(X)$ satisfies $\card \Mmax^*(T,  \phi) = 1$, $\Mmax^*(T,  \phi) = \Mmax^*(T,  \psi)$ for all $\psi \in  \Holder{\alpha}(X)$ sufficiently close to $\phi$ in $\Holder{\alpha}(X)$, and the unique limit-maximizing measure is 
uniformly distributed\footnote{
A consequence of Theorem \ref{t_weak_*_closure_M(I,T)} is that for the Gauss map $\Gau$, all finitely supported extremal points of  $\overline{\MMM(I,\Gau)}$ are \emph{equidistributions} on their support; by contrast, for more general maps $T\:X\to X$, finitely supported extremal points of  $\overline{\MMM(X,T)}$ need not give equal mass to their atoms. So it is natural to replace ``uniformly distributed'' by ``supported'' in our definition of the \emph{finite locking property} in such general settings.} on a finite set, 
we say that $\phi$ has the \emph{finite locking property} in $\Holder{\alpha}(X)$, and define $\Flock^\alpha(T)$ to consist of those $\phi \in \Holder{\alpha}(T)$ with the finite locking property in $\Holder{\alpha}(X)$.

 In the proof of Theorem~\ref{t_main_theorem_ess_compact_locking}, we show that for an arbitrary $\phi\in\bad^\alpha(\Gau)$, any perturbation of
 the form $\phi'=\phi -\epsilon d(\cdot,\cO)^\alpha$, with $\epsilon>0$ sufficiently small, belongs to $\sP^\alpha(\Gau)$,
 where $\cO$ is a particular periodic orbit. The perturbation argument in our proof of Theorem~1.3 is mainly inspired by ideas appearing in \cite{Co16}, \cite{Boc19}, \cite{HLMXZ25}, and \cite{LZ25}.

In addition to the overview of our proof strategy given in Section \ref{sct_intro}, we note that
compared with the ideas and techniques in the aforementioned works, our approach is to apply the closing lemma from the uniformly expanding scenario in a neighbourhood of the support of a maximizing measure (supported on $\bcf_m$ for some $m \in \mathbb{N}$), and then to carry out a local analysis following the perspective discussed in \cite{Boc19} near $\bcf_m$. The main difficulties in our setting arise from the fact that $\Gau$ has \emph{countably} many inverse branches, and is discontinuous.

The technical ingredients of the proof consist of
(1) quantitatively avoiding the discontinuities of $\Gau$, using the fact that $\Gau$ is Lipschitz and distance-expanding in a small neighbourhood of $\bcf_m$ for each $m \in \mathbb{N}$ (cf.~Lemma~\ref{l_basic_property_bounded_continued_fractions}) and
(2) handling the perturbation argument with constants that are necessarily local (such as $\eta_m$, $\lambda_m$, $L_1$, $L_2$, and $L_3$ in the proof of Theorem~\ref{t_main_theorem_ess_compact_locking}).

 The proof of Theorem~\ref{t_rationl_locking_property_locking} is inspired by the proof of the periodic locking property (cf.~\cite{BZ15} and \cite[Remark~4.5]{YH99}). We first prove a technical lemma (Lemma~\ref{l_rational_locking}). The technical part in the proof of Lemma~\ref{l_rational_locking} is the construction of a \emph{transport sequence}\footnote{This terminology follows \cite{BZ15}.} in a given rational orbit. Since there exists more than one FCF measure in a given rational orbit, the perturbation in the proof of Theorem~\ref{t_rationl_locking_property_locking} is also more sophisticated than the one for the periodic locking property.

\subsection{TFO for rationally maximized potentials}

In this subsection 
we will establish the following slightly stronger version of Theorem~\ref{t_rationl_locking_property} (which in particular implies Theorem~\ref{t_rationl_locking_property}):

\setcounter{thml}{3}

\begin{thml}[TFO for rationally maximized potentials] \label{t_rationl_locking_property_locking}
	For $\alpha\in(0,1]$, the set $\Flock^\alpha(\Gau)$ contains an open dense subset of $\ral^\alpha(\Gau)$ (in the $\alpha$-H\"older topology).
\end{thml}

It follows immediately from the definition that $\Flock^\alpha(\Gau) \cap \ral^\alpha(\Gau)$ is open in $\Holder{\alpha}(I)$.

\begin{notation}
	Recall that $\bigcup_{n\in\N}R_n\= (0,1)\cap \Q$ (see Lemma~\ref{l_Rn_disjoint_union}).
	For each $x\in I\smallsetminus \Q$, define $\widetilde{\cO}(x)$ as follows:
	\begin{equation*}
		\tcO(x)\=\begin{cases}
			\{0\} & \text{if } x=0  , \\
			\{0,\,1\} & \text{if } x=1  , \\
			\bigl\{x,\,\Gau(x),\,\dots,\,\Gau^{n-1}(x),\,0,\,1\bigr\} & \text{if } x\in R_n\text{ for some }n\in \N  . 
		\end{cases}
	\end{equation*}
	Define 
	\begin{equation*}
		\cM_0\= \{\delta_0\} \quad \text{ and } \quad
		\cM_1\=\{\delta_0,\,(\delta_0+\delta_1)/2\}.
	\end{equation*}
	Suppose $n\in \N$ and $x=[a_1,a_2,\dots,a_n]\in R_n$ with $\fa=(a_1,\dots,a_n)\in \cA_n$ and $\fb\=g(\fa)=(a_1,\dots,a_{n-1},a_n-1,1)\in \cB_{n+1}$. Define 
	\begin{equation}\label{e_cM_x}
		\cM_x\=\{\mu_{\fa},\,\mu_{\sigma(\fa)}\dots,\,\mu_{\sigma^{n-1}(\fa)},\,\mu_{\fb},\,\mu_{\sigma(\fb)}\dots,\,\mu_{\sigma^{n}(\fb)},\,\delta_0\}  . 
	\end{equation}
	Denote $R_\epsilon(x)\=\Gau_{\fa}(\epsilon)=[a_1,\dots,a_n+\epsilon]$ and $L_\epsilon(x)\=\Gau_{\fb}(\epsilon)=[a_1,\dots,a_n-1,1+\epsilon]$.

	Let $\conv(\cM_x)$ denote the convex hull of $\cM_x$.
\end{notation}

The following two technical lemmas will be used in the proof of Theorem~\ref{t_rationl_locking_property_locking}.
\begin{lemma}\label{l_neighbourhood_rational_number}
		Suppose $\epsilon\in(0,1)$, $n\in \N$, and $x=[a_1,\dots,a_n]\in R_n$ with $\fa=(a_1,\dots,a_n)\in \cA_n$.
		\begin{enumerate}[label=\rm{(\roman*)}]	
			\smallskip
			\item When $n$ is odd, $R_\epsilon(x)<x<L_{\epsilon}(x)$. When $n$ is even, $L_\epsilon(x)<x<R_\epsilon(x)$.
			\smallskip
			
			\item Suppose $n$ is odd. If $y\in (R_\epsilon(x),x]$ then $\Absbig{ \Gau^i(y)-\Gau^i(x)}\leq \epsilon$ for all $0\le i\le n$. If $y\in (x, L_\epsilon(x))$ then $\Absbig{\Gau^i(y)-\Gau^i(x)}\leq \epsilon$ for all $0\le i\le n-1$,  and moreover $\abs{ \Gau^n(y)-1} \leq \epsilon$ and $\Absbig{ \Gau^{n+1}(y)} \leq \epsilon$.
			
			\smallskip
			\item Suppose $n$ is even. If $y\in [x,R_\epsilon(x))$ then $\Absbig{ \Gau^i(y)-\Gau^i(x)} \leq \epsilon$ for all $0\le i\le n$. If $y\in (L_\epsilon(x), x)$ then $\Absbig{ \Gau^i(y)-\Gau^i(x)} \leq \epsilon$ for all $0\le i\le n-1$, and moreover $\abs{\Gau^n(y)-1}\leq \epsilon$ and $\Absbig{ \Gau^{n+1}(y)} \leq \epsilon$.

            \smallskip
            \item If $\delta\in (0,1)$ then $\abs{R_\epsilon(x)-x}\leq \Absbig{R_\epsilon\bigl(\Gau^i(x)\bigr)-\Gau^i(x)}\leq \epsilon$ and
            $\abs{L_\epsilon(x)-x}\leq \Absbig{L_\epsilon\bigl(\Gau^i(x)\bigr)-\Gau^i(x)}\leq\epsilon/(1+\epsilon)< \epsilon$ for all $0\le i\le n-1$.
		\end{enumerate}
\end{lemma}
\begin{proof}
    We will write $\fb\=g(\fa)=[a_1,\dots,a_{n-1},a_n-1,1]$ throughout this proof.

    \smallskip
    
	(i) Note that $R_\epsilon(x)=\Gau_\fa(\epsilon)$, $L_\epsilon(x)=\Gau_\fb(\epsilon)$, and $x=\Gau_\fa(0)=\Gau_\fb(0)$. 
    When $n$ is odd, by Proposition~\ref{p_inverse_branches}~(iii) $\Gau_\fa$ is strictly decreasing and $\Gau_\fb$ is strictly increasing. So $R_\epsilon(x)=\Gau_\fa(\epsilon)<\Gau_\fa(0)=x=\Gau_\fb(0)<\Gau_\fb(\epsilon)=L_{\epsilon}(x)$. 
    Similarly, when $n$ is even, $\Gau_\fa$ is strictly increasing and $\Gau_\fb$ is strictly decreasing. Then $L_{\epsilon}(x)=\Gau_\fb(\epsilon)<\Gau_\fb(0)=x=\Gau_\fa(0)<\Gau_\fa(\epsilon)=R_\epsilon(x)$.

    \smallskip
    
	(ii) If $y\in (R_\epsilon(x),x]$, and 
	denoting $z\=\Gau^n(y)\in [0,\epsilon)$, we have 
    that $y=[a_1,\dots,a_n+z]$, so for all 
    $0\le i\le n$,
	\begin{equation*}
		\Absbig{\Gau^i(y)-\Gau^i(x)}
        =\abs{{\Gau_{\sigma^i(\fa)}(z)-\Gau_{\sigma^i(\fa)}(0)}}
        \leq \epsilon  , 
	\end{equation*}
	since $\Absbig{ \Gau'_{\sigma^i(\fa)}} \leq 1$, so the first part of (ii) follows.
	
	Now assuming that $y\in (x, L_\epsilon(x))$, and denoting $w\=\Gau^{n+1}(y)\in (0,\epsilon)$, we see that $y=[a_1,\dots,a_{n-1},a_n-1,1+w]$, 
	so we have that, for all $0\le i\le n-1$, 
	\begin{equation*}
		\Absbig{ \Gau^i(y)-\Gau^i(x)}
        =\abs{ \Gau_{\sigma^i(\fb)}(w)-\Gau_{\sigma^i(\fb)}(0)}
        \leq \epsilon  , 
	\end{equation*} 
    since $\Absbig{ \Gau'_{\sigma^i(\fb)}} \leq 1$. 
	Moreover, $\abs{\Gau^n(y)-1}=\frac{w}{1+w}<w<\epsilon$, so the second part of (ii) follows.
    
	\smallskip
    
	(iii) The proof is very similar to the one for (ii).
    
	\smallskip
    
	(iv) For each $a\in \N$ and $x, \, y\in I$, 
    clearly $\abs{\Gau_a(x)-\Gau_a(y)}=\Absbig{ \frac{1}{a+x}-\frac{1}{a+y}} =\frac{\abs{x-y}}{(a+x)(a+y)}\leq \abs{x-y}$.
	So $\Absbig{ R_\delta\bigl(\Gau^{n-1}(x)\bigr)-\Gau^{n-1}(x)} =\abs{\Gau_{a_n}(\delta)-\Gau_{a_n}(0)}\leq \delta$,
	 and for each $0\le i\le n-1$ we obtain
	\begin{equation*}
		\Absbig{ R_\delta\bigl(\Gau^i(x)\bigr)-\Gau^i(x)}
        =\abs{ \Gau_{\sigma^i(\fa)}(\delta)-\Gau_{\sigma^i(\fa)}(0)}
        \geq  \abs{ \Gau_{\sigma^{i-1}(\fa)}(\delta)-\Gau_{\sigma^{i-1}(\fa)}(0)}
        =\Absbig{ R_\delta\bigl(\Gau^{i-1}(x)\bigr)-\Gau^{i-1}(x)}  , 
	\end{equation*}
	and so the first part of (iv) follows.
	
	Similarly, we obtain $\Absbig{L_\delta\bigl(\Gau^{n-1}(x)\bigr)-\Gau^{n-1}(x)} =\abs{\Gau_{a_n-1}(1/(1+\delta))-\Gau_{a_n-1}(1)}\leq \delta/(1+\delta)<\delta$
	and so for each $1\le i\le n-1$ we have
	\begin{equation*}
		\Absbig{ L_\delta\bigl(\Gau^i(x)\bigr)-\Gau^i(x)}
        =\abs{ \Gau_{\sigma^i(\fb)}(\delta)-\Gau_{\sigma^i(\fb)}(0)}
        \geq  \abs{ \Gau_{\sigma^{i-1}(\fb)}(\delta)-\Gau_{\sigma^{i-1}(\fb)}(0)}
        =\Absbig{ L_\delta\bigl(\Gau^{i-1}(x)\bigr)-\Gau^{i-1}(x)}  , 
	\end{equation*}
	and the second part of (iv) follows.
\end{proof}

\begin{lemma}\label{l_rational_locking}
	Suppose $\alpha \in (0,1]$ and $x\in I\cap \Q$.
    There exists $C_x >0$ such that for all $\nu\in\overline{\MMM (I, \Gau)}$ and  
	$\phi\in \Holder{\alpha} (I)$, 
	\begin{equation}\label{eq_lemma_rational_locking}
		\langle\nu,\phi\rangle\leq
	\max \{	\langle\mu,\phi\rangle : \mu \in \cM_x\}+C_x \Hseminorm{\alpha,\,I}{\phi}\bigl\langle\nu ,d \bigl(\cdot,\tcO(x) \bigr)^\alpha\bigr\rangle   . 
	\end{equation}
\end{lemma}

\begin{proof}
	Denote $p \= \card\tcO(x)$. Suppose $\nu\in\overline{\MMM (I, \Gau)}$ and  
	$\phi\in \Holder{\alpha} (I)$. Define
	\begin{equation}\label{e_proof_rational_locking_eta}
		\eta\=	\max \{	\langle\mu,\phi\rangle : \mu \in \cM_x \}  . 
	\end{equation}	
    If $p=1$ then $x=0$, $\tcO(x)=\{0\}$, and 
	(\ref{eq_lemma_rational_locking}) holds with $C_x=1$ because
	\begin{equation*}
			\langle\nu,\phi\rangle
		\le\langle \nu, \phi (0 ) + \Hseminorm{\alpha}{\phi}  d (\cdot,0 ) ^\alpha \rangle 
		=\langle\delta_0,\phi\rangle+ \Hseminorm{\alpha}{\phi}\bigl\langle \nu,d \bigl(\cdot,\widetilde{\cO}(x) \bigr)^\alpha\bigr\rangle   . 
	\end{equation*}
	
	\smallskip
	If $p\geq 2$, then by the ergodic decomposition theorem and the fact that $\MMM_{\irr}(I,\Gau)$ is dense in $\overline{\MMM(I,\Gau)}$ (see Lemma~\ref{l_rational_orbit/measure_close_to_periodic}), 
	it suffices to prove (\ref{eq_lemma_rational_locking}) for every ergodic measure $\nu \in \MMM_{\irr} (I, \Gau )$. 
	Fixing an arbitrary ergodic $\nu\in \MMM_{\irr}(I, \Gau )$, the Birkhoff ergodic theorem
	implies that there exists $a\in I$ with
	\begin{align}
		\langle\nu,\phi\rangle &= \lim_{k\to+\infty} \frac{1}{k} S_k\phi(w) \quad \text{ and} \label{eq ergodic theorem phi} \\
		\bigl\langle \nu,d \bigl(\cdot,\widetilde{\cO}(x) \bigr)^\alpha\bigr\rangle&=\lim_{k\to+\infty}\frac{1}{k}\sum_{i=0}^{k-1} d \bigl(\Gau^i (w), \widetilde{\cO}(x) \bigr)^\alpha .\label{eq ergodic theorem distance}
	\end{align}
    
    \smallskip
    
	\emph{Claim.} There exists $C_x >0$ and a \emph{transport sequence} $\underline{y}= \{y_i\}_{i=-1}^{+\infty}$ with entries from $\widetilde{\cO}(x)$ satisfying
	\begin{equation}\label{e_proof_rational_locking_limsup}
		 \limsup\limits_{k\to+\infty}\frac{1}{n}\sum\limits_{i=0}^{k-1}\phi(y_i)\leq \eta  \quad \text{ and }
	\end{equation}
	\begin{equation}\label{e_distance_control_transport_sequence}
		\Absbig{ \Gau^i (w)-y_i }\le C_x^{1/\alpha} d \bigl(\Gau^i (w),\tcO(x)\bigr) \quad \text{ for all } i\in\N_0.
	\end{equation}

    \smallskip
	
	Note that a consequence of this claim is, by (\ref{eq ergodic theorem phi}), (\ref{e_distance_control_transport_sequence}), (\ref{e_proof_rational_locking_limsup}), and (\ref{eq ergodic theorem distance}), that
	\begin{equation*}
		\begin{aligned}
			\langle\nu,\phi\rangle
			=\lim\limits_{k\to+\infty}\frac{1}{k}\sum_{i=0}^{k-1}\phi \bigl(\Gau^i (w) \bigr)
			&\le\limsup\limits_{k\to+\infty}\frac{1}{k}\sum_{i=0}^{k-1} \bigl(\phi (y_i )+ \Hseminorm{\alpha}{\phi}  \Absbig{ \Gau^i (w)-y_i }^\alpha \bigr)\\
			&\le\limsup\limits_{k\to+\infty}\frac{1}{k}\sum_{i=0}^{k-1} \phi (y_i )+\lim\limits_{k\to+\infty}\frac{1}{k}\sum_{i=0}^{k-1}C_x \Hseminorm{\alpha}{\phi} d \bigl(\Gau^i (w), \tcO(x) \bigr)^\alpha   \\
			&\leq\eta+C_x \Hseminorm{\alpha}{\phi}\bigl\langle\nu ,d \bigl(\cdot,\widetilde{\cO}(x) \bigr)^\alpha\bigr\rangle   , 
		\end{aligned}
	\end{equation*}
	so the required inequality (\ref{eq_lemma_rational_locking}) will hold, and the lemma will follow.
	
	\smallskip
    
	\emph{Proof of Claim.} Define
	\begin{equation}\label{e_Pf_ration_locking_delta}
		\delta\= \Delta \bigl( \tcO(x)\bigr) \big/ 3 , 
	\end{equation}
	where $\Delta\bigl(\tcO(x)\bigr)\=\min \bigl\{\abs{y-z}: y,z\in \tcO(x),\,y\neq z \bigr\}$.
    
    Define
	\begin{equation}\label{e_Pf_ration_locking_epsilon}
		\epsilon\=\begin{cases}
		      \delta / (1+\delta) & \text{if } x=1  , \\
			\min\{ \abs{R_\delta(x)-x},\, \abs{L_\delta(x)-x}\} & \text{if }x\in (0,1)\cap\Q  . 
		\end{cases}
	\end{equation}
	Define 
    \begin{equation*}
        C_x\=\epsilon^{-\alpha}>1.
    \end{equation*}
	
	 The sequence $\{y_i\}_{i=-1}^{+\infty}$ is constructed recursively as follows.  
	\smallskip
    
	\emph{Base step.} Define $y_{-1}\=0$.

    \smallskip
    
	\emph{Recursive step.} For some $t\in \N_0$, assume that $y_{-1},\, y_0,\, \dots,\, y_{t-1}$ are defined.
    
    \smallskip

	\emph{Case A.} Assume that $\Gau^t(w)\in [0,\epsilon)$. 
	Define $y_{t}\=0$. Then by (\ref{e_proof_rational_locking_eta}) and (\ref{e_cM_x}),
	\begin{equation}\label{e_Pf_rational_locking_case_A_less_than_eta}
		\phi(y_t)=\phi(0)\leq \eta  .  
	\end{equation}
	By (\ref{e_Pf_ration_locking_epsilon}) and (\ref{e_Pf_ration_locking_delta}), we have $\Absbig{\Gau^t(w)-y_t}<\epsilon\leq \delta\leq (1/3) \Delta\bigl(\tcO(x)\bigr)$ and so
	\begin{equation}\label{e_Pf_rational_locking_case_A_less_distance_control}
		\Absbig{\Gau^t(w)-y_t}=d\bigl(\Gau^t(w),\tcO(x)\bigr)  . 
	\end{equation}

    \smallskip
    
	\emph{Case B.} Assume that $\Gau^t(w)\in (1-\epsilon,1]$. 
	Define $y_{t}\=1$ and $y_{t+1}=0$. Then by (\ref{e_proof_rational_locking_eta}) and (\ref{e_cM_x}),
	 \begin{equation}\label{e_Pf_rational_locking_case_B_less_than_eta}
	 	(1/2)(\phi(y_t)+\phi(y_{t+1}))=(1/2)(\phi(1)+\phi(0))\leq \eta  .  
	 \end{equation}
	 By (\ref{e_Pf_ration_locking_epsilon}) and (\ref{e_Pf_ration_locking_delta}), we have $\Absbig{ \Gau^{t+i}(w)-y_{t+i}} \leq\delta\leq (1/3) \Delta\bigl(\tcO(x)\bigr)$ for $i\in\{0,\,1\}$ and so for $i\in \{0,\,1\}$,
	 \begin{equation}\label{e_Pf_rational_locking_case_B_less_distance_control}
	 	\Absbig{ \Gau^{t+i}(w)-y_{t+i}} 
        =d\bigl(\Gau^{t+i}(w),\tcO(x)\bigr)  . 
	 \end{equation}

    \smallskip
    
	\emph{Case C.} Assume that $\Gau^t(w)\in (z-\epsilon,z+\epsilon)$ for some $z\in \tcO(x)\smallsetminus\{0,\,1\}$. 
	Suppose $x=[a_1,\dots,a_{p-1}]\in R_{p-1}$ and $z=[a_{p-m},\dots,a_{p-1}]\in R_{m}$ for some 
    $0\le m \le p-2$. 
	Let us write $\fa\=(a_1,\dots,a_{p-2},a_{p-1})$ and  $\fb\=(a_1,\dots,a_{p-2},a_{p-1}-1,1)$.

    \smallskip
    
	\emph{Subcase (i).} Assume that $m$ is odd. By (\ref{e_Pf_ration_locking_epsilon}) and Lemma~\ref{l_neighbourhood_rational_number}~(i)(iv), we have $(z-\epsilon,z+\epsilon)\subseteq (R_\delta(z),L_\delta(z))$.
	
	If $\Gau^t(w)\in(R_\delta(z),z]$, we define $y_{t+i}\=\Gau^{i}(z)$ for all $0\le i\le m$. 
    Then by (\ref{e_proof_rational_locking_eta}) and (\ref{e_cM_x}),
	\begin{equation}\label{e_Pf_rational_locking_case_C(i)_R_less_than_eta}
		  (m+1)^{-1} S_{m+1}\phi (y_{t})=\langle \mu_{\sigma^{p-m-1}(\fa)}, \phi\rangle\leq \eta  . 
	\end{equation}
	By (\ref{e_Pf_ration_locking_epsilon}), (\ref{e_Pf_ration_locking_delta}), and Lemma~\ref{l_neighbourhood_rational_number}~(ii), we 
    see that for $0\le i \le m$,
	\begin{equation}\label{e_Pf_rational_locking_case_C(i)_R_less_distance_control}
		\Absbig{ \Gau^{t+i}(w)-y_{t+i}}
        =d\bigl(\Gau^{t+i}(w),\tcO(x)\bigr)  . 
	\end{equation}
	
	When $\Gau^t(w)\in (z,L_\delta(z))$, we define $y_{t+i}\=\Gau^{i}(z)$ for all $0\le i\le m-1$, and $y_{t+m}\=1$, and $ y_{t+m+1}\=0$. 
	Then by (\ref{e_proof_rational_locking_eta}) and (\ref{e_cM_x}),
	\begin{equation}\label{e_Pf_rational_locking_case_C(i)_L_less_than_eta}
		(m+2)^{-1} S_{m+2}\phi (y_{t})=\langle \mu_{\sigma^{p-m-1}(\fb)}, \phi\rangle\leq \eta  . 
	\end{equation}
	By (\ref{e_Pf_ration_locking_epsilon}), (\ref{e_Pf_ration_locking_delta}), and Lemma~\ref{l_neighbourhood_rational_number}~(ii), we get for $0\le i\le m+1$ that
	\begin{equation}\label{e_Pf_rational_locking_case_C(i)_L_less_distance_control}
		\Absbig{ \Gau^{t+i}(w)-y_{t+i}}
        =d\bigl(\Gau^{t+i}(w),\tcO(x)\bigr)  . 
	\end{equation}

    \smallskip
    
	\emph{Subcase (ii).} Assume that $m$ is even.
	 Using (\ref{e_Pf_ration_locking_epsilon}) and Lemma~\ref{l_neighbourhood_rational_number}~(i), (iv), we see that $(z-\epsilon,z+\epsilon)\subseteq (L_\delta(z),R_\delta(z))$.
	 
	When $\Gau^t(w)\in [z,R_\delta(z))$, we define $y_{t+i}\=\Gau^{i}(z)$ for all $0\le i \le m$. Then by (\ref{e_proof_rational_locking_eta}) and (\ref{e_cM_x}), 
	\begin{equation}\label{e_Pf_rational_locking_case_C(ii)_R_less_than_eta}
		(m+1)^{-1} S_{m+1}\phi (y_{t})=\langle \mu_{\sigma^{p-m-1}(\fa)}, \phi\rangle\leq \eta  . 
	\end{equation}
	By (\ref{e_Pf_ration_locking_epsilon}), (\ref{e_Pf_ration_locking_delta}), and Lemma~\ref{l_neighbourhood_rational_number}~(iii), we get that for $0\le i\le m$,
	\begin{equation}\label{e_Pf_rational_locking_case_C(ii)_R_less_distance_control}
		\Absbig{ \Gau^{t+i}(w)-y_{t+i}}
        =d\bigl(\Gau^{t+i}(w),\tcO(x)\bigr)  . 
	\end{equation}
	
	When $\Gau^t(w)\in (L_\delta(z),z)$, we define $y_{t+i}\=\Gau^{i}(z)$ for all $0\le i\le m-1$ and $y_{t+m}\=1$, and $y_{t+m+1}\=0$.	Then by (\ref{e_proof_rational_locking_eta}) and (\ref{e_cM_x}), 
    \begin{equation}\label{e_Pf_rational_locking_case_C(ii)_L_less_than_eta}
		(m+2)^{-1} S_{m+2}\phi (y_{t})=\langle \mu_{\sigma^{p-m-1}(\fb)}, \phi\rangle\leq \eta  . 
	\end{equation}
	By (\ref{e_Pf_ration_locking_epsilon}), (\ref{e_Pf_ration_locking_delta}), and Lemma~\ref{l_neighbourhood_rational_number}~(iii), we get for $0\le i\le m+1$,
	\begin{equation}\label{e_Pf_rational_locking_case_C(ii)_L_less_distance_control}
		\Absbig{ \Gau^{t+i}(w)-y_{t+i}}
        =d\bigl(\Gau^{t+i}(w),\tcO(x)\bigr)  . 
	\end{equation}
	
	\emph{Case D.} Assume that $d(\Gau^t(w),\tcO)\geq \epsilon $. Define $y_{t}\=0$.  
	Then by (\ref{e_proof_rational_locking_eta}) and (\ref{e_cM_x}),
	\begin{equation}\label{e_Pf_rational_locking_case_D_less_than_eta}
		\phi(y_t)=\phi(0)\leq \eta  .  
	\end{equation}
	By the definition of $C_x$ and (\ref{e_Pf_ration_locking_epsilon}), we get
	\begin{equation}\label{e_Pf_rational_locking_case_D_less_distance_control}
		\Absbig{ \Gau^t(w)-y_t}
        \leq  C_x^{1/\alpha}\epsilon
        \leq C_x^{1/\alpha}d\bigl(\Gau^t(w),\tcO(x)\bigr)  . 
	\end{equation}
	The recursive step is now complete, and
	combining (\ref{e_Pf_rational_locking_case_A_less_than_eta}), (\ref{e_Pf_rational_locking_case_B_less_than_eta}), (\ref{e_Pf_rational_locking_case_C(i)_R_less_than_eta}), (\ref{e_Pf_rational_locking_case_C(i)_L_less_than_eta}), (\ref{e_Pf_rational_locking_case_C(ii)_R_less_than_eta}), (\ref{e_Pf_rational_locking_case_C(ii)_L_less_than_eta}) and (\ref{e_Pf_rational_locking_case_D_less_than_eta}) gives (\ref{e_proof_rational_locking_limsup}).
	Combining (\ref{e_Pf_rational_locking_case_A_less_distance_control}), (\ref{e_Pf_rational_locking_case_B_less_distance_control}), (\ref{e_Pf_rational_locking_case_C(i)_R_less_distance_control}), (\ref{e_Pf_rational_locking_case_C(i)_L_less_distance_control}), (\ref{e_Pf_rational_locking_case_C(ii)_R_less_distance_control}), (\ref{e_Pf_rational_locking_case_C(ii)_L_less_distance_control}), and (\ref{e_Pf_rational_locking_case_D_less_distance_control}) gives (\ref{e_distance_control_transport_sequence}), thereby completing the proof of 
	the claim.
\end{proof}

\begin{lemma}\label{l_support_limiting_measure}
	Suppose $x\in I \cap \Q$ and $\alpha\in(0,1]$. Then $\bigl\langle \mu,  d\bigl(\cdot, \tcO(x)\bigr)^\alpha\bigr\rangle>0$ for all $\mu \in \overline{\MMM(I,\Gau)}\smallsetminus \conv(\cM_x)$.
\end{lemma}

\begin{proof}
	By Theorem~\ref{t_weak_*_closure_M(I,T)} and the ergodic decomposition theorem, it suffices to show that
    \begin{equation}\label{mu_strictly_pos}
    \bigl\langle \mu,  d\bigl(\cdot, \tcO(x)\bigr)^\alpha\bigr\rangle >0
    \end{equation}
    for all 
    measures $\mu$ that either are ergodic and in $\MMM_{\irr}(I,\Gau)$, or are FCF measures not contained in $\cM_x$.
	
	Assuming that $\mu \in \MMM_{\irr}(I,\Gau)$ is ergodic,
    if (\ref{mu_strictly_pos}) were false, then
	 $\bigl\langle \mu,  d\bigl(\cdot, \tcO(x)\bigr)^\alpha\bigr\rangle =0$ would give $\supp \mu  \subseteq \tcO(x)$, 
which contradicts the fact that $\mu (I\smallsetminus \Q)=1$.
	
	If, on the other hand, $\fa\in \N^*$ and
    $\mu=\mu_{\fa}$ is an FCF measure not contained in $\cM_x$,
	then $g(\fa)\notin \tcO(x)$, since otherwise $\fa$ is equal to the continued fraction expansion of some point in $\tcO(x)$, contradicting the definition of $\cM_x$.
	So $\supp \mu_\fa $ is not contained in $\tcO(x)$, so we get $\bigl\langle \mu , d\bigl(\cdot, \tcO(x)\bigr)^\alpha\bigr\rangle >0$, 
	and the result follows. 
\end{proof}

With the preceding lemmas in hand, we can now prove Theorem~\ref{t_rationl_locking_property_locking}.
\begin{proof}[\bf Proof of Theorem~\ref{t_rationl_locking_property_locking}]

	Fix arbitrary $\phi \in \ral^\alpha(\Gau)$ and real numbers $s>0$ and $t>0$. 
	
    When $\delta_0\in \Mmax(T,\phi)$, denote $x\=0$. When $\delta_0\notin \Mmax(T,\phi)$, let $n$ be the smallest integer such that there exists $\fa \in \N^n$ with $\mu_{\fa}\in \Mmax^*(\Gau,\phi)$, and
   choose $\fa=(a_1,a_2,\dots,a_n)\in \N^n$ satisfying $\mu_{\fa}\in \Mmax^*(\Gau,\phi)$, denote $x\= [a_1,\dots,a_n]$ and $\fb \= f(\fa)$.

   \smallskip
   
   \emph{Case A.} Assume that $\fa \in \cA_n$ or $\delta_0\in \Mmax(T,\phi)$. Define
   \begin{equation}\label{e_Pf_Phi_s}
   	\Phi_s\= \phi-sd(\cdot,\cO(x))^\alpha   . 
   \end{equation}
   By (\ref{e_Pf_Phi_s}), we have $\mea(\Gau,\phi)\geq \mea(\Gau,\Phi_s)$. Combining this with the fact that $\langle\mu_\fa,\Phi_s\rangle=\langle\mu_\fa,\phi\rangle=\mea (\Gau,\phi)$, we obtain $\mea(\Gau,\phi)\leq \mea(\Gau,\Phi_s)$ and $\mu_\fa \in \Mmax^*(\Gau,\Phi_s)$.
   By the choice of $n$, we get 
   \begin{equation*}
   	\langle \mu_{\fa} , \phi\rangle>	\langle \nu ,\phi \rangle  , 
   \end{equation*}
   for each $\nu \in \cM_x\smallsetminus \{\mu_{\fa},\,\mu_{\fb},\, \mu_{\sigma(\fb)}\}$.
   Since $1\notin \cO(x)$, $1\in \cO_{\fb}$, and $1\in \cO_{\sigma(\fb)}$ (see Remark~\ref{r_rational_orbit}), we get $\langle \mu_\fa,\Phi_s\rangle=\langle \mu_\fa,\phi\rangle\geq\langle \mu_\fb,\phi\rangle>\langle \mu_\fb,\Phi_s\rangle$ and $\langle \mu_\fa,\Phi_s\rangle=\langle \mu_\fa,\phi\rangle\geq\langle \mu_{\sigma(\fb)},\phi\rangle>\langle \mu_{\sigma(\fb)},\Phi_s\rangle$, and so  
   \begin{equation*}
   	\langle\mu_{\fa}, \Phi_s  \rangle>	\langle \nu ,\Phi_s \rangle  , 
   \end{equation*}
   for each $\nu \in \cM_x\smallsetminus \{\mu_{\fa}\}$.
    So we can define 
   \begin{equation}\label{e_Pf_rational_locking_delta_s}
   	\delta_s\= \langle\mu_{\fa} ,\Phi_s \rangle-\max \{ \langle\nu_{\fa} , \Phi_s\rangle : \nu \in \cM_x\smallsetminus \{\mu_{\fa}\} \}>0  . 
   \end{equation}
   Define
   \begin{equation*}\label{e_Pf_Phi_t,s}
   	\Phi_{t,s}\=\Phi_s-td\bigl(\cdot,\tcO(x)\bigr)^\alpha= \phi-sd(\cdot,\cO(x))^\alpha-td\bigl(\cdot,\tcO(x)\bigr)^\alpha   . 
   \end{equation*}
      Then for each $\nu \in \overline{\MMM(I,\Gau)}\smallsetminus \conv(M_x)$, when $\Hseminorm{\alpha}{\psi}< t/C_x$ (where $C_x>0$ is the constant from Lemma~\ref{l_rational_locking}), by Lemmas~\ref{l_rational_locking} and~\ref{l_support_limiting_measure},
   \begin{align*}
   	\langle \nu,\Phi_{t,s}+\psi  \rangle
    &=\langle \nu,\Phi_s \rangle+\langle \nu , \psi\rangle-\bigl\langle\nu, td\bigl(\cdot,\tcO(x)\bigr)^\alpha \bigr\rangle\\
   	&\leq \max\limits_{\mu\in \cM_x}\langle \mu, \Phi_s\rangle+\max\limits_{\mu\in \cM_x}\langle \mu, \psi\rangle+ (C_x\Hseminorm{\alpha}{\psi}-t)\bigl\langle\nu, d\bigl(\cdot,\tcO(x)\bigr)^\alpha \bigr\rangle
   	<\max\limits_{\mu\in \cM_x}\langle  \mu,\Phi_{t,s}+\psi\rangle . 
   \end{align*}
      For each $\nu \in \cM_x\smallsetminus\{\mu_\fa\}$, when $\Hnorm{\alpha}{\psi}< \delta_s /4$, by (\ref{e_Pf_rational_locking_delta_s}),
   \begin{align*}
   	\langle \nu,\Phi_{t,s}+\psi \rangle&= \langle \nu,\Phi_{s} \rangle+\langle \mu_\fa,\psi \rangle-\langle \mu_\fa,\psi \rangle+\langle \nu,\psi \rangle\\
   	&< \langle \mu_\fa,\Phi_s \rangle- (1/ 2 )\delta_s  + \langle \mu_\fa ,\psi\rangle+2\norm{\psi}_{\infty}
   	<\langle \mu_\fa,\Phi_{t,s}+\psi \rangle  . 
   \end{align*}
     So by Theorem~\ref{t_weak_*_closure_M(I,T)}, $\Mmax^*(\Gau,\Phi_{t,s}+\psi)=\{\mu_\fa\}$ when $\Hnorm{\alpha}{\psi}<\min  \{ t/C_x , \,  \delta_s/4  \}$, and consequently, $\Phi_{t,s}\in \ral^\alpha(\Gau) \cap \Flock^\alpha(\Gau)$.

   \smallskip
   
   \emph{Case B.} Assume that $\fa \in \cB_n$. Then $a_n=1$. 
   By the choice of $n$, we get 
   \begin{equation*}
   	\langle\mu_{\fa},\phi \rangle >	\langle\nu ,\phi \rangle   
   \end{equation*}
   for each $\nu \in \cM_x\smallsetminus \{\mu_{\fa}\}$. So we can define 
   \begin{equation}\label{e_Pf_rational_locking_delta}
   	\delta\= \langle \mu_{\fa},\phi \rangle-\max \{	\langle\nu, \phi \rangle : \nu \in \cM_x\smallsetminus \{\mu_{\fa}\} \}>0  . 
   \end{equation}
   For each $t>0$, define    
   \begin{equation*}\label{e_Pf_phi_t}
   	\phi_t\= \phi-td\bigl(\cdot,\tcO(x)\bigr)^\alpha   . 
   \end{equation*}
   Then for each $\nu \in \overline{\MMM(I,\Gau)}\smallsetminus \conv(M_x)$, when $\Hseminorm{\alpha}{\psi}< t/C_x$, by Lemmas~\ref{l_rational_locking} and~\ref{l_support_limiting_measure},
   \begin{align*}
      \langle \nu,\phi_t+\psi \rangle&=\langle \nu,\phi \rangle+\langle \nu,\psi \rangle-\bigl\langle \nu,td\bigl(\cdot,\tcO(x)\bigr)^\alpha \bigr\rangle\\
      &\leq\max\limits_{\mu\in \cM_x}\langle \mu,\phi \rangle+\max\limits_{\mu\in \cM_x}\langle\mu, \psi \rangle+(C_x\Hseminorm{\alpha}{\psi}-t)\bigl\langle\nu, d\bigl(\cdot,\tcO(x)\bigr)^\alpha \bigr\rangle
      <\max\limits_{\mu\in \cM_x}\langle \mu ,\phi_t+\psi\rangle   . 
   \end{align*}
   For each $\nu \in \conv(M_x)\smallsetminus\{\mu_\fa\}$, when $\norm{\psi}_{\infty}< \delta /4$, by (\ref{e_Pf_rational_locking_delta}),
   \begin{align*}
   	\langle \nu,\phi_t+\psi \rangle
    &=\langle \nu,\phi_t \rangle+\langle \mu_\fa,\psi \rangle-\langle \mu_\fa,\psi \rangle+\langle \nu,\psi \rangle\\
   	&< \langle \mu_\fa,\phi \rangle-  (1/2) \delta   +\langle \mu_\fa,\psi \rangle+2\norm{\psi}_{\infty}
   	<\langle \mu_\fa,\phi_t+\psi \rangle  . 
   \end{align*}
   So by Theorem~\ref{t_weak_*_closure_M(I,T)}, $\Mmax^*(\Gau,\phi_t+\psi)=\{\mu_\fa\}$ when $\Hnorm{\alpha}{\psi}<\min\{ t / C_x ,\, \delta/4 \}$, and consequently, $\phi_t\in \ral^\alpha(\Gau) \cap \Flock^\alpha(\Gau)$. 
   
   From the preceding two cases we deduce that $\ral^\alpha(\Gau) \cap \Flock^\alpha(\Gau)$ is dense in $\ral^\alpha(\Gau)$.

    It follows immediately from the definition that $\Flock^\alpha(\Gau) \cap \ral^\alpha(\Gau)$ is open in $\Holder{\alpha}(I)$. Therefore the theorem is proved.
\end{proof}

\subsection{Failure of TPO for H\"older continuous potentials}  \label{subsct:FailTPO}
In the example below, we construct a function $\phi\in \ral^\alpha(\Gau)$ with a limit-maximizing measure $(1/2)(\delta_0+\delta_1)$.

\begin{example}\label{ex_rational_maximized}
	For $\alpha\in (0,1]$, let $\phi\in \Holder{\alpha}(I)$ satisfy $\phi(0)=-1$, $\phi(1)=1$, $\phi(1/3)=\phi(3/4)=-2$, and be affine on each of the intervals $[0,1/3]$, $[1/3,3/4]$, and $[3/4,1]$. More precisely, $\phi$ is given by
	\begin{equation*}
		\phi(x)\=\begin{cases}
		-3x-1  & \text{if } 0\leq x\leq 1/3,\\
		-2     & \text{if } 1/3\leq x\leq 3/4,\\
		12x-11  & \text{if } 3/4\leq x\leq 1.
		\end{cases}
	\end{equation*} 
	We claim that $\mea(\Gau,\phi)=0$, $\phi\in \ral^\alpha(\Gau)$, and $\phi \notin \bad^\alpha(\Gau)$.
	
	Using that $(1/2)(\phi(0)+\phi(1))=0$, and that $(1/2)(\delta_0+\delta_1)\in \overline{\MMM(I,\Gau)}$ (see Lemma~\ref{l_rational_orbit/measure_close_to_periodic}), we see that $\mea(\Gau,\phi)\geq 0$.
	
	Fix $m\in \N$ with $m\geq 4$ and $x=[a_1,\dots,a_n,\dots]\in \bcf_m$. Denote $x_n\=\Gau^n(x)=[a_{n+1},a_{n+2},\dots]$ for $n\in\N_0$. We will recursively construct a sequence $\{n_k\}_{k\in \N}$ in $\{1,2\}$.
	
	\emph{Base step.} Case 1. Assume that either $a_1\geq 2$ or both $a_1=1$ and $a_2\leq 3$ hold. Then $x_0\leq 4/5$, and define $n_1\=1$. By the definition of $\phi$, we have 
	$\phi(x_0)\leq -1$. 
	
	Case 2. Assume that $a_1=1$ and $3< a_2\leq m$. Then $x_0\leq [1,m+1]= (m+1)/(m+2)$ and $1/(m+2)\leq x_1\leq 1/3$. Define $n_1\=2$. By the definition of $\phi$,
	\begin{equation}   \label{e_Pf_example}
    \begin{aligned}
		S_2\phi(x_0)
        &=(1/2)(\phi(x_0)+\phi(x_1))
        \leq (1/2)(\phi([1,m+1])+\phi(1/(m+2)))\\
		&=(1/2)\biggl(\frac{12(m+1)}{m+)}-11-\frac{3}{m+2}-1\biggr)
        =-\frac{15}{2(m+2)}.
	\end{aligned}	    
	\end{equation}

	\emph{Recursive step.} Assume that for some $t\in \N$, $\{n_i\}_{i=1}^t$ are defined. Denote $N_t\=\sum_{i=1}^t n_i$. 
	
	Case 1. Assume that either $a_{N_t+1}\geq 2$ or both $a_{N_t+1}=1$ and $a_{N_t+2}\leq 3$ hold. Then $x_{N_t}\leq 4/5 $, and define $n_{t+1}\=1$. By the definition of $\phi$, we have 
	$\phi(x_{N_t})\leq -1$.

	Case 2. Assume that $a_{N_t+1}=1$ and $3< a_{N_t+2}\leq m$. Then $x_{N_t}\leq [1,m+1]=(m+1)/(m+2)$ and $1/(m+2)\leq x_{N_t+1}\leq 1/3$. Define $n_{t+1}\=2$. By the definition of $\phi$, as in (\ref{e_Pf_example}),
	\begin{equation*}
		S_2\phi(x_{N_t})
        =(1/2)(\phi(x_{N_t})+\phi(x_{N_t+1}))
        \leq- 15/ (2(m+2)).
	\end{equation*}
	This finishes the recursive step. Hence, by our construction above, we get 
	\begin{equation*}
		\liminf\limits_{n\to +\infty}\frac{1}{n}S_n\phi(x)\leq \limsup\limits_{k\to +\infty}\frac{1}{N_k}S_{N_k}\phi (x)\leq \max\Bigl\{-1,-\frac{15}{2(m+2)}\Bigr\}.
	\end{equation*}
	Combining this with \cite[Proposition~2.2]{Je19} gives $\mea_m(\Gau,\phi)\leq \max\bigl\{-1,\,-\frac{15}{2(m+2)}\bigr\}$. Letting $m$ tend to infinity, from Proposition~\ref{p_mea_equal_sup_time_average}~(i) we see that  $\mea(\Gau,\phi)\leq 0$. 
	Consequently we conclude that $\mea(\Gau,\phi)=0$ and $\mea(\Gau,\phi)>\mea_m(\Gau,\phi)$ for all $m\in \N$. So $\phi\in \ral^\alpha(\Gau)$ but $\phi \notin \bad^\alpha(\Gau)$, and $(1/2)(\delta_0+\delta_1)$ is a $(\Gau,\phi)$-limit-maximizing measure.

\end{example}

\begin{proof}[\bf Proof of Theorem~\ref{t:FailTPO}]
    By Theorem~\ref{t_rationl_locking_property}, there is an open subset $U$ of $\psi\in  \Holder{\alpha}(I)$ with the finite optimization property such that $U$ is a dense subset of $\ral^\alpha(\Gau)$. Then the function $\phi$ counstructed in Example~\ref{ex_rational_maximized} is in the closure of $U$. Let $B^\epsilon$ denote the open ball in $\Holder{\alpha}(I)$ center at $\phi$ with radius $\epsilon>0$. Since each $\psi$ in $U$ admits a unique $(\Gau,\psi)$-limit-maximizing measure, it follows immediately that there exists $\delta>0$ such that each $\psi$ in the open set $B^\delta \cap U$ admits a unique $(\Gau,\psi)$-limit-maximizing measure and such that this measure is not $\delta_0$.
\end{proof}

In fact, by repeating arguments for Case~B in the proof of Theorem~\ref{t_rationl_locking_property_locking}, we can show that $\phi_t \=\phi- t d(\cdot,\{0,1\})^\alpha \in \Flock(G)$ for each real number $t>0$. We leave the proof for the interested reader.

\subsection{TPO for essentially compact potentials}  \label{subsec_The set Bad(I)}

The goal of this subsection is to prove that the set $\bad^\alpha(\Gau)$ is contained in the closure of $\sP^\alpha (\Gau)$. We will establish the following
slightly stronger version of Theorem~\ref{t_main_theorem_ess_compact} (which in particular implies Theorem~\ref{t_main_theorem_ess_compact}).

\begin{thml}[TPO for essentially compact potentials] \label{t_main_theorem_ess_compact_locking}
            For $\alpha\in (0,1]$, the set $\Lock^{\alpha}(\Gau)$ contains an open dense subset of $\bad^\alpha(\Gau)$ (in the $\alpha$-H\"older topology).
\end{thml}
It follows immediately from the definition that $\Lock^\alpha(\Gau) \cap \bad^\alpha(\Gau)$ is open in $\Holder{\alpha}(I)$.

Recall the maximizing set defined in Definition~\ref{Maximizing set}. To prove Theorem~\ref{t_main_theorem_ess_compact_locking}, we first require:

\begin{lemma}\label{l_equivalent_bad_alpha}
	Suppose $\alpha\in(0,1]$ and $\phi\in\Holder{\alpha}(I)$. The following are equivalent:
	\begin{enumerate}[label=\rm{(\roman*)}]	
		\smallskip
		\item $\phi\in \bad^\alpha(\Gau)$ as defined in Definition~\ref{d_classification_of_potential}.
		
		\smallskip
		\item There exists $\mu \in \Mmax(\Gau,\phi)$ with $\supp \mu \subseteq \bcf_{m}$ for some $m\in \N$.
       
		\smallskip
		\item There exists $\mu \in \Mmax(\Gau,\phi)$ with $0\notin \supp \mu$.	
        
        \smallskip
		\item $\cK(\phi)\cap \Sigma_m\neq \emptyset$ for some $m\in \N$.
	\end{enumerate}
\end{lemma}

\begin{proof}
    (i)$\implies$(ii): Assume that $\phi\in \bad^\alpha(\Gau)$. 
	Then there exists $m\in \N$ with $\mea(\Gau,\phi)=\mea_m(\Gau,\phi)$ by Definition~\ref{d_classification_of_potential}. Since $\Gau|_{\bcf_m}$ is continuous, and $\bcf_m$ is compact (see Lemma~\ref{l_basic_property_bounded_continued_fractions}), there exists $\mu\in \Mmax(\Gau|_{\bcf_m},\phi|_{\bcf_m})\neq \emptyset$. Since $\mu$ can be seen as a measure in $\MMM(I,\Gau)$, we obtain that $\mu \in \Mmax(\Gau,\phi)$ satisfies $\supp \mu \subseteq \bcf_{m}$.
	
	\smallskip

    (ii)$\implies$(iii) since $0\notin \bcf_m$ for all $m\in \N$.

	\smallskip

(iii)$\implies$(iv). Assume that there exists $\mu \in \Mmax(\Gau,\phi)$ with $0\notin \supp \mu$. 
By Proposition~\ref{p_ergodic_op_relations}~(ii), there exists $\nu\in \Mmax\bigl(\sigma_{\hSigma},\phi\circ\hpi\bigr)$ with $\hpi_*(\nu)=\mu$. 
By Lemma~\ref{l_maximizing set}~(iv), we obtain $\supp \nu \subseteq \cK(\phi)$.

Since $\hpi^{-1}(0)=\cC(\infty)$ (see Remark~\ref{r_hpi_equi_def}) and $0\notin \supp \mu=\hpi(\supp \nu) $ (see e.g.~\cite[p.~156]{Ak93}), it follows that 
$\cC(\infty)\cap \supp \nu =\emptyset$.

     We claim that there exists $m\in \N$ with $\supp \nu \subseteq \Sigma_m$. Suppose, to the contrary, that $\supp \nu$ is not contained in $\Sigma_n$ for all $n\in \N$. 
     Then since $\sigma(\supp \nu)=\supp \nu $ (see e.g.~\cite[p.~156]{Ak93}), for each $n\in \N$, there exist $A_n \in \supp \nu $ and $a_n\in\N$ satisfying $A_n\in \cC(a_n)$ and $a_n>n$.  As $(\hSigma,d_{\hrho})$ is compact, $\{A_n\}_{n\in \N}$ admits an accumulation point $A$. 
     By the fact that $\supp \nu $ is closed and the fact that $A_n\in \cC(a_n)$ and $a_n>n$, we get that $A\in \cC(\infty)\cap\supp \nu$, which contradicts to the fact that $\cC(\infty)\cap \supp \nu =\emptyset$.
     
     Therefore, we conclude that $\supp \nu \subseteq \cK(\phi) \cap \Sigma_m$ for some $m\in \N$ and $\cK(\phi)\cap \sigma_m\neq \emptyset$.

    \smallskip
    
	(iv)$\implies$(i): Assume that $\cK(\phi)\cap \Sigma_m\neq \emptyset$ for some $m\in \N$. 
    Since $\sigma(\cK(\phi))\subseteq \cK(\phi)$ and $\sigma(\Sigma_m)\subseteq \Sigma_m$, and both $\cK(\phi)$ and $\Sigma_m$ are compact, there exists $\nu \in \MMM\bigl(\hSigma,\sigma_{\hSigma}\bigr)$ with $\supp \nu \subseteq \Sigma_m$. By Lemma~\ref{l_maximizing set}~(iv), we see that $\mu \in\Mmax\bigl(\sigma_{\hSigma}, \phi\circ\hpi\bigr)$. By Proposition~\ref{p_ergodic_op_relations}~(ii), $\hpi_*(\nu)\in \Mmax(\Gau,\phi)$. Combining this with the fact that $\supp \hpi_*(\nu)\subseteq \bcf_m$ (see e.g.~\cite[p.~156]{Ak93}), we obtain that $\phi \in \bad^\alpha(\Gau)$, as required.             
\end{proof}

\begin{notation} 
For a compact metric space $(X,d)$,
a map $T \: X\rightarrow X$,
and a periodic orbit $\cO$ of $T$, 
the corresponding \emph{gap} is defined as
\begin{equation}\label{e_Def_gap}
	\Delta (\cO) = \Delta^d (\cO) \= \min\{ d(x, y) : x,\, y \in \cO, \, x \neq y \},
\end{equation}
with the convention that $\min \emptyset = +\infty$. 
For  $r, \, \theta >0$, 
define the \emph{$(r,\theta)$-gap} of $\cO$ by
\begin{equation}  \label{e_Def_gap_r_theta}
	\Delta_{r,\,\theta} (\cO) = \Delta^d_{r,\,\theta} (\cO) \= \min\{ r, \, \theta \cdot \Delta (\cO) \}.
\end{equation}
\end{notation}

We will need the following closing lemma, which was first introduced in \cite{HLMXZ25}.
\begin{lemma}  \label{l_sum_bound_by_gap}
	Let $\Gau$ be the Gauss map, $d$ the Euclidean metric on $I$,  $\alpha\in(0,1]$, and 
     $m\in \N$. Let $\mathcal{K}$ be a nonempty compact subset of $\bcf_m$ with $\Gau(\cK)\subseteq \cK$. For $r>0$, $\theta>0$, and $\tau > 0$, there exists a periodic orbit $\cO\subseteq \bcf_m$ of $\Gau$ with
	\begin{equation*} \label{e_l_sum bound_by_gap}
		\sum\limits_{x\in\cO} d( x, \cK )^\alpha \leq \tau \cdot ( \Delta_{r,\,\theta} ( \cO ) )^{\alpha}  . 
	\end{equation*}
\end{lemma}
\begin{proof}
	By Lemma~\ref{l_basic_property_bounded_continued_fractions}, $\Gau|_{\bcf_m}$ is open, Lipschitz, and distance-expanding, so
	the result is immediate from \cite[Proposition~2.1]{HLMXZ25}.
\end{proof}

\begin{proof}[\bf Proof of Theorem~\ref{t_main_theorem_ess_compact_locking}]
    By Proposition~\ref{p_periodic_locking}, it suffices to prove that $\bad^\alpha(\Gau)$ is contained in the closure of $\sP^\alpha(\Gau)$.
    
	For each periodic orbit $\cO$ of $\Gau$, define the measure $\mu_{\cO}$ by
	\begin{equation}\label{e_Pf_t_periodic_measure}
		\mu_{\cO} \= \frac{1}{ \card \cO} \sum\limits_{x\in \mathcal{O}} \delta_x \in \MMM(I, \Gau).
	\end{equation}
	Fix $\phi \in \bad^\alpha(\Gau)$ with no $\phi$-maximizing measure in $\Mmax(\Gau,\phi)$ supported on a periodic orbit of $\Gau$.
	Let $u_\phi \in \Holder{\alpha} (I)$ be the calibrated sub-action for $\phi$ and $\Gau$ (i.e., a fixed point of $\cL_{\overline\phi}$) from Proposition~\ref{p_calibrated_sub-action_exists}. Define
	\begin{equation} \label{e_Pf_main_theorem_1_tphi}
		\tphi \= \phi - \mea(\Gau, \phi) + u_\phi - u_\phi  \circ \Gau ,
	\end{equation}
	where $\mea(\Gau,\phi)$ is defined in (\ref{e_ergodicmax}). Then by Theorem~\ref{t_mane},
	\begin{equation}    \label{e_pf_main_theorem_1_tphi_nonpositive}
		\tphi (x) \leq 0 \quad \text{ for all }x\in I.
	\end{equation}

	By Lemma~\ref{l_equivalent_bad_alpha}, there exists $\mu \in \Mmax(\Gau,\phi)$ and $m\in \N$ satisfying $\supp \mu\subseteq \bcf_m$. 
	Let $\eta_m>0$ and $\lambda_m>1$ be the constants defined in Lemma~\ref{l_basic_property_bounded_continued_fractions} and denote $\cK\=\supp \mu$. Since $\tphi$ is continuous on $\cK$ (note that $\cK\subseteq E_m\subseteq I\smallsetminus\Q$), by (\ref{e_pf_main_theorem_1_tphi_nonpositive}) we obtain that 
    \begin{equation}\label{e_Pf_cK_contained_in_zeros}
        \tphi|_{\cK}\equiv 0  . 
    \end{equation} 
	Without loss of generality, assume that $\cK$ contains no periodic orbits of $\Gau$. 
	Recall (cf.~(\ref{e_def_set_F_m})) that the
    closed $\eta_m$-neighbourhood of $E_m$ is denoted by
	\begin{equation}\label{e_set_F_m}
		F_m\=\overline{B}_{d}^{\eta_m}(\bcf_m)=\{x\in I: d(x,\bcf_m)\leq \eta_m\}  . 
	\end{equation}
	Now $\phi, \, u_{\phi}\in \Holder{\alpha}(I)$ (see Proposition~\ref{p_calibrated_sub-action_exists}~(ii)), and $\Gau|_{F_m}$ is Lipschitz (see Lemma~\ref{l_basic_property_bounded_continued_fractions}~(ii)),
    with Lipschitz constant $\Hseminorm{\mathrm{LIP},\,F_m}{\Gau} \=\Hseminorm{1,\,F_m}{\Gau}$,
    so using (\ref{e_Pf_main_theorem_1_tphi}) we see that $\tphi|_{F_m}\in \Holder{\alpha}(F_m)$.  Let us write
	\begin{equation}\label{e_Pf_t_Holderseminorm_psi}
	L_1\=\Hseminormbig{\alpha,\,F_m}{\tphi}.
	\end{equation}
	Fix $\epsilon\in(0,1)$, and define constants
	\begin{align}
		L_2  &\= \Hseminorm{\mathrm{LIP},\,F_m}{\Gau} ,  \label{e_Pf_t_main1_C_1}  \\
		r   &\=\eta_m,\label{e_r}\\
     	\theta &\= \min\{1/3,\,1/(3L_2)\}   ,  \label{e_Pf_t_main1_theta} \\	
		L_3&\=\frac{L_1+1}{1-\lambda_m^{-\alpha}}>0, \label{e_Pf_t_main1_C_2} \\	
		\tau &\= \min\biggl\{ 1, \, \frac{ \epsilon }{ 2L_1},  \,  \frac{ \epsilon }{  ( 1 + L_3 \epsilon^{-1} )   L_1 } \biggr\} \leq 1 \label{e_Pf_t_Density_Lattes_tau}  . 
	\end{align}
	
	By Lemma~\ref{l_sum_bound_by_gap}, there exists a periodic orbit $\cO\subseteq \bcf_m$ of $\Gau$, of  period $p \= \card \mathcal{O}$, satisfying 
	\begin{equation}   \label{e_Pf_t_Density_Lattes_bound_by_gap}
		\sum\limits_{x\in \cO} d( x, \cK )^\alpha \leq \tau \cdot ( \Delta_{r,\,\theta} ( \cO ) )^{\alpha}.
	\end{equation}
	Define functions
	\begin{align}
		\phi' &\= \phi -\mea(\Gau,\phi)- \epsilon d (\cdot, \cO)^\alpha \in \Holder{\alpha}(I)  \text{ and}  \label{e_Pf_t_Density_phi'} \\
		\psi  &\= \tphi -  \epsilon  d (\cdot, \mathcal{O})^\alpha  =  \phi' - \mea(\Gau,\phi) + u_\phi - u_\phi \circ \Gau   .  \label{e_Pf_t_Density_psi}
	\end{align}

By (\ref{e_Pf_t_Density_phi'}), (\ref{e_Pf_t_Density_psi}), and Lemma~\ref{l_normalised_cohomologous},
	\begin{equation}  \label{e_Pf_t_Density_Lattes_Mmax_identical}
		\mea(\Gau,\phi')=\mea(\Gau,\psi)\text{ and }\Mmax (\Gau, \phi')  = \Mmax(\Gau,  \psi)  . 
	\end{equation} 
	
	\smallskip
	\emph{Claim.} The measure $\mu_{\cO}$ 
    (cf.~(\ref{e_Pf_t_periodic_measure})) belongs to $\Mmax (\Gau,  \psi)$, i.e., $\mea(\Gau, \psi) = \gamma$, where
	\begin{equation}  \label{e_Pf_t_gamma}
		\gamma \= \int\! \psi \, \mathrm{d}\mu_\cO 
		= \frac{1}{ p } \sum_{x\in\cO} \psi (x) 
		= \frac{1}{ p } \sum_{x\in\cO} \tphi (x)
		< 0.
	\end{equation}
	Note that the equality 
    $\frac{1}{ p } \sum_{x\in\cO} \psi (x) 
		= \frac{1}{ p } \sum_{x\in\cO} \tphi (x)$
        in (\ref{e_Pf_t_gamma})
    follows from (\ref{e_Pf_t_Density_psi}), whereas the inequality in (\ref{e_Pf_t_gamma}) follows from (\ref{e_pf_main_theorem_1_tphi_nonpositive}) and the assumption that $\phi \notin \sP^\alpha(\Gau)$.
	
	\smallskip
	
	Note  that if the claim holds 
	then (\ref{e_Pf_t_Density_Lattes_Mmax_identical}) gives that $\phi' \in \sP^\alpha(\Gau)$, and since $\epsilon$ can be chosen arbitrarily small we see that
	$\phi$ belongs to the closure
    of $\sP^\alpha(\Gau)$ in $\Holder{\alpha}(I)$, which is the required conclusion of Theorem~\ref{t_main_theorem_ess_compact}.
	
	\smallskip
	So to prove
    Theorem~\ref{t_main_theorem_ess_compact}
    it suffices to establish the claim. From the definitions of $\gamma$ (cf.~(\ref{e_Pf_t_gamma})) and $\mea(\Gau,\psi)$ (cf.~(\ref{e_ergodicmax})), we see that $\mea(\Gau,\psi)\geq \gamma$, so it remains to show that $\mea(\Gau,\psi)=\mea(\Gau,\phi')\leq \gamma$. 
	By Proposition~\ref{p_mea_equal_sup_time_average}~(ii), we only need to prove that for all $x\in I\smallsetminus \Q$,
	\begin{equation}\label{e_Pf_liminf_psi}
		\liminf\limits_{n\to+\infty}\frac{1}{n}S_n\phi'(x)=\liminf\limits_{n\to+\infty}\frac{1}{n}S_n\psi(x)\leq \gamma  , 
	\end{equation}
	where the first identity follows from (\ref{e_Pf_t_Density_phi'}), (\ref{e_Pf_t_Density_psi}), and the fact that $u_\phi$ is bounded (see Proposition~\ref{p_calibrated_sub-action_exists}~(i)).
	
	Fix $x\in I\smallsetminus\Q$. 
	In the remainder of this proof, we will divide the orbit of $x$ into segments such that the average on each segment is less than $\gamma$, i.e., we will recursively construct a sequence $\{x_k\}_{k\in \N}$ in $\cO(x)$ and a sequence $\{n_k\}_{k\in \N}$ in $\N$ satisfying 
	$x_{k+1}=\Gau^{n_k}(x_{k})$ and $S_{n_k}\psi (x_k)\leq n_{k}\gamma$.
    
	We first observe that by (\ref{e_Pf_t_gamma}), (\ref{e_pf_main_theorem_1_tphi_nonpositive}), (\ref{e_Pf_cK_contained_in_zeros}), (\ref{e_Pf_t_Holderseminorm_psi}), and (\ref{e_Pf_t_Density_Lattes_bound_by_gap}), 
\begin{equation}  \label{e_Pf_t_est_p_gamma}
	p \abs{\gamma} 
	= \sum_{x\in\cO} \Absbig{  \tphi (x) - 0 }
    \leq \sum_{x\in\cO} L_1  d(x,\cK)^\alpha
	\leq L_1 \tau \cdot ( \Delta_{r,\theta} ( \cO) )^{\alpha}  . 
\end{equation}

By (\ref{e_Pf_t_est_p_gamma}), (\ref{e_Pf_t_Density_Lattes_tau}), and (\ref{e_Def_gap_r_theta}),  
\begin{equation}  \label{e_Pf_t_est_rho}
	\rho  \=  \epsilon^{-1/\alpha} \abs{\gamma}^{1/\alpha}  
	\leq   ( L_1  \tau / \epsilon )^{1/\alpha}  \Delta_{r,  \theta} ( \cO )  \\
	<    \Delta_{r, \theta} ( \cO )\leq r   . 
\end{equation} 
	Let us denote $U \= \overline{B}_d^{\rho} (\cO)=\{x\in I: d(x,\cO)\leq \rho\}$ (cf.~Section~\ref{sct_Notaion}).
	
	\smallskip
    \emph{Base step.} Define $x_1 \= x$.

    \smallskip
     \emph{Recursive step.} Assume that for some $t\in \N$, $\{x_k\}_{k=1}^t$ and $\{n_k\}_{k=1}^{t-1}$ are defined. We now divide our discussion into three cases, the third of which requires some delicate analysis.	
     
	\smallskip
	\emph{Case~A.}   Assume $x_t \in \cO$. Then define $n_t \= p $ and $x_{t+1} \= \Gau^{n_t}(x_t) = x_t$. Thus, by (\ref{e_Pf_t_gamma}) and (\ref{e_Pf_t_Density_psi}), we have 
	\begin{equation}\label{e_Pf_t_main_case_A}
		S_{n_t} \psi (x_t) = n_t \gamma  . 
	\end{equation}
	
	\smallskip
    \emph{Case~B.} Assume $x_t \notin U$. Then define $n_t \= 1$ and $x_{t+1} \= \Gau(x_t)$, so that combining (\ref{e_Pf_t_Density_psi}), (\ref{e_pf_main_theorem_1_tphi_nonpositive}), and (\ref{e_Pf_t_est_rho}) gives 
\begin{equation}\label{e_Pf_t_main_case_B}
	S_{n_t}\psi (x_t) = \psi(x_t)=\tphi(x_t)-\epsilon d(x_t,\cO)^\alpha\leq -\epsilon d(x_t,\cO)^\alpha< -\epsilon\rho^\alpha  = \gamma  . 
\end{equation}
	
	\smallskip
    \emph{Case~C.} Assume $x_t \in U \smallsetminus \cO$. Then $0<d(x_t,\cO)\leq \rho$. By (\ref{e_Pf_t_est_rho}), (\ref{e_Def_gap_r_theta}), and (\ref{e_Pf_t_main1_theta}), $\rho < \Delta_{r,\, \theta} (\cO) \leq \frac13 \Delta(\cO)$. So by (\ref{e_Def_gap}), there is a unique point $y \in \cO$ which is closest to $x_t$ among points in the periodic orbit $\cO$. By (\ref{e_Def_gap_r_theta}) and (\ref{e_Pf_t_est_rho}), $\abs{x_t-y}\leq \rho <\Delta_{r, \theta} (\cO)\leq r$.
	
	Let $N \in \N $ be the smallest positive integer satisfying 
	\begin{equation}  \label{e_Pf_t_Density_Lattes_N}
		\Absbig{ \Gau^{N}(x_t)- \Gau^{N}(y) }> \Delta_{ r, \theta} (\cO)   . 
	\end{equation}
	Such a positive integer always exists; otherwise, we have $\abs{\Gau^{n}(x_t)- \Gau^{n}(y) }\leq \Delta_{ r, \theta} (\cO)\leq r=\eta_m$ for all $n\in \N$ by (\ref{e_Def_gap_r_theta}) and (\ref{e_r}).  Then for each $n\in \N$, by Lemma~\ref{l_basic_property_bounded_continued_fractions}~(i), $\Gau^{n}(x_t)$ and $\Gau^{n}(y)$ are contained in the same interval $(1/(l+1),1/l)$ for some $l\in\N$. So $x_t$ and $y$ have the same continued fraction digits, and consequently $x_t=y\in \cO$,  contradicting the assumption that $x_t\notin \cO$.   

    From the definition of $N$, 
    together with (\ref{e_Def_gap_r_theta}) and (\ref{e_Pf_t_main1_theta}), we see that for each 
    $0\le j\le N-1$,
	\begin{equation}   \label{e_Pf_t_mian1_N_less} 
		d \bigl( \Gau^j (x_t) , \cO \bigr) 
		\leq  \Absbig { \Gau^j (x_t)- \Gau^j(y) }
        \leq    \Delta_{r, \, \theta} (\cO) 
        \leq   \Delta(\cO) /3  .   
	\end{equation}  
	By (\ref{e_Pf_t_mian1_N_less}), (\ref{e_Def_gap_r_theta}), and (\ref{e_r}), for each $0\le j\le N-1$,  
	\begin{equation}\label{e_Pf_t_mian1_in_F_m}
		\Absbig{ \Gau^j (x_t)- \Gau^j(y) }
        \leq    \Delta_{r, \, \theta} (\cO)
        \leq r
        =\eta_m   . 
	\end{equation}
So by (\ref{e_Pf_t_mian1_in_F_m}), the fact that $\cO\subseteq \bcf_m$, and (\ref{e_set_F_m}), we have that
$\Gau^j(x_t)\in F_m$ for each 
$0\le j\le N-1$.
Therefore, by (\ref{e_Pf_t_mian1_N_less}), (\ref{e_Pf_t_mian1_in_F_m}), and Lemma~\ref{l_basic_property_bounded_continued_fractions}~(iii), for each $0\le j\le N-1$,
\begin{equation}  \label{e_Pf_t_main1_y_orbit_distance_expansion}
	d \bigl( \Gau^j(x_t), \cO \bigr) 
    = \Absbig{ \Gau^j(x_t)- \Gau^j(y) }
    \geq \lambda_m^j \abs{x_t-y}  . 
\end{equation}

By (\ref{e_Pf_t_main1_C_1}), (\ref{e_Pf_t_mian1_N_less}), (\ref{e_Def_gap_r_theta}), and (\ref{e_Pf_t_main1_theta}),
	\begin{equation}\label{e_Pf_TN_gap}
		\Absbig{  \Gau^{N} (x_t) - \Gau^{N} (y) } 
		\leq L_2 \Absbig{ \Gau^{N-1} (x_t) - \Gau^{N-1}(y) } 
        \leq L_2\Delta_{r, \, \theta} (\cO) 
        \leq L_2\theta \Delta(\cO)
        \leq \Delta(\cO) /3 . 
	\end{equation}

	Set $n_t\=N+1$ and $x_{t+1}\=\Gau^{n_t}(x_t)$. We now aim to show that $S_{n_t} \psi (x_t) \leq n_t \gamma$.
	
	Let $n \in \N$ be the smallest positive integer satisfying 
	\begin{equation}  \label{e_Pf_t_Density_Lattes_n_property}
		\abs{\Gau^n(x_t)-\Gau^n(y)} > \rho   .  
	\end{equation}
	Such an integer $n$ exists and satisfies $1\leq n\leq N$ since $\abs{x_t-y} \leq \rho < \Delta_{r, \theta} (\cO)$ (see (\ref{e_Pf_t_est_rho})), and by the definition of $N$. Moreover, we have
	\begin{equation}  \label{e_Pf_t_main1_n}
	\Absbig{ \Gau^{n-1}(x_t)-\Gau^{n-1}(y)} \leq  \rho   . 
	\end{equation}
	
	We will separately estimate two parts of the sum 
	\begin{equation}  \label{e_Pf_t_Density_Lattes_sum12}
		S_{n_t} ( \gamma - \psi ) (x_t) = S_n ( \gamma - \psi ) (x_t) + S_{n_t-n} ( \gamma - \psi ) ( \Gau^n(x_t) )  \eqqcolon \operatorname{I}  +  \operatorname{II}.
	\end{equation}
	For each $j\in \N$ with $n\le j\le N$, by (\ref{e_Pf_t_main1_y_orbit_distance_expansion}), (\ref{e_Pf_t_Density_Lattes_n_property}), and the fact that $\lambda_m > 1$ (see Lemma~\ref{l_basic_property_bounded_continued_fractions}), we have
	\begin{equation*}
		d \bigl( \Gau^j(x_t), \cO \bigr) 
		= \Absbig{ \Gau^j(x_t) - \Gau^j(y) } 
		\geq  \lambda_m^{j-n} \abs{\Gau^n(x_t)- \Gau^n(y)}
		>\rho  . 
	\end{equation*}
	Thus $\Gau^j (x_t) \notin U$, and by (\ref{e_Pf_t_main_case_B}), $\gamma - \psi \bigl( \Gau^j(x_t) \bigr) >0$ for each $n\le j\le N$. Hence by (\ref{e_Pf_t_Density_psi}), (\ref{e_pf_main_theorem_1_tphi_nonpositive}),  (\ref{e_Pf_TN_gap}), (\ref{e_Pf_t_Density_Lattes_N}), we have
	\begin{align*} \label{e_Pf_t_Density_Lattes_bound_II}
		\operatorname{II}
		& \geq  \gamma - \psi \bigl( \Gau^{N} (x_t) \bigr)
		   =  \gamma - \tphi \bigl( \Gau^{N} (x_t) \bigr)  + \epsilon  d \bigl( \Gau^{N}(x_t), \cO \bigr)^\alpha   \\
		& \geq  \gamma + \epsilon \Absbig{ \Gau^{N}(x_t)- \Gau^{N} (y) }^\alpha  
		 \geq  \gamma + \epsilon   ( \Delta_{r, \theta} (\cO) )^\alpha. 
	\end{align*}
	To estimate $\operatorname{I}$, we write
	\begin{equation}   \label{e_Pf_t_Density_Lattes_sum34}
		\operatorname{I} = ( n \gamma - S_n \psi(y) ) + ( S_n \psi(y) - S_n \psi(x_t) ) \eqqcolon \operatorname{III} + \operatorname{IV}
	\end{equation}
	and will bound each of the parts $\operatorname{III}$
    and $\operatorname{IV}$ below.
	
	We write $n = pq + r$ for $q,\,r\in\N_0$ with $0 \leq r \leq p - 1$. Then by (\ref{e_Pf_t_Density_psi}), (\ref{e_pf_main_theorem_1_tphi_nonpositive}), and (\ref{e_Pf_t_est_p_gamma}), we have $S_n \psi (y) \leq S_n\tphi (y) = S_{pq} \tphi (y) + S_r \tphi (y) \leq pq\gamma$. Thus, considering $\gamma<0$ (see (\ref{e_Pf_t_est_p_gamma})), we obtain
	\begin{equation*}    \label{e_Pf_t_Density_Lattes_bound_III}
		\operatorname{III} \geq r \gamma \geq (p-1) \gamma  . 
	\end{equation*}
	
	Next, by (\ref{e_Pf_t_Density_psi}), (\ref{e_Pf_t_main1_y_orbit_distance_expansion}), (\ref{e_Pf_t_Holderseminorm_psi}), Lemma~\ref{l_basic_property_bounded_continued_fractions}~(iii), (\ref{e_Pf_t_main1_n}), and (\ref{e_Pf_t_est_rho}), we have
	\begin{align*}
		\abs{\operatorname{IV}}
		&  \leq \sum_{j=0}^{n-1}  \Absbig{  \psi \bigl( \Gau^j (x_t) \bigr) - \psi \bigl( \Gau^j (y) \bigr) }  \\
		&  \leq \sum_{j=0}^{n-1}  \bigl( \Absbig{  \tphi \bigl( \Gau^j (x_t) \bigr) - \tphi \bigl( \Gau^j (y) \bigr) } + \epsilon d \bigl( \Gau^j(x_t), \cO\bigr)^\alpha \bigr)  \\
		&  \leq \sum_{j=0}^{n-1}  ( L_1 + \epsilon  ) \Absbig{\Gau^j(x_t)- \Gau^j(y)}^\alpha \\
		&  \leq \sum_{j=0}^{n-1}  ( L_1 + \epsilon) \lambda_m^{-(n-1-j) \alpha} \Absbig{\Gau^{n-1}(x_t)-\Gau^{n-1}(y) }^\alpha  \\
		&  \leq \rho^\alpha ( L_1 +  \epsilon) / (1-\lambda_m^{-\alpha} )    \\
		&  \leq   \epsilon^{-1} \abs{\gamma}  L_3  . 
	\end{align*}
	
	Combining the above estimates for $\operatorname{II}$, $\operatorname{III}$, and $\operatorname{IV}$, we obtain from (\ref{e_Pf_t_Density_Lattes_sum12}), (\ref{e_Pf_t_Density_Lattes_sum34}), (\ref{e_Pf_t_est_p_gamma}), (\ref{e_Pf_t_Density_Lattes_tau}) the final estimate
	\begin{equation}\label{e_Pf_t_main_case_C}
			\begin{aligned}
			n_t \gamma -  S_{n_t}\psi (y) 
			&  =      \operatorname{II} + \operatorname{III} + \operatorname{IV}  \\
			& \geq   \gamma + \epsilon  ( \Delta_{r, \, \theta} (\cO) )^\alpha - (p-1) \abs{\gamma} -  \epsilon^{-1} \abs{\gamma} L_3 \\
			& \geq    \epsilon ( \Delta_{r, \, \theta} (\cO) )^\alpha -  \bigl( 1 + L_3 \epsilon^{-1} \bigr)  p\abs{\gamma}\\
			& \geq   \bigl(  \epsilon     -  \bigl( 1 + L_3 \epsilon^{-1} \bigr) L_1 \tau  \bigr)      ( \Delta_{r, \, \theta} (\cO) )^\alpha\\
			& \geq 0   . 
		\end{aligned}
	\end{equation}
	\smallskip
	
	Denote $N_k\=\sum_{i=1}^{k}n_i$. Combining (\ref{e_Pf_t_main_case_A}), (\ref{e_Pf_t_main_case_B}), and (\ref{e_Pf_t_main_case_C}) now gives
	\begin{equation*}
			\liminf_{n\to +\infty} \frac{1}{n}S_n \psi(x) \le \liminf_{k \to +\infty} \frac{1}{N_k} \sum_{i=1}^{k} S_{n_i} \psi(x_i) \le \liminf_{k\to +\infty}\frac{1}{ N_k} \sum_{i=1}^k n_i \gamma = \gamma.
	\end{equation*}
	Therefore, since $x\in I \smallsetminus\Q$ was arbitrary, Proposition~\ref{p_mea_equal_sup_time_average}~(ii) and (\ref{e_Pf_liminf_psi}) give $\mea(\Gau,\psi)\leq \gamma$, thereby completing the proof of the 
	claim.
\end{proof}

\appendix

\section{Periodic locking property} \label{sct_locking_property}

In this appendix we prove the periodic locking property for the Gauss map. The proof of Lemma~\ref{l_periodic the locking property} and Proposition~\ref{p_periodic_locking} is similar to \cite{BZ15}.

\begin{lemma}\label{l_periodic the locking property}
	Let $\Gau\: I\to I$ be the Gauss map. Let $\mu \in \MMM(I,\Gau)$ be a measure supported on a periodic orbit $\cO$. Then there exists a constant $C_\mu \geq 1$ such that for all $\nu\in\MMM(I,\Gau)$ and $\phi \in \Holder{\alpha}(I)$,  
	\begin{equation}\label{e_inequality of Wasserstein distance}
		\langle \nu, \phi\rangle
        \leq	\langle \mu, \phi\rangle+ C_\mu\Hseminorm{\alpha}{\phi}	\langle  \nu, d(\cdot,\mathcal{O})^\alpha \rangle . 
	\end{equation}
	
\end{lemma}

\begin{proof}
	We will write $p\=\card \cO$ throughout this proof. If $p=1$, i.e., if $\cO$ consists of a single point $x_0$, then for every $\nu\in \mathcal{M}(I,\Gau)$ and $\phi\in \Holder{\alpha}(I)$,
	\begin{equation*}
		\langle \nu, \phi\rangle 
        \leq	\phi(x_0)+  \Hseminorm{\alpha}{\phi} \langle \nu,  d(\cdot,x_0)^\alpha \rangle . 
	\end{equation*}
	So (\ref{e_inequality of Wasserstein distance}) holds with $C_\mu=1$.
	
	From now on assume that $p\geq 2$. Clearly, $\cO\cap\Q=\emptyset$. Denote $\delta \= \Delta(\cO)$ as defined in (\ref{e_Def_gap}). 
    Fix an arbitrary $y\in \cO$. Since $\Gau$ is continuous at each irrational number, there exists $\epsilon_y>0$ such that $\Absbig{ \Gau^i(x)-\Gau^i(y)} < \delta/2$ for all $x\in(y-\epsilon_y,y+\epsilon_y)$ and $0\leq i\leq p-1$. Moreover, since $\cO$ is a finite set, there exists $\epsilon\=\min \{ \epsilon_y : y\in \cO\}$ such that if $x\in I$, $y\in \cO$, and $d(x,y)<\epsilon$, then $\Absbig{ \Gau^i(x)-\Gau^i(y)}< \delta/2$ for all $0\leq i\leq p-1$.  
	
	Define $C_\mu \=1/\epsilon$. We aim to check that the inequality (\ref{e_inequality of Wasserstein distance}) is satisfied for every $\nu\in \MMM(I,\Gau)$. It is sufficient to consider ergodic $\nu$; the general case will follow using ergodic decomposition.
	
	Fix $x\in I$ and $\phi\in \Holder{\alpha}(I)$ such that the Birkhoff averages of the continuous function $\phi$ along the orbit of $x$ converge to $\langle \mu, \phi \rangle$. We will inductively define a transport sequence $\{y_i\}_{i=-1}^{+\infty}$ in $\cO$. As an auxiliary device for the definition of the sequence, each integer $i\geq-1$ will be labelled as good or bad. The definition is as follows: the point $y_{-1}\in \cO$ is chosen arbitrarily, and the time $-1$ is labelled bad. As inductive hypothesis let us suppose that $y_{-1},\,y_0,\,\dots,\,y_i$ are already defined and that the times $-1,\,\dots,\,i$ are already labelled. Then
	\begin{enumerate}
        \smallskip
	    \item[(a)] If $d\bigl(\Gau^{i+1}(x),\cO\bigr)<\epsilon$ then each time $j\in\{i+1,\,i+1,\,\dots,\,i+p\}$ is labelled good, and $y_j$ is defined as the unique point in $\cO$ that is closest to $\Gau^j(x)$. Note that $y_j=\Gau^{j-i}(y_i)$, and in particular each point of $\cO$ appears exactly once in the sequence $y_{i+1},\,y_{i+2},\,\dots,\,y_{i+p}$;

        \smallskip
	    \item[(b)] If $d\bigl(\Gau^{i+1}(x),\cO\bigr)\geq\epsilon$ then the time $i+1$ is labelled bad, and  we define $y_{i+1}$ as $\Gau(y_k)$, where $k$ is the largest bad time less than or equal to $i$.      
	\end{enumerate}
	
	This completes the definition of the transport sequence. Notice that $\lim\limits_{n\to +\infty}\frac{1}{n}\sum_{i=0}^{n-1}\phi(y_i)=\langle\mu , \phi\rangle$. On the other hand, for all $i \in \N_0$, the distance $d\bigl(\Gau^i(x),\cO\bigr)$ is equal to $\Absbig{ \Gau^i(x)-y_i}$ if $i$ is a good time, and otherwise is at least $\epsilon$. In either case we have
	\begin{equation*}
		\Absbig{ \Gau^i(x)-y_i}
        \leq C_\mu d\bigl(\Gau^i(x),\cO\bigr)  . 
	\end{equation*} 
	Using these properties we obtain, for every $\phi\in \Holder{\alpha}(I)$,
	\begin{equation*}
		\begin{aligned}
			\langle \nu,  \phi \rangle
            &=\lim\limits_{n\to +\infty}\frac{1}{n}\sum\limits_{i=0}^{n-1}\phi\bigl(\Gau^i(x)\bigr)
			\leq\lim\limits_{n\to +\infty}\frac{1}{n}\sum\limits_{i=0}^{n-1} \bigl(\phi(y_i)+\Hseminorm{\alpha}{\phi}\Absbig{ \Gau^i(x)-y_i}^\alpha \bigr)\\
			&\leq\lim\limits_{n\to +\infty}\frac{1}{n}\sum\limits_{i=0}^{n-1} \bigl( \phi(y_i)+C_\mu\Hseminorm{\alpha}{\phi}d\bigl(\Gau^i(x),\cO\bigr)^\alpha \bigr)
			=	\langle \mu,  \phi\rangle +C_\mu\Hseminorm{\alpha}{\phi} \langle \nu, d(\cdot,\cO)^\alpha \rangle. 
		\end{aligned}
	\end{equation*}
	
	The inequality (\ref{e_inequality of Wasserstein distance}) follows.
\end{proof}

\begin{prop}[Locking property for periodic orbits]\label{p_periodic_locking}
	The set $\Lock^\alpha(\Gau)$ is an open dense subset of $\sP^\alpha(\Gau)$.
\end{prop}

\begin{proof}
	By definition, the set $\Lock^\alpha(I)$ is open and contained in $\sP^\alpha(\Gau)$, so we only need to prove that it is dense.
	
	Let $\phi\in \sP(I)$, and suppose $\mu\in \mathcal{M}_{\max}(\Gau,\phi)$ is supported on a periodic orbit $\cO$.
    For each $t>0$, consider $\phi_t\=\phi-td(\cdot,\cO)^\alpha$. The functions $\phi_t$ belong to the Banach space $\Holder{\alpha}(I)$, and converge to $\phi$ as $t\to 0$.
    Moreover, for any $\psi\in \Holder{\alpha}(I)$ and $\nu\in\mathcal{M}(I,\Gau)$, by Lemma~\ref{l_periodic the locking property} we obtain
	\begin{equation*}
		\begin{aligned}
		\langle \nu, \phi_t+\psi \rangle
        \leq\langle \nu, \phi\rangle + \langle \nu,  \psi\rangle -t\langle \nu, d(\cdot,\cO)^\alpha\rangle
		&\leq \langle \mu,  \phi\rangle+ \langle \mu, \psi \rangle + (C_\mu\Hseminorm{\alpha}{\psi}-t) \langle \nu, d(\cdot,\cO)^\alpha\rangle\\
		&=\langle \mu, \phi_t+\psi \rangle + (C_\mu\Hseminorm{\alpha}{\psi}-t)\langle \nu,  d(\cdot,\cO)^\alpha\rangle. 
	\end{aligned}
\end{equation*}
	Therefore, if $\Hseminorm{\alpha}{\psi}<t/C_\mu$ then $\mu$ is the unique maximizing measure for $\phi_t+\psi$. This shows that $\phi_t\in\Lock^\alpha(\Gau)$ for each $t>0$. So $\phi$ belongs to the closure of $\Lock^\alpha(\Gau)$. Hence we conclude that $\Lock^\alpha(\Gau)$ is dense in $\sP^\alpha(\Gau)\cap\Holder{\alpha}(I)$.
\end{proof}

\end{document}